\documentclass[preprint]{elsarticle}

\usepackage{lineno,hyperref}
\modulolinenumbers[5]

\journal{}

\usepackage{lipsum}
\usepackage{amsmath,amsfonts,amsopn}
\usepackage{graphicx}
\usepackage{epstopdf}
\usepackage{algorithm,algorithmic}
\usepackage{setspace}
\usepackage{rotating}
\usepackage{comment}
\usepackage{hyperref}
\usepackage{xcolor}
\usepackage{fancyhdr}

\DeclareMathOperator{\diag}{diag}
\DeclareMathOperator*{\minimize}{minimize}
\def\st{\mbox{s.t.}}
\newcommand{\N}{\mathbb{N}}
\newcommand{\R}{\mathbb{R}}
\def\bb{\mathbf b}

\def\bd{\mathbf d}
\def\be{\mathbf e}
\def\bu{\mathbf u}
\def\bv{\mathbf v}

\def\bx{\mathbf x}
\def\by{\mathbf y}
\def\bz{\mathbf z}
\def\b0{\mathbf 0}

\def\mD{\mathcal{D}}
\def\mN{\mathcal{N}}
\def\mS{\mathcal{S}}
\newtheorem{theorem}{Theorem}
\newtheorem{lemma}[theorem]{Lemma}
\newenvironment{proof}{\vspace*{-5mm}\paragraph{\textnormal{\emph{Proof}}}}{\hfill$\square$ \par\medskip\smallskip}

\ifpdf
  \DeclareGraphicsExtensions{.eps,.pdf,.png,.jpg}
\else
  \DeclareGraphicsExtensions{.eps}
\fi
\graphicspath{./figures/}

\usepackage{enumitem}
\setlist[enumerate]{leftmargin=.5in}
\setlist[itemize]{leftmargin=.5in}

\date{July 15, 2019}

\begin{document}
	
\begin{frontmatter}

\title{ACQUIRE: an inexact iteratively reweighted norm approach for TV-based Poisson image restoration\tnoteref{t1}}

\tnotetext[t1]{This work
	was partially supported by Gruppo Nazionale per il Calcolo Scientifico - Istituto Nazionale di Alta Matematica (GNCS-INdAM).}

\author[1]{Daniela di Serafino\corref{cor1}}\ead{daniela.diserafino@unicampania.it}
\author[2]{Germana Landi}\ead{germana.landi@unibo.it}
\author[3]{Marco Viola}\ead{marco.viola@uniroma1.it}

\address[1]{Dipartimento di Matematica e Fisica, Universit\`a degli Studi della Campania \\ ``Luigi Vanvitelli'', viale A. Lincoln~5, 81100 Caserta, Italy}
\address[2]{Dipartimento di Matematica, Universit\`a degli Studi di Bologna, Piazza di Porta S.~Donato~5, 40126 Bologna, Italy}
\address[3]{Dipartimento di Ingegneria Informatica, Automatica e Gestionale ``Antonio Ruberti'', Sapienza Universit\`a di Roma, via Ariosto~25, 00185 Roma, Italy}

\begin{abstract}
We propose a method, called ACQUIRE, for the solution of constrained optimization problems modeling the restoration of images
corrupted by Poisson noise. The objective function is the sum of a generalized Kullback-Leibler divergence term and a TV regularizer,
subject to nonnegativity and possibly other constraints, such as flux conservation. ACQUIRE is a line-search method that considers
a smoothed version of TV, based on a Huber-like function, and computes the search directions by minimizing
quadratic approximations of the problem, built by exploiting some second-order information. A~classical second-order Taylor
approximation is used for the Kullback-Leibler term and an iteratively reweighted norm approach for the smoothed TV term.
We prove that the sequence generated by the method has a subsequence converging to a minimizer of the smoothed problem
and any limit point is a minimizer. Furthermore, if the problem is strictly convex, the whole sequence is convergent. We note
that convergence is achieved without requiring the exact minimization of the quadratic subproblems; low accuracy in this
minimization can be used in practice, as shown by numerical results. Experiments on reference test problems show that our method
is competitive with well-established methods for TV-based Poisson image restoration, in terms of both computational
efficiency and image quality.
\end{abstract}
\end{frontmatter}

\section{Introduction\label{sec:intro}}

Restoring images corrupted by Poisson noise is required in many applications, such as fluorescence microscopy~\cite{SarderNehorai2006},
X-ray computed tomography (CT) \cite{Herman2009}, positron emission tomography (PET)~\cite{SheppVardi1982},
confocal microscopy~\cite{Pawley1996} and astronomical imaging~\cite{StarckMurtagh2006,BerteroBoccacciDesideraVicidomini2009}.
Thus, this is a very active research area in image processing. We consider a discrete formulation of the problem, where the object
to be restored is represented by a vector
$\bx \in \R^n$ and the measured data are assumed to be a vector $\by \in \N_0^m$, whose entries $y_j$
are samples from $m$ independent Poisson random variables $Y_j$ with probability
$$
    P(Y_j = y_j) = \frac{e^{-(A \bx + \bb)_j}(A\bx + \bb)_ {j}^{y_j}}{y_j!},
$$
where the matrix $A = (a_{ij}) \in \R^{m \times n}$ models the observation mechanism of the imaging system
and $\bb \in \R^{m}$,  $\bb > 0$, models the background radiation detected by the sensors.
Standard assumptions on $A$ are
\begin{equation}
\label{eq:matrixa}
     a_{ij} \geq 0 \mbox{ for all }  i, j, \qquad \sum_{i=1}^m a_{ij} = 1 \mbox{ for all } j.
\end{equation}
By applying a maximum-likelihood approach \cite{BerteroBoccacciDesideraVicidomini2009,SheppVardi1982},
we can estimate $\bx$ by minimizing the Kullback-Leibler (KL) divergence of $A \bx+\bb$ from $\by$:
\begin{equation} \label{eq:kl}
    D_{KL}(A \bx + \bb, \by) = \sum_{j=1}^m \left( y_j \ln \frac{y_j}{(A \bx+\bb)_j} + (A \bx+\bb)_j - y_j \right),
\end{equation}
where we set $y_j \ln (y_j / (A \bx+\bb)_j) = 0 $ if $y_j = 0$.
A regularization term is usually added to~\eqref{eq:kl} to deal with the inherent ill-conditioning of the estimation problem.
We focus on edge-preserving regularization by Total Variation (TV)~\cite{RudinOsherFatemi1992},
which has received considerable attention because of
its ability of preserving edges and smoothing flat areas of the images. We note that, although TV regularization
is known to suffer from undesirable staircase artifacts, it is still widely used in many medical and biological applications (see, e.g., \cite{BarnardBilheuxToops2018,MotaMatelaOliveraAlmeida2016,ZhangHuNagy2018}, \url{http://ranger.uta.edu/~huang/R_CSMRI.htm}).
Furthermore, by focusing on TV-regularized problems, we introduce and test a novel solution method that allows for extensions
to other models, such as high-order TV~\cite{LiuHuangLvWang2017,PapafitsorosSchonlieb2014} and Total Generalized
Variation~\cite{BrediesHoller2014,BrediesKunischPock2010}, proposed to reduce the staircase effect.

Assuming, for simplicity, that $\bx$ is obtained by stacking the columns of a 2D image $X = (X_{k,l}) \in \R^{r \times s}$, i.e.,
$x_i = X_{k,l}$ with $i = (l-1)r + k$ and $n=rs$, the following discrete version of the TV functional can be defined~\cite{Chambolle2004}:
$$
    TV(X) = \sum_{k=1}^{r}\sum_{l=1}^{s}\sqrt{(X_{k+1,l}-X_{k,l})^2 + (X_{k,l+1}-X_{k,l})^2},
$$
where $X$ is supposed to satisfy some boundary conditions, e.g., periodic. This can be also
written as
\begin{equation}
\label{eq:tv}
    TV(\bx) = \sum_{i=1}^{n} \| D_i \bx \|,
\end{equation}
where
$$
 D_i =
 \left( \begin{array}{c}
   \be_{(l-1)r + k + 1}^T - \be_{(l-1)r + k}^T \\
   \be_{lr+k}^T - \be_{(l-1)r + k}^T
 \end{array} \right),
 \quad
i = (l-1)r + k,
$$
$\be_q \in \R^{n}$ is the $q$th standard basis vector, and $\| \cdot \|$ is the 2-norm
%
%
%

Thus, we are interested in solving the following problem:
\begin{equation}
\label{eq:poispb}
\begin{array}{ll}
\minimize & \displaystyle D_{KL}(\bx)  + \lambda \, TV(\bx), \\
\st            & \bx \in \mS,
\end{array}
\end{equation}
where $D_{KL}(\bx)$ is a shorthand for $D_{KL} (A \bx + \bb, \by)$, $\lambda > 0$
is a regularization parameter, and $\bx \in \mS$ corresponds to some physical constraints.
The nonnegativity of the image intensity naturally leads to the constraint $\bx \ge \b0$.
When the matrix $A$ comes from the discretization of a convolution operator and it is normalized as
in~\eqref{eq:matrixa}, the constraint $\sum_{i=1}^{n} x_i = \sum_{j=1}^{m} (y_i-b_i)$
can be added, since the convolution performs a modification of the intensity distribution, while the total intensity remains
constants~\cite{BerteroLanteriZanni2008}.\footnote{We have implicitly assumed that $\by$ has been converted into
a real vector with entries ranging in the same interval as the entries of $\bx$.}
In other words, common choices of $\mS$ are
\begin{equation}
\label{eq:constr1}
   \mS = \mS_1 := \{ \bx \in \R^n : \bx \ge \b0 \}
\end{equation}
or
\begin{equation}
\label{eq:constr2}
   \mS = \mS_2 :=  \{ \bx \in \R^n : \bx \ge \b0, \, \be^T \bx =  \overline\be^T (\by - \bb) \},
\end{equation}
where $\be$ and $\overline\be$ denote the vectors of all 1's of sizes $n$ and $m$, respectively.

Various approaches have been proposed to solve problem~\eqref{eq:poispb}, mostly with $\mS = \mS_1$;
a key issue in all cases is to deal with the nondifferentiability of the TV functional.
Some representative methods are listed next.
A classical approach consists in approximating the TV functional with a smooth version of it
and using well-established techniques such as expectation-maximization methods
\cite{JohnssonHuangChan1998,PaninZengGullberg1999},
gradient-projection methods with suitable scaling techniques aimed at accelerating convergence
\cite{BonettiniZanellaZanni2009,PiccolominiColiMorottiZanni2017,ZanellaBoccacciZanniBertero2009},
and alternating linearized minimization methods~\cite{LiMalgouyresZeng2017}.

The approximation of TV can be avoided, e.g., by using forward-backward splitting techniques;
this is the case of the proximal-gradient methods proposed in \cite{BonettiniLorisPortaPrato2016,HarmanyMarciaWillet2012}
and the forward-backward EM method discussed in~\cite{SawatzkyBruneWuebbelingKoestersSchaefersBurger2008}.
On the other hand, the previous methods require, at each step, the solution of a Rudin-Osher-Fatemi (ROF) denoising
subproblem \cite{RudinOsherFatemi1992}, which can be computed only approximately, using, e.g., the algorithms
proposed in \cite{BeckTeboulle2009,Chambolle2004}.
Methods based on ADMM and SPLIT BREGMAN techniques, such as those presented in
\cite{FigueiredoBioucasDias2010,Getreuer2012,SetzerSteidlTeuber2010}, do not exploit smooth TV
approximations too. They generally use more memory because of auxiliary variables of the same size
as $\bx$ or $\by$, and require the solution of linear systems involving $A^TA$ and, possibly,
the solution of ROF subproblems. Finally, a different approach to avoid the difficulties associated with
the nondifferentiability of the TV functional is based on the idea of reformulating~\eqref{eq:poispb} as
a saddle-point problem and solving it by a primal-dual algorithm. In this context, an alternating extragradient
scheme has been presented in \cite{BonettiniRuggiero2011}, and a procedure exploiting the
Chambolle-Pock algorithm \cite{ChambollePock2011} has been described in \cite{WenChanZeng2016}.
%

In this paper we take a different approach, aimed at exploiting some second-order information not considered
by the aforementioned methods. We consider a smoothed version of TV, based on a Huber-like function,
and propose a line-search method, called ACQUIRE (Algorithm based on Consecutive QUadratic and Iteratively
REweighted norm approximations), which minimizes a sequence of quadratic models obtained by
a second-order Taylor approximation of the KL divergence and an iteratively reweighted norm (IRN) approximation of the smoothed TV.
We prove the convergence of ACQUIRE with inexact solution of the inner quadratic problems. We show by numerical
experiments that exploiting some second-order information can lead to fast image restorations even with
low accuracy requirements on the solution of the inner problems, without affecting the quality of the reconstructed images.
In particular, ACQUIRE generally produces a strong reduction of the reconstruction error in the first iterations,
thus achieving a good tradeoff between accuracy and efficiency, and resulting competitive with
state-of-the-art methods.

The remainder of this paper is organized as follows. In Section~\ref{sec:preliminaries} we recall some preliminary
concepts that will be exploited later. In Section~\ref{sec:algorithm} we describe our method and in Section~\ref{sec:convergence}
we prove that it is well posed and convergent. We provide implementation details and discuss the results obtained
by applying the proposed method to several test problems in Section~\ref{sec:experiments}. Some conclusions are reported in
Section~\ref{sec:conclusions}.

\section{Preliminaries\label{sec:preliminaries}}

We first provide some useful details about the KL divergence and introduce a smooth version of the TV functional.
Then we recall the concept of projected gradient and its basic properties, exploited later in this work.

Assumptions~\eqref{eq:matrixa} and $\bb > 0$ ensure that, for any given $\by \ge 0$, $D_{KL}$ is a nonnegative,
convex, coercive, twice continuosly differentiable function in $\R_+^n$ (see, e.g.,
\cite{BerteroBoccacciTalentiZanellaZanni2010,FigueiredoBioucasDias2010}). Its gradient and Hessian
are given by
$$
    \nabla D_{KL}(\bx)= A^T \left( \be - \frac{\by}{A \bx + \bb} \right)
$$
and
\begin{equation}
     \label{eq:hesskl}
    \nabla^2 D_{KL}(\bx) = A^T U(\bx)^2 A, \quad U(\bx) = \diag \left( \frac{\sqrt{\by}}{A \bx + \bb} \right),
\end{equation}
where the square root and the ratios are intended componentwise, and $\diag (\bv)$ denotes the diagonal matrix with
diagonal entries equal to the entries of $\bv$. It can be proved that $\nabla D_{KL}$ is Lipschitz
continuous \cite{HarmanyMarciaWillet2012}; furthermore, it follows from~\eqref{eq:hesskl}
that $\nabla^2 D_{KL}$ is positive definite, i.e., $D_{KL}$ is strictly convex, whenever $\by > \b0$ and $A$ has nullspace
$\mN (A) = \{ \b0 \}$. In this case, if $\bx$ is constrained to be in a bounded subset of the nonnegative orthant, e.g.,
the set $\mS_2$ in~\eqref{eq:constr2}, the minimum eigenvalue of $\nabla^2 D_{KL}(\bx)$  is bounded below
independently of $\bx$, and $D_{KL}$ is strongly convex.

From a practical point of view, it is interesting to note that $A$ is usually the representation of a convolution operator, and hence
the computation of $\nabla D_{KL}$ or of matrix-vector products involving $ \nabla^2 D_{KL}$
can be performed efficiently via fast algorithms for discrete Fourier, cosine or sine transforms.

The  $TV$ functional is nonnegative, convex and continuous. Thus problem~\eqref{eq:poispb} admits
a solution, which is unique if  $\by > 0$ and $\mN (A) = \{ \b0 \}$. Since $TV$ is not differentiable,
we use a regularized version of it, $TV_\mu$. Taking into account the discussion
in~\cite{WeissBlancFeraudAubert2009} about smoothed versions of TV, we consider
$$
TV_\mu(\bx) = \sum_{i=1}^{n} \phi_\mu \left(\|D_i \bx \| \right),
$$
where $\phi_\mu$ is the Huber-like function
$$
	\phi_\mu(z) = \left\lbrace \begin{array}{ll}
	|z| & \mbox{if } |z| > \mu,\\
	\frac{1}{2}(\frac{z^2}{\mu}+\mu) & \mbox{otherwise},
	\end{array}  \right.
$$
and $\mu > 0$ is small. It is easy to verify that $TV_\mu$ is Lipschitz continuously differentiable and its gradient reads as follows:
$$
    \nabla TV_\mu (\bx) = \sum_{i=1}^{n} \nabla \phi_\mu \left(\|D_i \bx \| \right), \quad
    \nabla \phi_\mu \left( \|D_i \bx \| \right) = \left\{
    \begin{array}{ll}
    \displaystyle \frac{D_i^T D_i \bx}{\| D_i \bx \|} & \mbox{if } \|D_i \bx \| > \mu, \\[10pt]
    \displaystyle \frac{D_i^T D_i \bx}{\mu} & \mbox{otherwise}.
    \end{array} \right.
$$
We also observe that $TV_\mu$ is not twice continuously differentiable, but has continuous Hessian for all
$\bx$ such that $\|D_i \bx \| \ne \mu$:
$$
    \nabla^2 TV_\mu (\bx) = \sum_{i=1}^{n} \nabla^2 \phi_\mu \left(\|D_i \bx \| \right),
$$
\begin{equation}
\label{eq:hestvmu}
    \nabla^2 \phi_\mu \left( \|D_i \bx \| \right) =  \left\{
    \begin{array}{ll}
    \displaystyle \frac{D_i^T D_i}{\| D_i \bx \|} - \frac{(D_i^T D_i \bx) (D_i^T D_i \bx)^T}{\|D_i \bx \|^3} & \mbox{if } \|D_i \bx \| > \mu , \\[10pt]
    \displaystyle \frac{D_i^T D_i}{\mu} & \mbox{if } \|D_i \bx \| < \mu .
    \end{array} \right .
\end{equation}
\par

Now we recall basic notions about the projected gradient.
Let $\mS$ be a nonempty, closed and convex set. For any continuously differentiable
function $f: \mD \subseteq \R^n \to \R$, with $\mD$ open set containing $\mS$, the projected gradient
of $f$ at $\bx \in \mS$ is defined as the orthogonal projection of $-\nabla f$ onto the tangent cone
to $\mS$ at $\bx$, denoted by $T_\mS(\bx)$:
$$
   \nabla_\mS f (\bx) = \arg\min \left\{  \Vert \bv + \nabla f (\bx) \Vert
   \; \mbox{ s.t.} \; \bv \in T_\mS(\bx) \right\},
$$
When $\mS$ is the set $\mS_1$ defined in~\eqref{eq:constr1}, the tangent cone takes the form
$$
    T_\mS(\bx) = \left\{ \bv \in \R^n \, : \; v_i \geq 0 \mbox{ if } x_i  = 0 \right\}
$$
and the computation of $\nabla_\mS f (\bx)$ is straightforward; when $\mS$ is the set
$\mS_2$ in~\eqref{eq:constr2},
$$
     T_\mS (\bx) = \left\{ \bv \in \R^n \, : \; \be^T\bv=0 \mbox{ and } v_i \geq 0 \mbox{ if } x_i  = 0 \right\},
$$
and $\nabla_\mS f (\bx)$ can be efficiently determined too, thanks to
the availability of low-cost algorithms for computing the projection in this case (see, e.g.,
\cite{CalamaiMore1987a,Condat2016,DaiFletcher2006}).

Since the projection onto $\mS$ is nonexpansive, for all $\bx, \overline\bx \in \mS$ it is
$$
   \|  \nabla_\mS f (\bx) - \nabla_\mS f (\bar \bx) \| \le \| \bx - \overline\bx \|;
$$
furthermore,
\begin{equation}
\label{eq:grad_decomp}
    - \nabla f(\bx) = \nabla_\mS f(\bx) + P_{N_\mS(\bx)} (- \nabla f(\bx)),
\end{equation}
where $P_{N_\mS(\bx)}$ denotes the orthogonal projection operator onto the normal cone to $\mS$ at $\bx$,
$$
    N_\mS (\bx) = \left\{ \bv \in \R^n \, : \; \bv^T \bu \le 0 \mbox{ for all } \bu \in \mS(\bx) \right\},
$$
which is the polar cone of $T_\mS (\bx)$ (see, e.g., \cite[Lemma 2.2]{Zarantonello1971}).

Finally, it is well known that any constrained stationary point $\bx^*$ of~$f$ in $\mS$ is characterized by
$\nabla_\mS f (\bx^*) = \b0$
and that $\|  \nabla_\mS f\|$ is lower semicontinuous on $\mS$ (see, e.g., \cite{CalamaiMore1987}).

\section{IRN-based inexact minimization method\label{sec:algorithm}}

We propose an iterative method for solving the problem
\begin{equation}
\label{eq:poispb2}
\begin{array}{ll}
\minimize & \displaystyle D_{KL}(\bx)  + \lambda \, TV_\mu (\bx), \\
\st            & \bx \in \mS,
\end{array}
\end{equation}
where $\mS$ can be any nonempty, closed and convex subset of $R_+^n$, although our practical interest
is for the feasible sets in~\eqref{eq:constr1}-\eqref{eq:constr2}.
This method is based on two main steps: the inexact solution of a quadratic model of~\eqref{eq:poispb}
and a line-search procedure.

Given an iterate $\bx^{(k)} \in \mS$, we consider the following quadratic approximation of $D_{KL}$:
\begin{equation}
\label{eq:quadkl}
\begin{array}{ll}
    \displaystyle
    D_{KL} (\bx)  \approx D_{KL}^{(k)}(\bx) & \displaystyle \!\!\!\! = D_{KL}(\bx^{(k)}) + (\bx-\bx^{(k)})^T \nabla D_{KL}(\bx^{(k)}) \\[5pt]
                                                                   & \displaystyle \!\!\!\! +  \; \frac{1}{2} (\bx-\bx^{(k)})^T ( \nabla^2 D_{KL} (\bx^{(k)}) + \gamma I ) (\bx-\bx^{(k)}),
\end{array}
\end{equation}
where $I$ is the identity matrix and $\gamma > 0$. Note that $\gamma I$ has been introduced to ensure that $D_{KL}^{(k)}$ is strongly convex;
obviously, we can set $\gamma = 0$ if $\by > \b0$, $\mN (A) = \{ \b0 \}$ and $\mS$ is bounded.

In order to build a quadratic model of $TV_\mu$, we use the IRN approach described in \cite{RodriguezWohlberg2009},
i.e., we approximate $TV_\mu(\bx)$ as follows:
$$
   TV_\mu(\bx) \approx TV_\mu^{(k)}(\bx) = \frac{1}{2} \sum_{i=1}^{n} w_i^{(k)} \| D_i \bx \|^2 + \frac{1}{2} TV_\mu(\bx^{(k)}),
$$
where
$$
   w_i^{(k)} = \left\{ \begin{array}{ll}
   \|D_i \bx^{(k)} \|^{-1}     & \mbox{if }  \|D_i \bx^{(k)} \| > \mu, \\
   \mu^{-1} & \mbox{otherwise}.
   \end{array} \right.
$$
Trivially,
$$
   TV_\mu^{(k)} (\bx^{(k)}) = TV_\mu (\bx^{(k)}), \quad \nabla TV_\mu^{(k)} (\bx^{(k)})  = \nabla TV_\mu (\bx^{(k)}) ;
$$
furthermore,
$$
     \nabla^2 TV_\mu^{(k)} (\bx^{(k)}) = \sum_{i=1}^{n} w_i^{(k)} D_i^T D_i ,
$$
and hence, for any $\bx$ such that $\| D_i \bx^{(k)} \| \ne \mu$, the Hessian $\nabla^2 TV_\mu^{(k)}(\bx^{(k)})$ can
be regarded as an approximation of $\nabla^2 TV_\mu (\bx^{(k)})$, obtained by neglecting the higher order term
in the right-hand side of~\eqref{eq:hestvmu}, which generally increases the ill-conditioning of the Hessian matrix.
Thus, we can say that $TV_\mu^{(k)}$ contains some second-order information about $TV_\mu$.
It is worth noting that the higher order term of the Hessian of a smoothed TV function is also neglected
in the lagged diffusivity method by Vogel and Oman \cite{VogelOman1996}.

In the following, to simplify the notation we set
\begin{eqnarray*}
   F(\bx)         & = & D_{KL} (\bx) + \lambda \, TV_\mu(\bx), \\
   F_k(\bx) & = & D_{KL}^{(k)} (\bx) + \lambda \, TV_\mu^{(k)} (\bx).
\end{eqnarray*}
At iteration $k$, our method computes a feasible approximation $\widehat{\bx}^{(k)}$ to the solution $\overline{\bx}^{(k)}$
of the quadratic problem
\begin{equation}
\label{eq:quadpb}
\begin{array}{ll}
\minimize & \displaystyle F_k (\bx), \\
\st            & \bx \in \mS,
\end{array}
\end{equation}
and performs a line search along the direction
$$
     \bd^{(k)} = \widehat{\bx}^{(k)} - \bx^{(k)},
$$
until an Armijo condition is satisfied, to obtain an approximation $\bx^{(k+1)}$ to the solution of problem~\eqref{eq:poispb2}.
This procedure is sketched in Algorithm~\ref{alg:acquire} and is called ACQUIRE, which comes from ``Algorithm based
on Consecutive QUadratic and Iteratively REweighted norm approximations''.
%
\begin{algorithm}
\caption{ -- ACQUIRE \label{alg:acquire}}
\vskip 1mm
\begin{spacing}{1.15}
\begin{algorithmic}[1]
\STATE choose $\bx_0 \in \mS$, $\eta \in (0,1)$, $\delta \in (0,1)$, $\{ \varepsilon_k \}$ such that
           $\varepsilon_k > 0$ and $ \lim_{k \to \infty} \varepsilon_k = 0 $
\FOR{$k = 1, 2, \ldots$}
\STATE compute an approximate solution $\widehat{\bx}^{(k)} \in \mS$ to the quadratic problem~\eqref{eq:quadpb},
          such that
          \par \vspace*{-5mm} 
          \begin{equation}
          \label{eq:err}
             \| \widehat{\bx}^{(k)} - \overline{\bx}^{(k)} \|  \le  \varepsilon_k \; \mbox{ and } \;
             F_k (\widehat{\bx}^{(k)}) \le F_k ( \bx^{(k)} )  
          \end{equation}
          \par \vspace*{-4mm}
          \label{alg:approxmin}
\STATE $\alpha_k := 1$
\STATE $\bd^{(k)} := \widehat{\bx}^{(k)} - \bx^{(k)}$
\STATE $\bx^{(k)}_\alpha := \bx^{(k)} + \alpha_k  \bd^{(k)}$
\WHILE{ $F( \bx^{(k)}_\alpha ) >  F(\bx^{(k)})  + \eta \alpha_k \nabla F (\bx^{(k)})^T \bd^{(k)}$ } \label{alg:line_search}
\STATE $\alpha_k := \delta \alpha_k$
\STATE $\bx^{(k)}_\alpha := \bx^{(k)} + \alpha_k  \bd^{(k)}$
\ENDWHILE
\STATE $\bx^{(k+1)} = \bx^{(k)}_\alpha$
\ENDFOR
\end{algorithmic}
\end{spacing}
\end{algorithm}

ACQUIRE is well posed (i.e., a steplength $\alpha_k$ satisfying the Armijo condition
can be found in a finite number of iterations) and is convergent; this is proved in Section~\ref{sec:convergence}.
Step~\ref{alg:approxmin} does not require the exact solution of problem~\eqref{eq:quadpb}, but only
the computation of an approximate solution such that condition \eqref{eq:err} at line~\ref{alg:approxmin}
of the algorithm holds, with $ \lim_{k \to \infty} \varepsilon_k = 0 $.

In Section~\ref{sec:convergence} we also show that the first condition in \eqref{eq:err} is satisfied if
\begin{equation}
\label{eq:innerstop}
    \| \nabla_\mS F_k (\widehat{\bx}^{(k)}) \| \le \theta^k  \| \nabla_\mS F_k (\bx^{(0)}) \| ,
\end{equation}
and $\theta \in (0,1)$. Therefore, the first condition in \eqref{eq:err} can be replaced by another one
which is simple to verify when the projected gradient can be easily computed, e.g.,
in the practical cases where $\mS$ is one of the sets in~\eqref{eq:constr1}-\eqref{eq:constr2}.

The second condition in~\eqref{eq:err} can be achieved by using any constrained minimization algorithm.
We note that, for the restoration problems considered in this work, gradient-projection methods, such as those
in~\cite{BonettiniZanellaZanni2009,diSerafinoToraldoViolaBarlow2018,MoreToraldo1991},
are suited to the solution of the inner problems~\eqref{eq:quadpb}. Indeed, numerical experiments
have shown that very low accuracy is required in practice in the solution of the inner problems;
furthermore, the computational cost per iteration of gradient projection methods is modest
when low-cost algorithms for the projection onto the feasible set are available. More details on the
inner method used in our experiments are given in Section~\ref{sec:experiments}.

\section{Well-posedness and convergence\label{sec:convergence}}

In order to prove that ACQUIRE is well posed, we need the following lemma \cite[Lemma A24]{Bertsekas1999}.
\begin{lemma}[Descent lemma]
\label{lem:descent}
Let $f : \R^n \to \R$ be continuously differentiable and let $\bx, \by \in \R^n$. If
there exists $L > 0$ such that
$$
    \| \nabla f(\bx + t \by) - \nabla f (\bx) \| \le L t \| \by \| \; \mbox{ for all } t \in [0,1] ,
$$
then
$$
   f (\bx + \by) \le f(\bx) + \nabla f (\bx)^T \by + \frac{L}{2} \| \by \|^2.
$$
\end{lemma}

We also observe that, at step~\ref{alg:approxmin} of Algorithm~\ref{alg:acquire}, we can find $\widehat{\bx}^{(k)} \ne \bx^{(k)}$
unless $\bx^{(k)}$ is the solution $\overline\bx^{(k)}$ of problem~\eqref{eq:quadpb}. However, in this case
$\bx^{(k)}$ is the solution of problem~\eqref{eq:poispb2}, since the gradients, and hence the projected gradients,
of the objective functions of the two problems coincide at $\bx^{(k)}$. Therefore, in the following
we can assume that $\widehat{\bx}^{(k)} \ne \bx^{(k)}$.

The next theorem shows that the steplength $\alpha_k$ required to obtain the iterate $\bx^{(k+1)}$ can be found
after a finite number of steps and that it is bounded away from zero.
\begin{theorem}
\label{thm:steplength}
Let $\delta \in (0,1)$. There exist $\overline\alpha > 0$ independent of $k$ and an integer $j_k \ge 0$
such that for $\alpha_k = \delta^{\, j_k}$
\begin{eqnarray}
   & & F(\bx^{(k)}_\alpha) \le  F(\bx^{(k)}) + \eta \alpha_k \nabla F (\bx^{(k)})^T (\widehat{\bx}^{(k)} - \bx^{(k)}),
          \label{eq:armijo} \\
   & & \alpha_k \ge \overline\alpha \label{eq:lboundalpha}.
\end{eqnarray}
\end{theorem}

\begin{proof}
For $F$ has Lipschitz continuous gradient, by applying Lemma~\ref{lem:descent} we get
$$
   F(\bx^{(k)}_\alpha)  \le F (\bx^{(k)}) + \alpha_k \nabla F (\bx^{(k)})^T (\widehat{\bx}^{(k)} - {\bx}^{(k)}) +
   \alpha_k^2 \frac{L }{2} \| \widehat{\bx}^{(k)} - {\bx}^{(k)} \|^2 ,
$$
where $L$ is the Lipschitz constant of $\nabla F$. Then, \eqref{eq:armijo} holds if we find
$\alpha_k$ such that
$$
   \nabla F (\bx^{(k)})^T (\widehat{\bx}^{(k)} - {\bx}^{(k)}) + \alpha_k  \frac{L }{2} \| \widehat{\bx}^{(k)} - {\bx}^{(k)} \|^2
   \le \eta \nabla F (\bx^{(k)})^T (\widehat{\bx}^{(k)} - {\bx}^{(k)}),
$$
or, equivalently,
\begin{equation}
\label{eq:thstep_1}
   (1 - \eta) \nabla F (\bx^{(k)})^T (\widehat{\bx}^{(k)} - {\bx}^{(k)}) +
    \alpha_k \frac{L }{2} \| \widehat{\bx}^{(k)} - {\bx}^{(k)} \|^2 \le 0 .
\end{equation}
From $\nabla F_k(\bx^{(k)}) = \nabla F(\bx^{(k)})$, the strong convexity of $F_k$ and
step~\ref{alg:approxmin} of Algorithm~\ref{alg:acquire}, it follows that
\begin{eqnarray}
   \nabla F(\bx^{(k)})^T (\widehat{\bx}^{(k)} - {\bx}^{(k)})
       &  =   & \nabla F_k(\bx^{(k)})^T (\widehat{\bx}^{(k)} - {\bx}^{(k)}) \nonumber \\
       & \le & F_k(\widehat{\bx}^{(k)}) - F_k(\bx^{(k)}) - \frac{\gamma}{2} \| \widehat{\bx}^{(k)} - \bx^{(k)} \|^2
           \label{eq:thstep_2} \\
       & \le & - \frac{\gamma}{2} \| \widehat{\bx}^{(k)} - \bx^{(k)} \|^2 , \nonumber
\end{eqnarray}
where $\gamma$ is the strong convexity parameter of $F_k$. Thus, \eqref{eq:thstep_1}
holds for any $\alpha_k$ such that
$$
   \frac{\gamma}{2} (\eta - 1)  \| \widehat{\bx}^{(k)} - \bx^{(k)} \|^2 +
   \alpha_k \frac{L }{2} \| \widehat{\bx}^{(k)} - {\bx}^{(k)} \|^2 \le 0 .
$$
By choosing the first nonnegative integer $j_k$ such that
$$
    \delta^{\, j_k} \le  \min \left \{ 1, \frac{ \gamma \, (1 - \eta)}{L} \right \}
$$
and setting
$$
   \overline\alpha = \min \left \{ 1, \frac{\delta\, \gamma \, (1 - \eta)}{L} \right \}
$$
we get the thesis.
\end{proof}

Now we prove that the sequence generated by ACQUIRE  has a subsequence converging to
a solution of problem~\eqref{eq:poispb2}. Because of the convexity of $F$, it is sufficient to prove that the
subsequence converges to a constrained stationary point of $F$.
\begin{theorem}
\label{th:conv}
Let $\{ \bx^{(k)} \}$ be the sequence generated by Algorithm~\ref{alg:acquire}. Then there exists a subsequence
$\{ \bx^{(k_j)} \}$ such that
$$
     \lim_{k_j \to \infty}  \bx^{(k_j)} = \overline\bx,
$$
where $\overline{\bx} \in \mS$ is such that $\nabla_\mS F(\overline\bx) = 0$.
Furthermore, any limit point $\widetilde\bx$ of $\{ \bx^{(k)} \}$ is such that $\nabla_\mS F(\widetilde\bx) = 0$.
\end{theorem}
\begin{proof}
Let $\alpha_k = \delta^{\, j^k}$, where $j_k$ is given in Theorem~\ref{thm:steplength}.
By~\eqref{eq:armijo} and \eqref{eq:thstep_2} we have
$$
   F(\bx^{(k+1)}) - F(\bx^{(k)}) \le - \, \alpha_k \eta \, \frac{\gamma}{2} \, \| \widehat{\bx}^{(k)} - \bx^{(k)} \|^2 \le 0 ;
$$
then $\{ F(\bx^{(k)}) \}$ is convergent, 
and the coercivity of $F$ implies that $\{ \bx^{(k)} \}$ is bounded. Since $\alpha_k \ge \overline\alpha > 0$, we have that
\begin{equation}
\label{eq:limxhat}
    \lim_{k \to \infty} \| \widehat\bx^{(k)} - \bx^{(k)} \| = 0
\end{equation}
and $\{ \widehat\bx^{(k)} \}$ is bounded. This, together with $\| \overline\bx^{(k)} - \bx^{(k)} \| \le
\| \overline\bx^{(k)}  - \widehat\bx^{(k)} \| + \| \widehat\bx^{(k)} - \bx^{(k)} \|$ and
the first inequality in~\eqref{eq:err}, implies that
\begin{equation}
\label{eq:diffxover}
    \lim_{k \to \infty} \| \overline\bx^{(k)} - \bx^{(k)} \| = 0
\end{equation}
and hence $\{ \overline\bx^{(k)} \}$ is bounded. Passing to subsequences, we have
\begin{equation}
\label{eq:limxover}
     \lim_{k_j \to \infty}  \bx^{(k_j)} = \lim_{k_j \to \infty}  \overline\bx^{(k_j)} = \overline\bx \in \mS.
\end{equation}
Since the projection onto a nonempty closed convex set is nonexpansive, we get
$$
   \| \nabla_\mS F (\overline\bx^{(k_j)}) \| =
   \| \nabla_\mS F (\overline\bx^{(k_j)}) - \nabla_\mS F_{k_j} (\overline\bx^{(k_j)}) \|  \le
   \| \nabla F (\overline\bx^{(k_j)}) - \nabla F_{k_j} (\overline\bx^{(k_j)}) \| ,
$$
and,  by using~\eqref{eq:limxover},
$$
    \lim_{k_j \to \infty} \| \nabla_\mS F (\overline\bx^{(k_j)}) \| =
    \lim_{k_j \to \infty} \| \nabla F (\overline\bx^{(k_j)}) - \nabla F_{k_j} (\overline\bx^{(k_j)}) \| = 0.
$$
Then, for the lower semicontinuity of $\| \nabla_\mS F \|$, we have
$$
   \nabla_\mS F(\overline\bx) = \b0.
$$
If $\widetilde\bx$ is any limit point of $\{ \bx^{(k)} \}$, then $\widetilde\bx \in \mS$
and, by exploiting \eqref{eq:diffxover} and passing to subsequences, we have
\begin{equation}
\label{eq:limxtilde}
     \lim_{k_r \to \infty}  \bx^{(k_r)} = \lim_{k_r \to \infty}  \overline\bx^{(k_r)} = \widetilde\bx \in \mS.
\end{equation}
By reasoning as above we get
$$
   \nabla_\mS F(\widetilde\bx) = \b0,
$$
which concludes the proof.
\end{proof}

We note that ACQUIRE fits into the very general algorithmic framework presented
in~\cite{FacchineiLamparielloScutari2017} and hence Theorem~\ref{th:conv}
could be derived by specializing and adapting the convergence theory of that framework, taking into
account the specific properties of the functions $D_{KL}(\bx)$ and $TV_\mu (\bx)$ and their quadratic approximations
$D_{KL}^{(k)} (\bx)$ and $TV_\mu^{(k)} (\bx)$, and the line search used. However, for the sake
of clarity and self-consistency, we decided to prove the convergence of~Algorithm~\ref{alg:acquire} from scratch.

Now we show that if the objective function is strictly convex, the whole sequence $\{ \bx^{(k)} \}$
converges to the minimizer of problem~\eqref{eq:poispb2}.
\begin{theorem}
\label{th:conv_whole_sequence}
Assume that the function $F$ is strictly convex. Then the sequence $\{ \bx^{(k)} \}$ generated
by Algorithm~\ref{alg:acquire} converges to a point $\overline\bx \in \mS$ such that $\nabla_\mS F(\overline\bx) = 0$.
\end{theorem}
\begin{proof}
We follow the line of the proof of Lemma~2 in~\cite{BirginKrejicMartinez2003}.
By Theorem~\ref{th:conv} we know that there exists a limit point $\overline\bx$ of $\{ \bx^{(k)} \}$
such that  $\nabla_\mS F(\overline\bx) = 0$. Since $F$ is strictly convex, $\overline\bx$ is the optimal solution
of problem~\eqref{eq:poispb2}. We must prove that $\{ \bx^{(k)} \}$ converges to $\overline\bx$.

From $\alpha_k \le 1$ it follows that $\| \bx^{(k+1)} - \bx^{(k)} \|\leq \| \widehat\bx^{(k)} - \bx^{(k)} \|$ and, by \eqref{eq:limxhat},
$$
   \lim_{k \to \infty} \| \bx^{(k+1)} - \bx^{(k)} \| = 0.
$$
Since $\overline\bx$ is a strict minimizer, there exists $\delta > 0$ such that $F(\overline\bx) < F(\bx)$ for all $\bx \in \mS$
such that $0 < \| \bx - \overline\bx \| \le \delta$. For all $\varepsilon \in (0, \delta)$, it follows from Theorem~\ref{th:conv} that
the set $B = \left\{ \bx \in \mS : \delta \le \| \bx - \overline\bx \| \le \varepsilon \right\}$ does not contain any limit point of $\{ \bx^{(k)} \}$;
thus, there exists $k_0$ such that $\bx^{(k)} \not\in B$ for all $k > k_0$ .
Let $ k_1 \ge k_0 $ such that, for all $ k > k_1$,
$$
     \| \bx^{(k+1)} - \bx^{(k)} \|< \delta - \varepsilon.
$$
Let $K$ be the set of indices defining a subsequence of $\{ \bx^{(k)} \}$ converging to $\overline\bx$. There exists $k \in K$, $k > k_1$, such that
$$
   \| \bx^{(k)} - \overline\bx \| < \varepsilon ,
$$
and hence
$$
    \| \bx^{(k+1)} - \overline\bx \|  \leq  \| \bx^{(k+1)} - \bx^{(k)}\| + \| \bx^{(k)} - \overline\bx \| < \delta - \varepsilon + \varepsilon = \delta.
$$
Since $\bx^{(k+1)} \not \in B$, we get
$$
    \| \bx^{(k+1)} -\overline\bx \|< \varepsilon.
$$
By the same argument we can prove that $\| \bx^{(k+j)} -\overline\bx \|< \varepsilon$ implies $\| \bx^{(k+j+1)} -\overline\bx \|< \varepsilon$,
and hence, by induction, we have
$$
   \| \bx^{(k+j)} -\overline\bx \|< \varepsilon \quad \mbox{for all } j.
$$
Since $\varepsilon$ is arbitrary, the thesis holds.
\end{proof}

We conclude this section by showing that the stopping criterion~\eqref{eq:innerstop}
can be used to determine $\widehat\bx^{(k)}$ at step~\ref{alg:approxmin} of ACQUIRE.
\begin{theorem}
\label{thm:innerstop}
Assume that~\eqref{eq:innerstop} holds for some $\theta \in (0,1)$. Then, there exists $\{ \varepsilon_k \}$,
with $\varepsilon_k>0$ and $\lim_{k \to \infty} \varepsilon_k = 0$, such that~\eqref{eq:err} holds.
\end{theorem}

\begin{proof}
First we recall that $- \nabla F_k(\bx) = \nabla_\mS F_k(\bx) + P_{N_\mS (\bx)} (- \nabla F_k(\bx))$ (see \eqref{eq:grad_decomp}).
Since $F_k$ is strongly convex with parameter $\gamma$ and $\overline\bx^{(k)}$ is the
solution of problem~\eqref{eq:quadpb}, we have
\begin{eqnarray*}
    \frac{\gamma}{2} \| \widehat\bx^{(k)} -\overline\bx^{(k)} \|^2 & \le &
                  ( \nabla F_k(\widehat\bx^{(k)}) - \nabla F_k(\overline\bx^{(k)}) )^T
                  ( \widehat\bx^{(k)} - \overline\bx^{(k)} ) \\
    & = & ( \nabla_\mS F_k(\widehat\bx^{(k)}) )^T  ( \overline\bx^{(k)} - \widehat\bx^{(k)} ) +
              P_{N_\mS (\widehat\bx^{(k)})} ( - \nabla F_k(\widehat\bx^{(k)}) )^T  ( \overline\bx^{(k)} - \widehat\bx^{(k)} ) \\[1.5mm]
    & + & P_{N_\mS (\overline\bx^{(k)})} ( - \nabla F_k(\overline\bx^{(k)}) )^T  ( \widehat\bx^{(k)} - \overline\bx^{(k)} ).
\end{eqnarray*}
Since $\overline\bx^{(k)} - \widehat\bx^{(k)}$ belongs to the tangent cone at $\widehat\bx^{(k)}$ and
$\widehat\bx^{(k)} - \overline\bx^{(k)}$ belongs to the tangent cone at $\overline\bx^{(k)}$,
we get
$$
    \frac{\gamma}{2} \| \widehat\bx^{(k)} -\overline\bx^{(k)} \|^2 \le
    ( \nabla_\mS F_k(\widehat\bx^{(k)}) )^T  ( \overline\bx^{(k)} - \widehat\bx^{(k)} ) \le
    \| \nabla_\mS F_k(\widehat\bx^{(k)}) \|  \| \overline\bx^{(k)} - \widehat\bx^{(k)} \| .
$$
It follows that
$$
    \| \widehat\bx^{(k)} -\overline\bx^{(k)} \| \le \frac{2}{\gamma}  \| \nabla_\mS F_k(\widehat\bx^{(k)})  \| ;
$$
thus, by requiring that
$$
   \| \nabla_\mS F_k(\widehat\bx^{(k)})  \| \le \theta^k \| \nabla_\mS F_k(\bx^{(0)}) \|
$$
and setting $\varepsilon_k = \theta^k (2/\gamma) \| \nabla_\mS F_k(\bx^{(0)}) \|$, we get
$$
    \| \widehat\bx^{(k)} -\overline\bx^{(k)} \|  \le \varepsilon_k.
$$
\par
\vspace*{-22pt}
\
\end{proof}


\section{Numerical experiments\label{sec:experiments}}

ACQUIRE was implemented in MATLAB, using as inner solver the scaled gradient projection (SGP)
method proposed in~\cite{BonettiniZanellaZanni2009}, widely applied in the solution of image restoration problems.
In particular, the implementation of SGP provided by the \texttt{SPG-dec} MATLAB code,
available from \url{http://www.unife.it/prin/software}, was exploited.

The SGP iteration applied to problem~\eqref{eq:quadpb} reads:
$$
     \bz^{(j+1)} = \bz^{(j)} + \rho_j \left( P_{\mS, \, C_j^{-1}} \left( \bz^{(j)} - \nu_j C_j \nabla F_k (\bz^{(j)}) \right) - \bz^{(j)} \right) ,
$$
where $\bz^{(0)} = \bx^{(k)}$, $\rho_j$ is a line-search parameter ensuring that $\bz^{(j+1)}$ satisfies a sufficient decrease
condition, $\nu_j$ is a suitably chosen steplength, $C_j$ is a diagonal positive definite matrix with diagonal entries bounded
independently of $j$, and $P_{\mS, \, C_j^{-1}}$ is the projection operator onto $\mS$ with respect to the norm induced
by the matrix $C_j^{-1}$ (the dependence on $k$ has been neglected for simplicity).
Several efficient rules can be exploited to define the steplength $\nu_j$ for the quadratic problem~\eqref{eq:quadpb}  (see, e.g.,
\cite{BarzilaiBorwein1988,DeAsmundisdiSerafinoHagerToraldoZhang2014,DeAsmundisdiSerafinoLandi2016,
DeAsmundisdiSerafinoRiccioToraldo2013,Fletcher2012,FrassoldatiZanniZanghirati2008} and the references therein).
In particular, SGP uses a modification of the ABB$_{\rm min}$ adaptive Barzilai-Borwein steplength defined
in~\cite{FrassoldatiZanniZanghirati2008}, which takes into account the scaling matrix $C_j$ (see~\cite{BonettiniZanellaZanni2009}
for details); according to the analysis in~\cite{diSerafinoRuggieroToraldoZanni2018}, this steplength appears very effective.
Since the steplength is computed by taking into account a certain number, say $q$, of suitable previous steplengths,
we modified \texttt{SPG-dec} to avoid resetting the steplength each time the code was called, and to compute it by using $q$ steplengths
from the previous call. The diagonal scaling matrix $C_j$ was set as in~\cite[section 3.3]{ZanellaBoccacciZanniBertero2009} and
$q$ was chosen equal to its defaul value in \texttt{SPG-dec}, i.e., $q = 3$.
The SGP iterations were stopped according to~\eqref{eq:innerstop}. For all the tests considered here, we found experimentally
that $\theta = 0.1$ worked well in the first iterations of ACQUIRE; on the other hand, criterion \eqref{eq:innerstop} with this value
of $\theta$ soon becomes demanding, and fixing also a maximum number inner iterations was a natural choice. Setting this number
to 10 was effective in our experiments. Defaults were used for the remaining features of \texttt{SPG-dec}.


The parameter $\gamma$ in~\eqref{eq:quadkl} was set equal to $10^{-5}$. The nonmonotone line search proposed
in~\cite{GrippoLamparielloLucidi1986} was implemented at line~\ref{alg:line_search}
of Algorithm~\ref{alg:acquire}, with memory length equal to 5, $\eta = 10^{-5}$, and $\delta = 0.5$.
ACQUIRE was stopped using the following criterion
\begin{equation}
\label{eq:stop}
\| \bx^{(k+1)} - \bx^{(k)} \| \le \text{Tol} \, \|\bx^{(k)} \| ,
\end{equation}
i.e., when the relative change in the restored image went below a certain threshold.

ACQUIRE was compared with five state-of-the-art methods: PDAL, SGP,  SPIRAL-TV, SPLIT BREGMAN and VMILA.
By PDAL we denote our MATLAB implementation of the primal-dual algorithm proposed in \cite[Algorithm~2]{WenChanZeng2016},
where we replaced the Chambolle-Pock algorithm \cite{ChambollePock2011} by the more efficient Primal Dual Algorithm with Linesearch introduced
in~\cite{MalitskyPock2018}. Concerning the parameters of PDAL, following \cite[Section~6]{MalitskyPock2018} we set $ \mu=0.7$,
$\delta=0.99$ and  $\beta = 25$. The initial steplength was chosen as $ \tau = \sqrt{2/\omega} $, where $\omega$ is an underestimate of $\|M^T M\|$
and $M = \left[ A^T \; D_1^T \; \ldots \; D_n^T \right]^T $ is the matrix linking the primal and dual variables.
SPIRAL-TV is the proximal-gradient method presented in~\cite{HarmanyMarciaWillet2012}; a MATLAB
implementation of it is available from \url{http://drz.ac/code/}.
By SPLIT BREGMAN we denote a version of the method proposed in \cite{GoldsteinOsher2009},
which was specialized for problem~\eqref{eq:poispb} \cite{Getreuer2012}
and implemented in the MATLAB code \texttt{tvdeconv} available from \url{http://dev.ipol.im/~getreuer/code/}.
Finally, VMILA is the variable-metric inexact line-search proximal-gradient method described in~\cite{BonettiniLorisPortaPrato2016},
whose MATLAB implementation can be found at
\url{http://www.oasis.unimore.it/site/home/software.html}. In all the methods, the stopping criterion~\eqref{eq:stop}
was applied. SGP was run with the same setting of parameters used to solve the subproblems in ACQUIRE.
For SPIRAL-TV, SPLIT BREGMAN and VMILA, the default values of the parameters were used.

PDAL, SPIRAL-TV, SPLIT BREGMAN and VMILA do not require any smooth approximation of TV and were
run directly on problem~\eqref{eq:poispb}. Therefore, our comparison also provides some insight into
the effects of using a smoothed version of TV. ACQUIRE was run with and without the flux constraint, i.e., using
both feasible sets $\mS_1$ and $\mS_2$ -- see~\eqref{eq:constr1} and \eqref{eq:constr2}. However, since
the use of the flux constraint did not lead to any significant difference
in the restored images, and this constraint was not available in the implementations of  SPIRAL-TV, SPLIT BREGMAN
and VMILA, we report only the results for $\mS = \mS_1$.

As already noted, when the matrix $A$ represents a convolution, the matrix-vector products involving the matrices $A$
and $A^T$ can be performed by using fast algorithms. This is the case for all the experiments considered in this work.
Since periodic boundary conditions were considered for all the images used as test problems, the matrix-vector products were performed
by exploiting the MATLAB FFT functions \texttt{fft2} and \texttt{ifft2}.

In order to build the test problems used in the experiments, four reference images were chosen:
\emph{cameraman}, \emph{micro}, \emph{phantom} and \emph{satellite},
shown in Figure~\ref{fig:test_images}. The cameraman image,
available in the MATLAB Image Processing Toolbox,
is widely used in the literature since it contains both sharp edges and flat regions
and presents a nice mixture of smooth and nonsmooth regions;
micro is the confocal microscopy phantom described in~\cite{WillettNowak2003};
phantom is the famous Shepp-Logan brain phantom described in~\cite{SheppLogan1974};
finally, the satellite image comes from the RestoreTools package~\cite{NagyPalmerPerrone2004}.
The size of cameraman, phantom and satellite is $256 \times 256$, while the size of
micro is $128 \times 128$.

A first set of test problems, T1, was obtained by convolving each reference image with a Gaussian PSF
and corrupting the resulting image with Poisson noise. A further set of test problems, T2, was built by
convolving some of the images with a motion blur PSF and an out-of-focus PSF, and then introducing Poisson noise.
Details about the PSFs and the Poisson noise are given in Subsections~\ref{sec:results1} and~\ref{sec:results2},
where the results of numerical experiments performed by using the corrupted images are also reported.

All the experiments were carried out on a 2.5 GHz Intel Core i7 processors with 16 GB of RAM,
4 MB of L3 cache and the macOS 10.13.6 operating system, using MATLAB R2018b.

\subsection{Results on images with Gaussian blur\label{sec:results1}}

The Gaussian blur PSF for constructing the test set T1 was computed by using
the function \texttt{psfGauss} from~\cite{NagyPalmerPerrone2004},
choosing the variance $\sigma$ as specified in Table~\ref{tab:test_pbs}.
In order to take into account the existence of some background emission,
$10^{-10}$ was added to all the pixels of the blurred image; obviously, the vector $\bb$ in $D_{KL}(\bx)$
was set as $\bb = 10^{-10} \be$. The Poisson noise was introduced with the
function \texttt{imnoise} from the MATLAB Image Processing Toolbox. Note that for this type of noise,
which affects the photon counting process, the Signal-to-Noise Ratio (SNR) is usually estimated by
\begin{equation*}
    \text{SNR} = 10\log_{10}\left(\frac{N_\text{exact}}{\sqrt{N_\text{exact} + N_\text{background}}} \right),
\end{equation*}
where $N_\text{exact}$ and $N_\text{background}$ are the total number of photons in the exact image to be recovered
and in the background term, respectively. Therefore, in order to obtain noisy and blurred images with SNR equal to 35 and 40,
the intensities of the reference images were suitably pre-scaled. The resulting images are shown
in Figures~\ref{fig:cameraman}-\ref{fig:satellite} (left columns).

The regularization parameter $\lambda$ was set by trial and error, as described next.
The search for a suitable value of $\lambda$ was carried out separately
for the minimization problem~\eqref{eq:poispb}, which uses the original TV, and the minimization
problem~\eqref{eq:poispb2}, which uses the smoothed TV. In the latter case, ACQUIRE was run
several times on each test image, for 25 seconds and with Tol $= 0$ (see~\eqref{eq:stop}),
slowly varying the value of $\lambda$ at each execution. The value of $\lambda$ corresponding
to the smallest relative error at the last iteration, was chosen to perform the experiments with ACQUIRE
and SGP.  Note that, by running SGP for more than 25 seconds, we also verified that the selected value
of $\lambda$ was suited to SGP too. The relative error was
computed as $\| \bx^{(k)} - \bx^* \| / \| \bx^* \|$, where $\bx^*$ denotes the original image.
The values of $\lambda$ for problem~\eqref{eq:poispb}
were set using the same strategy applied to~\eqref{eq:poispb2}. In this case, instead of ACQUIRE,
for each image we used the method that appeared more efficient among PDAL, SPIRAL-TV, SPLIT-BREGMAN
and VMILA, on the basis of preliminary experiments. All the values of $\lambda$ are reported in Table~\ref{tab:test_pbs}.
The same value of the regularization parameter was determined for both the original and the smoothed TV,
except for the satellite image; however, very close values of $\lambda$ were obtained in this case.

The parameter $\mu$ in the smoothed version of TV was set as $\mu = 10^{-2}$, which, by numerical experiments,
appeared to achieve a good tradeoff between approximation accuracy and computational effort, for all the test problems.
Indeed, as the value of $\mu$ decreases, $TV_\mu$ becomes a more accurate approximation of $TV$, but the
condition number of its Hessian increases. For both versions of TV,  each corrupted image was scaled by division
by its largest intensity value before applying any method; the scaled image was also used as starting guess, i.e., we set
$\bx^{(0)} = \by$. We also performed experiments by setting $\bx^{(0)}$ as the vector with entries equal to the
flux of the scaled image divided by the number of pixels of the image, but we could not see any significant difference in the results.

ACQUIRE was compared with all the other methods on the test problems previously described,
in terms of accuracy and execution time.
Six values of Tol were considered, Tol $= 10^{-2}, 10^{-3}, \ldots, 10^{-7}$,
with the aim of assessing the behavior of ACQUIRE with different accuracy requirements
and getting useful information for the effective use of an automatic stopping rule.
A maximum execution time of 25 seconds was also set for all the methods.

Figures~\ref{fig:compar35} and \ref{fig:compar40} show the relative errors and the execution times of each method,
in seconds, versus the stopping tolerances, for the problems with SNR $= 35$ and SNR $= 40$, respectively.
The images obtained with ACQUIRE and corresponding to the smallest
errors are shown in Figures~\ref{fig:cameraman}-\ref{fig:satellite} (right columns).
Further details concerning all the methods
are given in Tables~\ref{tab:compar35_minerr} and \ref{tab:compar40_minerr}, where we report the
smallest errors, the iterations performed to achieve them, the corresponding execution times and tolerances,
and the values of MSSIM for the restored images corresponding to the smallest errors.  MSSIM is
a structural similarity measure index~\cite{WangBovikSheikhSimoncelli2004}
which is related to the perceived visual quality of the image; the higher its value, the better
the perceived similarity between the restored and original images.

We see that ACQUIRE generally does not need small tolerances to achieve small errors, because of its fast progress in the first
iterations, which produces large changes in the iterate. We note that in four test cases it reaches its minimum error
with Tol $= 10^{-3}$; this is consistent with the exploitation of second-order information to build the quadratic
model at each iteration. SGP generally achieves errors comparable with those of ACQUIRE, but its progress at each
iteration is slower, and hence it often requires smaller tolerances to avoid stopping prematurely. On the other hand,
a single iteration of ACQUIRE requires more time than an iteration of SGP, and the former method may be either
faster or slower than the latter in obtaining small errors. PDAL is able to achieve errors comparable with those of
ACQUIRE, but it generally requires smaller tolerances and larger times. VMILA is very efficient on both
instances of the cameraman problem and on the phantom problem with SNR $=35$, where it is faster than ACQUIRE
or comparable with it. However, there are some problems where VMILA makes very little progress in the first
iterations, leading to very premature stops, as shown by the almost constant execution times in the pictures.
The remaining methods are generally less efficient than the previous ones, because of their very slow progress in
reducing the error. We note that the errors of ACQUIRE and SGP show a light semiconvergence for some problems.
We were not able to completely remove this behavior by increasing the regularization parameter without significantly
deteriorating the visual quality of the image and decided to keep the value of $\lambda$ determined by the
procedure previously described. Finally, we note that the values of MSSIM corresponding to ACQUIRE
confirm that in most cases this method is able to provide better or similar quality images in comparison with
the other methods.

\subsection{Results on images with moving and out-of-focus blurs\label{sec:results2}}

In order to understand if the previous behavior of ACQUIRE also holds for blurs different from the Gaussian
one, and to further compare ACQUIRE with the other methods, we built the test set T2. We introduced in two reference images,
cameraman and satellite, the motion blur and and the out-of-focus blur, which simulate the linear motion of a camera
and the out-of-focus effect, respectively. Both blurs were computed by using the Matlab function \texttt{fspecial}. 
Poisson noise with SNR equal to 35 and 40 was introduced in the blurred images, as in the case of Gaussian blur.
The length and the angle (in degrees) of the motion, \texttt{len} and $\varphi$, and the radius of the disk kernel
for the out-of-focus effect, \texttt{rad}, are specified in Table~\ref{tab:test_mot_of}. The values of $\lambda$,
obtained with the procedure described in Subsection~\ref{sec:results1}, are reported in the same table.
In this case, small differences can be observed between the values of the regularization parameter
corresponding to the original and the smoothed TV. The parameter $\mu$ in the smoothed version of TV
was set again as $\mu = 10^{-2}$, on the basis of numerical experiments. Each noisy and blurred image
was scaled as in the case of Gaussian blur. The vector with entries equal to the flux of the scaled image
divided by the number of pixels of the image was used as starting guess, because, with the motion
and out-of-focus blurs, this choice generally appeared more effective than the choice of the scaled image.

The error and time histories shown in Figures~\ref{fig:compar_mot} and~\ref{fig:compar_of}
confirm the behavior of ACQUIRE: it is able to strongly reduce the error in the first iterations and hence,
although its single iteration is usually more expensive than a single iterations of the other methods, it
is competitive with those methods. Furthermore, ACQUIRE allows an effective use
of an automatic stopping rule, avoiding premature stopping, which may happen with other methods.
This is confirmed by the data in Tables~\ref{tab:compar_mot_minerr} and \ref{tab:compar_of_minerr},
which report, for all the methods, the smallest errors and the corresponding MSSIM values, number of iterations,
execution times and tolerances. The images restored by ACQUIRE and corresponding to the smallest errors are shown
in the right columns of Figures~\ref{fig:cameraman_mot}-\ref{fig:satellite_of}).

\section{Conclusions\label{sec:conclusions}}

We proposed ACQUIRE, a method for TV-based restoration of images corrupted by Poisson noise, modeled
by~\eqref{eq:poispb}. ACQUIRE is a line-search method which considers a smoothed version of TV and computes
the search directions by minimizing quadratic models built by exploiting second-order information about the
objective function, which is usually not taken into account in methods for problem~\eqref{eq:poispb}.
We proved that the sequence generated by our method has a subsequence converging to a minimizer of
the smoothed problem~\eqref{eq:poispb2} and that any limit point is a minimizer; furthermore, if the
problem is strictly convex, the whole sequence is convergent.
We note that convergence holds without requiring the exact minimization of the quadratic models;
low accuracy in this minimization can be used in practice, as shown by the numerical results.

Computational experiments on reference test cases showed that the exploitation of second-order information
is beneficial, since it generally leads to a significant reduction of the reconstruction error in the first iterations,
Furthermore, the capability of achieving a tradeoff between accuracy and efficiency makes ACQUIRE
competitive with well-established methods for TV-based Poisson image restoration.

Finally, we observe that our approach can be extended to other regularization models,
such as high-order Total Variation \cite{LiuHuangLvWang2017,PapafitsorosSchonlieb2014} and Total Generalized Variation \cite{BrediesHoller2014,BrediesKunischPock2010}, which have been proposed to reduce the staircase effect of TV and
retain the fine details of the image.

\section*{Acknowledgments}
We wish to thank the anonymous reviewers for their insightful comments and useful suggestions, which helped
us improve the quality of our work.

\clearpage


\begin{figure}[b!]
	\vspace*{2mm}
	\begin{center}
		\hspace*{-.55cm}
		\begin{tabular}{ccc}
			\includegraphics[width=.47\textwidth]{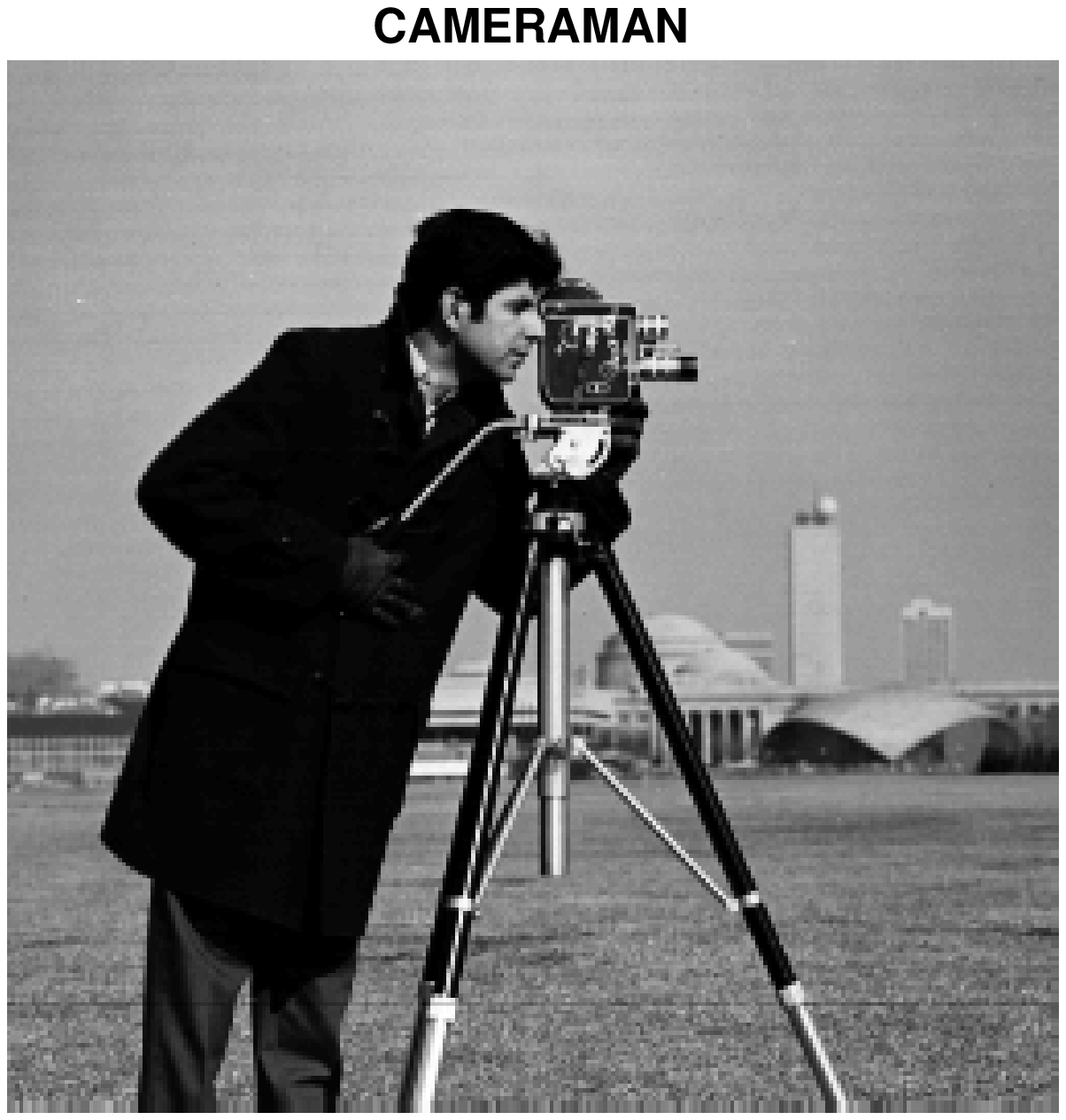}      &
			\hspace*{-1.5cm}
			\includegraphics[width=.47\textwidth]{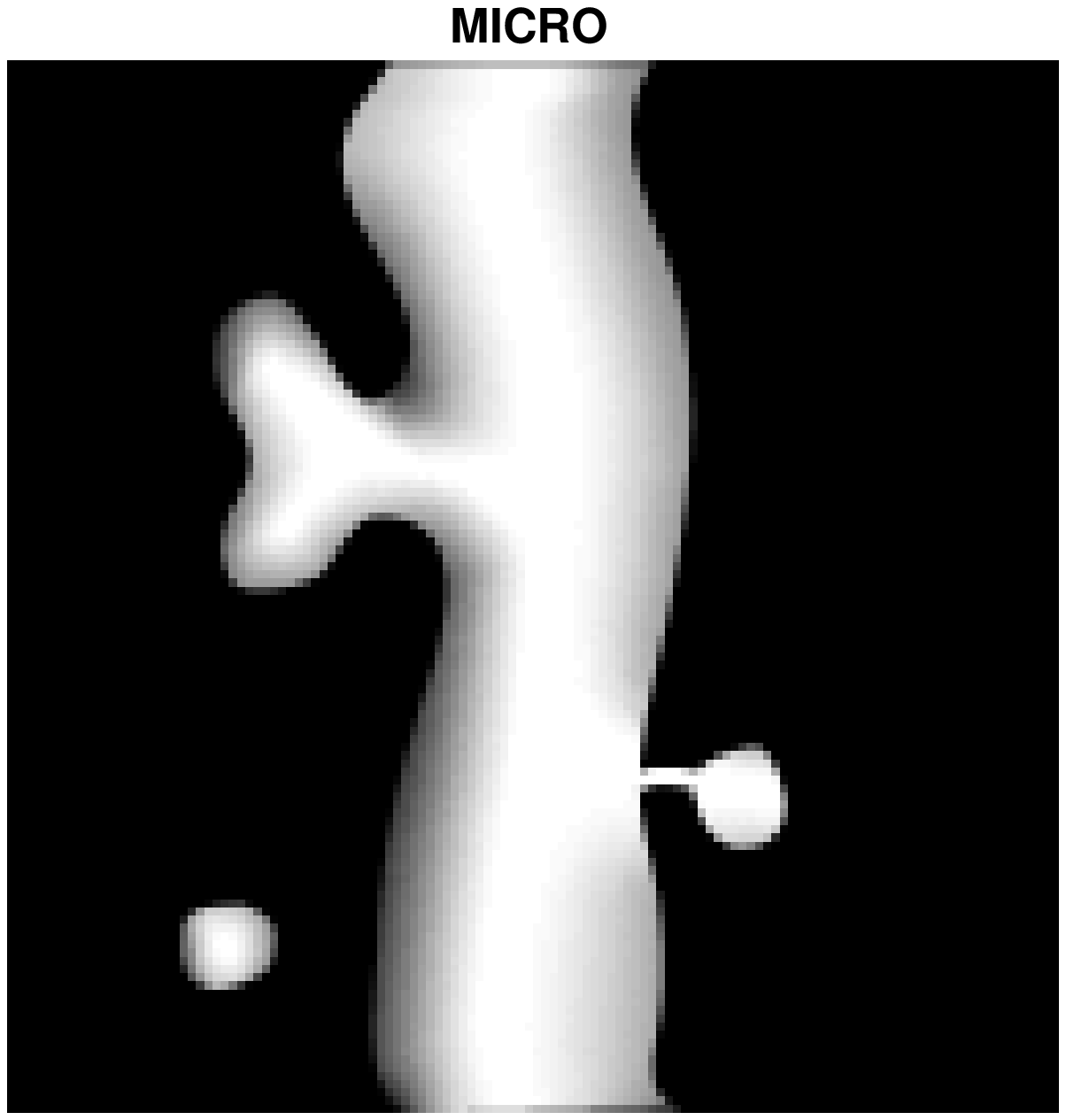}                \\[-4mm]
			\includegraphics[width=.47\textwidth]{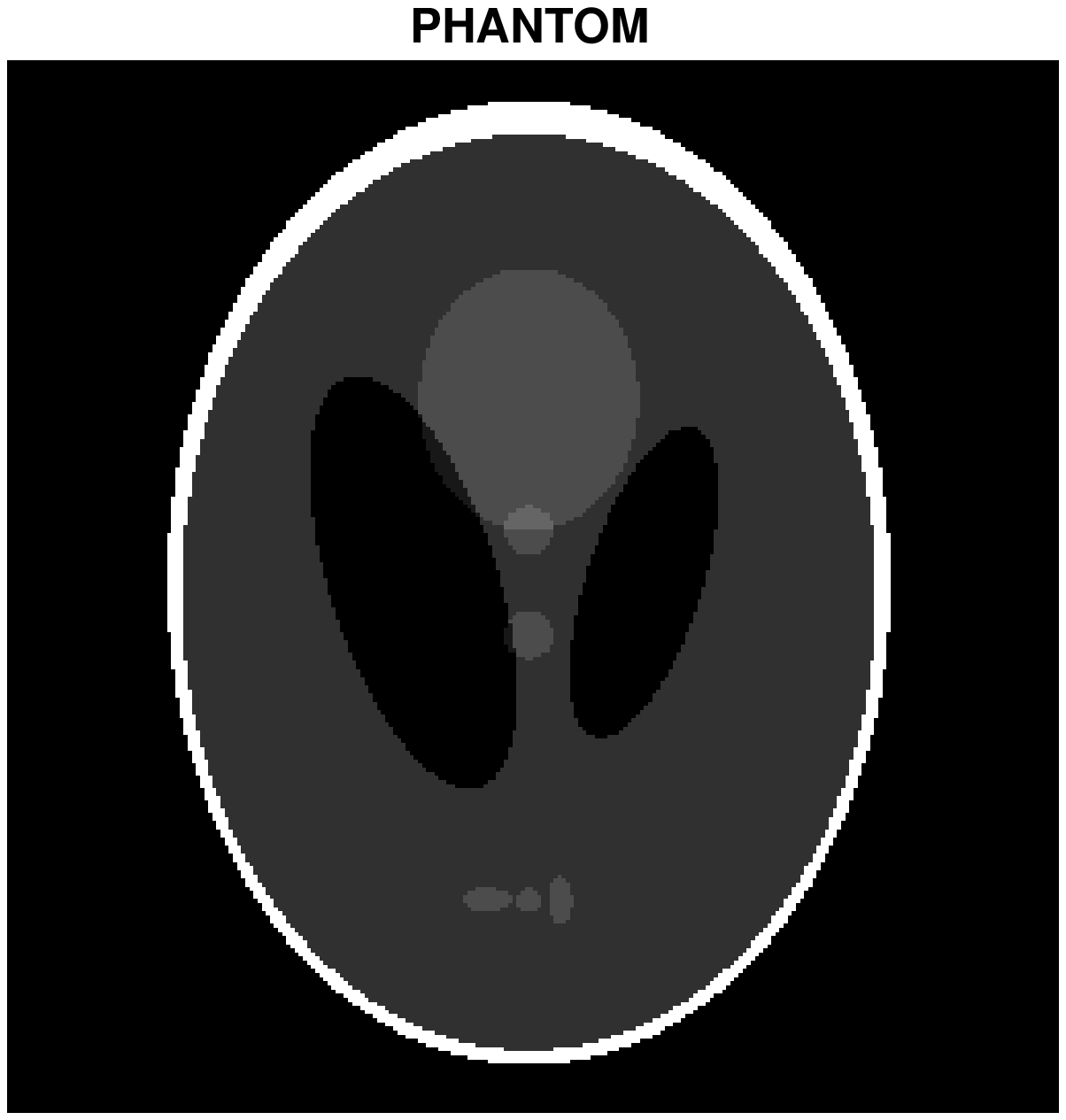}           &
			\hspace*{-1.5cm}
			\includegraphics[width=.47\textwidth]{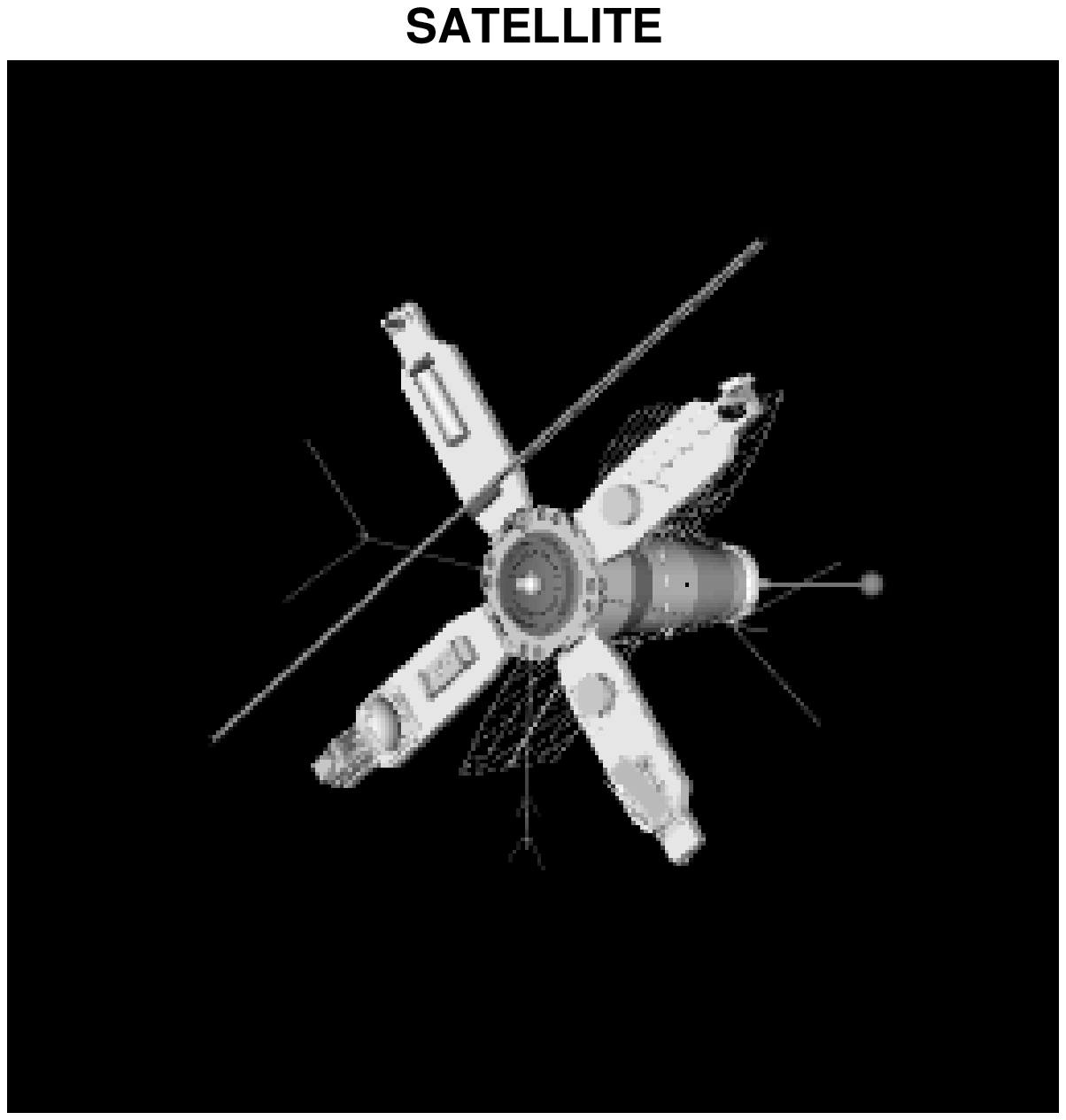}
		\end{tabular}
		\vspace*{-6mm}
		\caption{Reference images:
			cameraman, micro, phantom and satellite.\label{fig:test_images}}
	\end{center}
\end{figure}

\clearpage

\begin{figure}[t!]
	\vspace*{2mm}
	\begin{center}
		\hspace*{-.55cm}
		\begin{tabular}{ccc}
			\includegraphics[width=.47\textwidth]{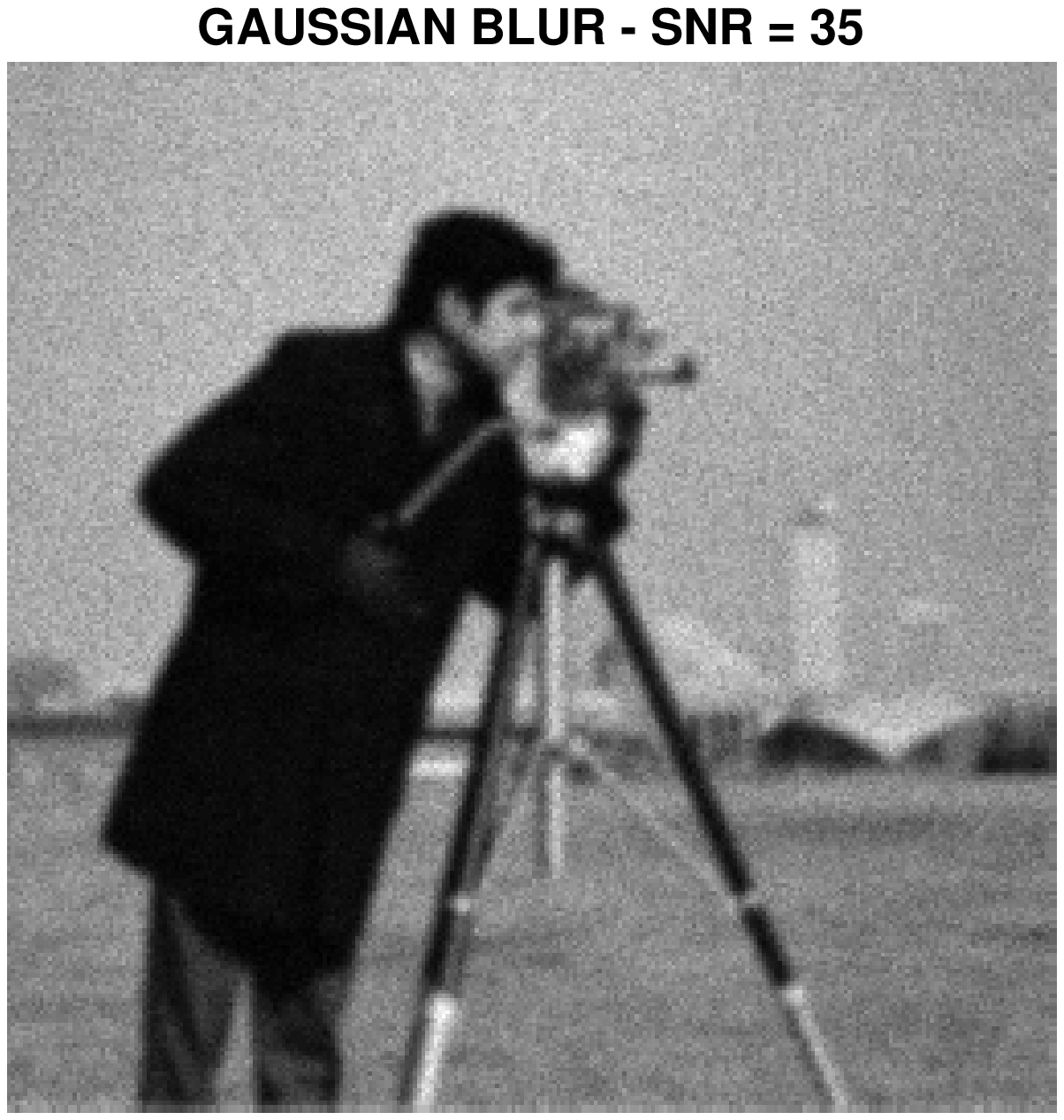}      &
			\hspace*{-1.5cm}
			\includegraphics[width=.47\textwidth]{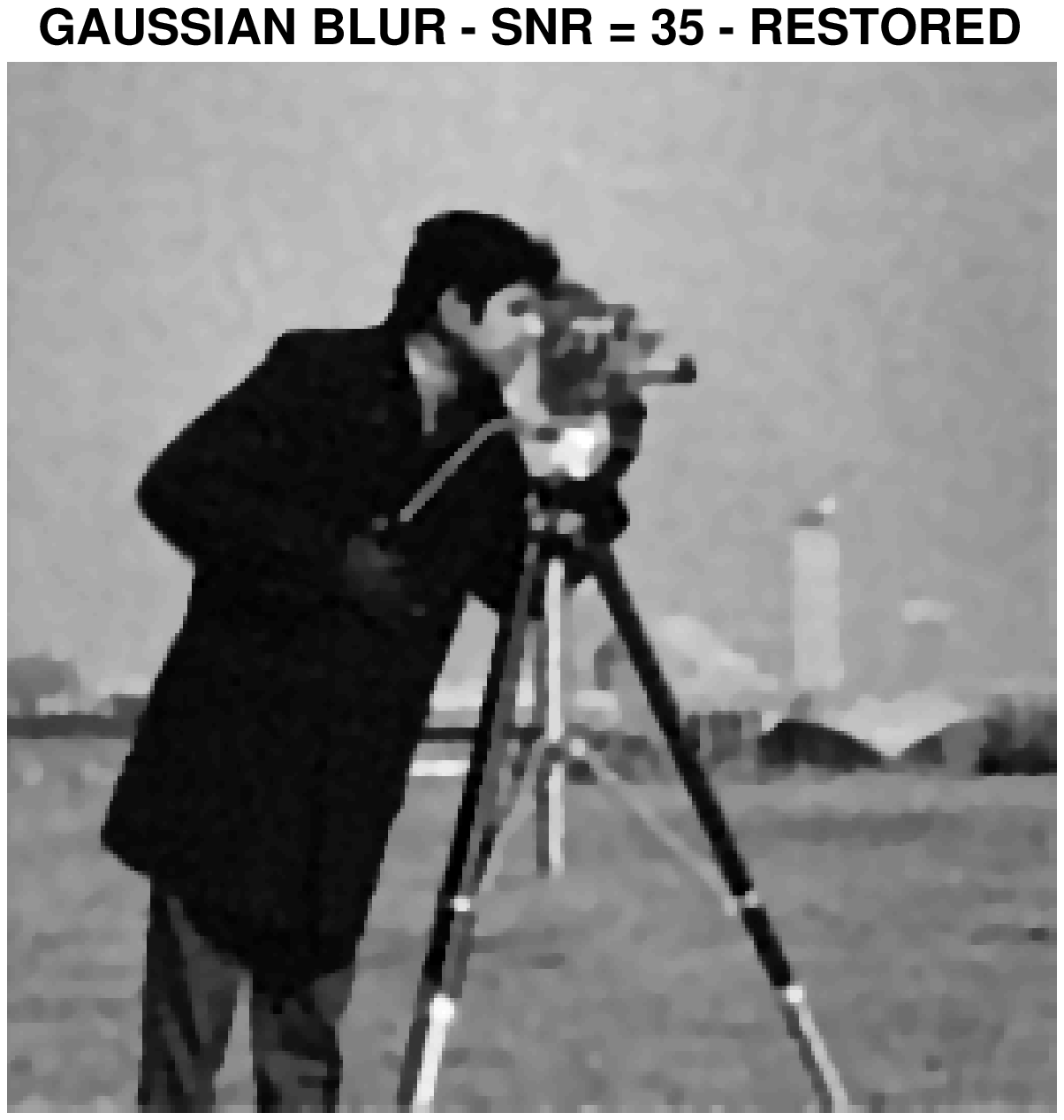}  \\[-4mm]
			\includegraphics[width=.47\textwidth]{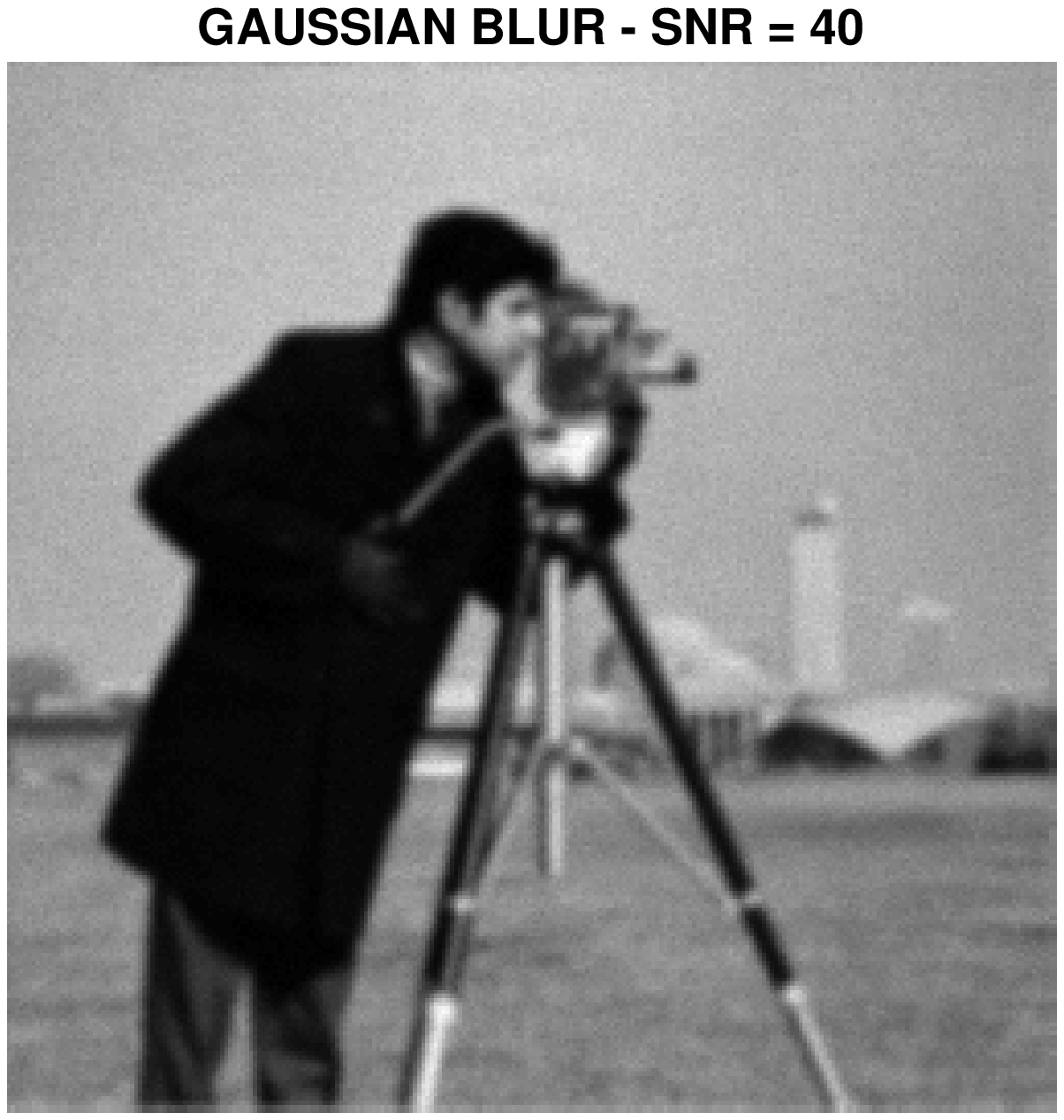}       &
			\hspace*{-1.5cm}
			\includegraphics[width=.47\textwidth]{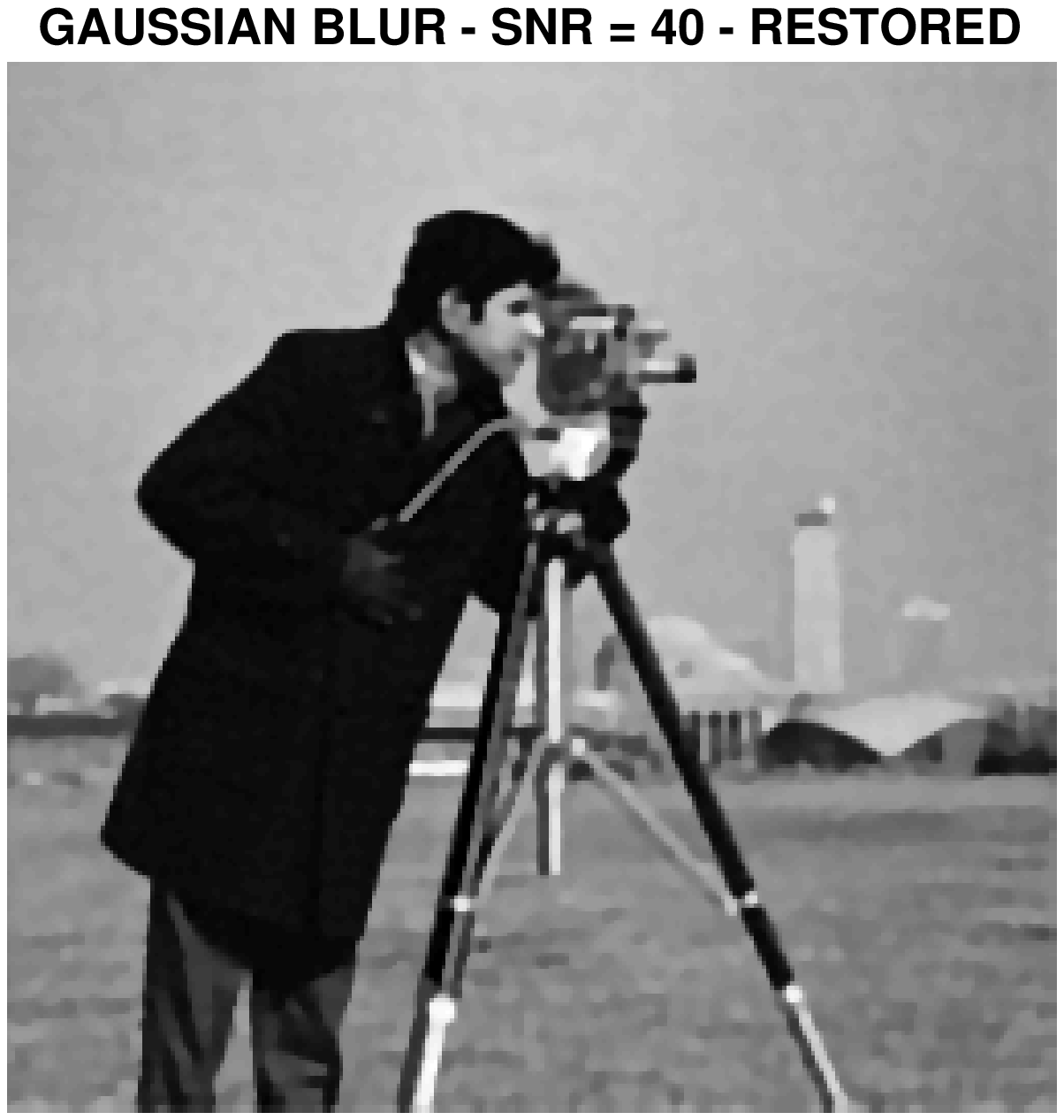}
		\end{tabular}
		\vspace*{-6mm}
		\caption{Cameraman: images corrupted by Gaussian blur and Poisson noise (left) and images
		 restored by ACQUIRE (right).  
		 \label{fig:cameraman}}
	\end{center}
\end{figure}

\begin{figure}[b!]
	\vspace*{2mm}
	\begin{center}
		\hspace*{-.55cm}
		\begin{tabular}{ccc}
			\includegraphics[width=.47\textwidth]{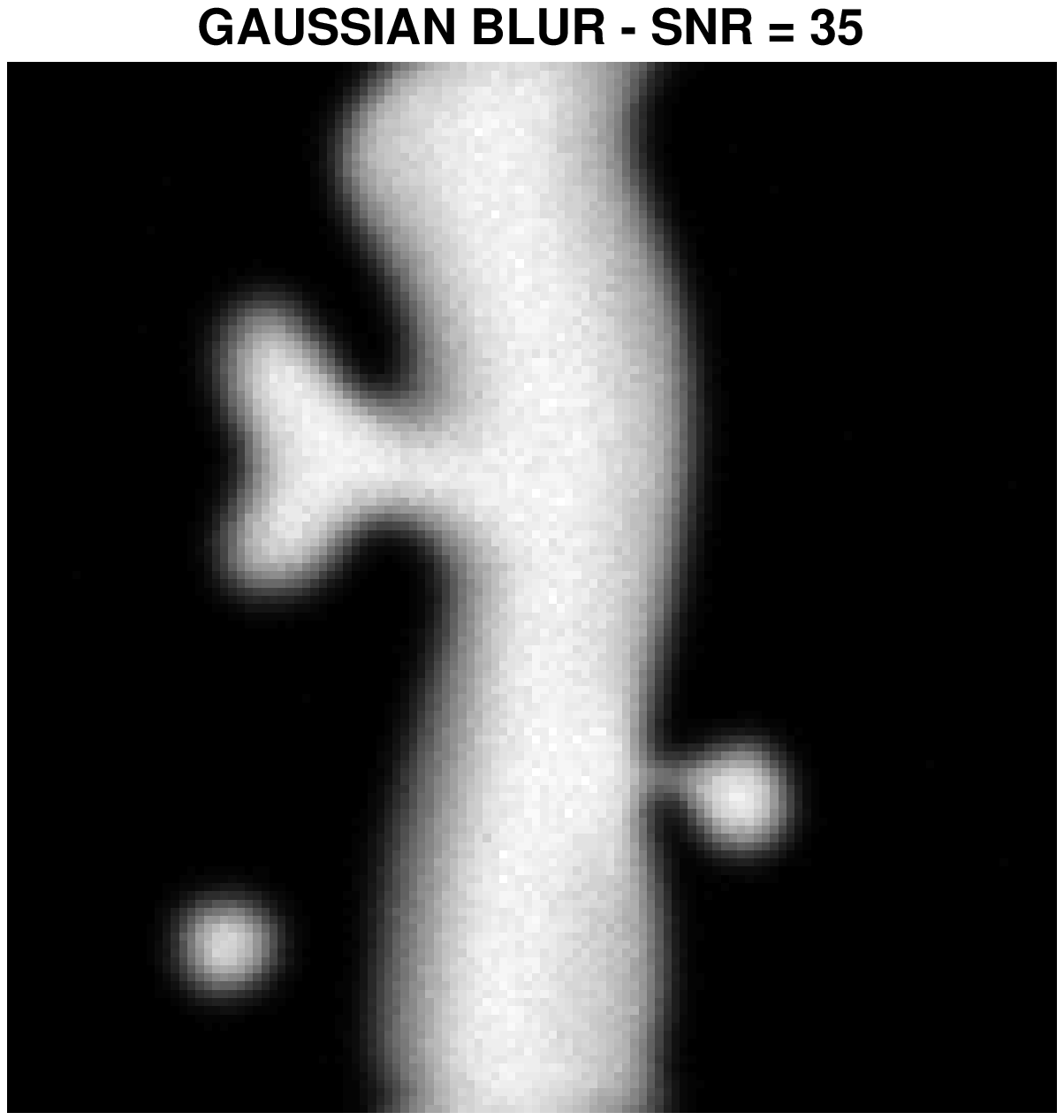}      &
			\hspace*{-1.5cm}
			\includegraphics[width=.47\textwidth]{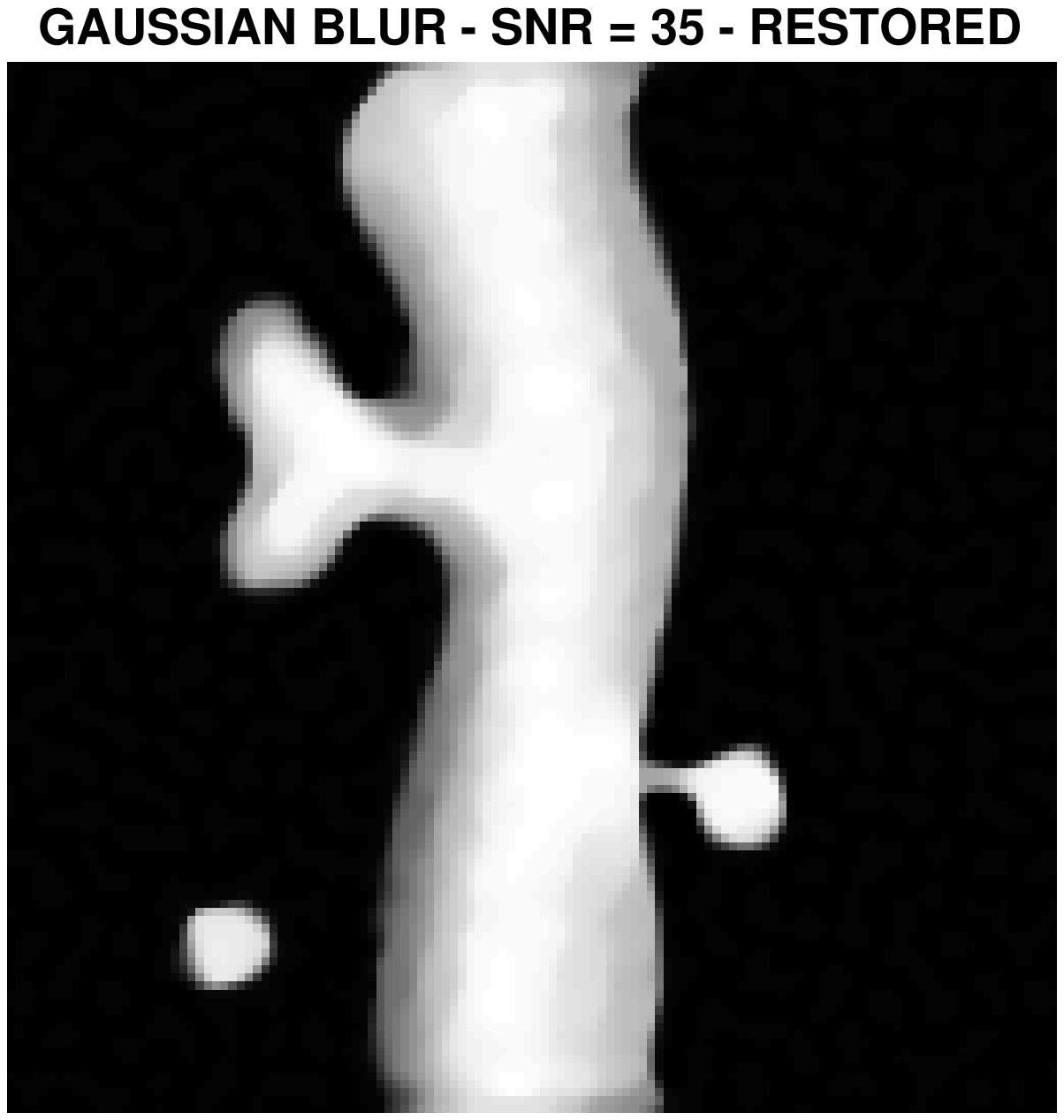}  \\[-4mm]
			\includegraphics[width=.47\textwidth]{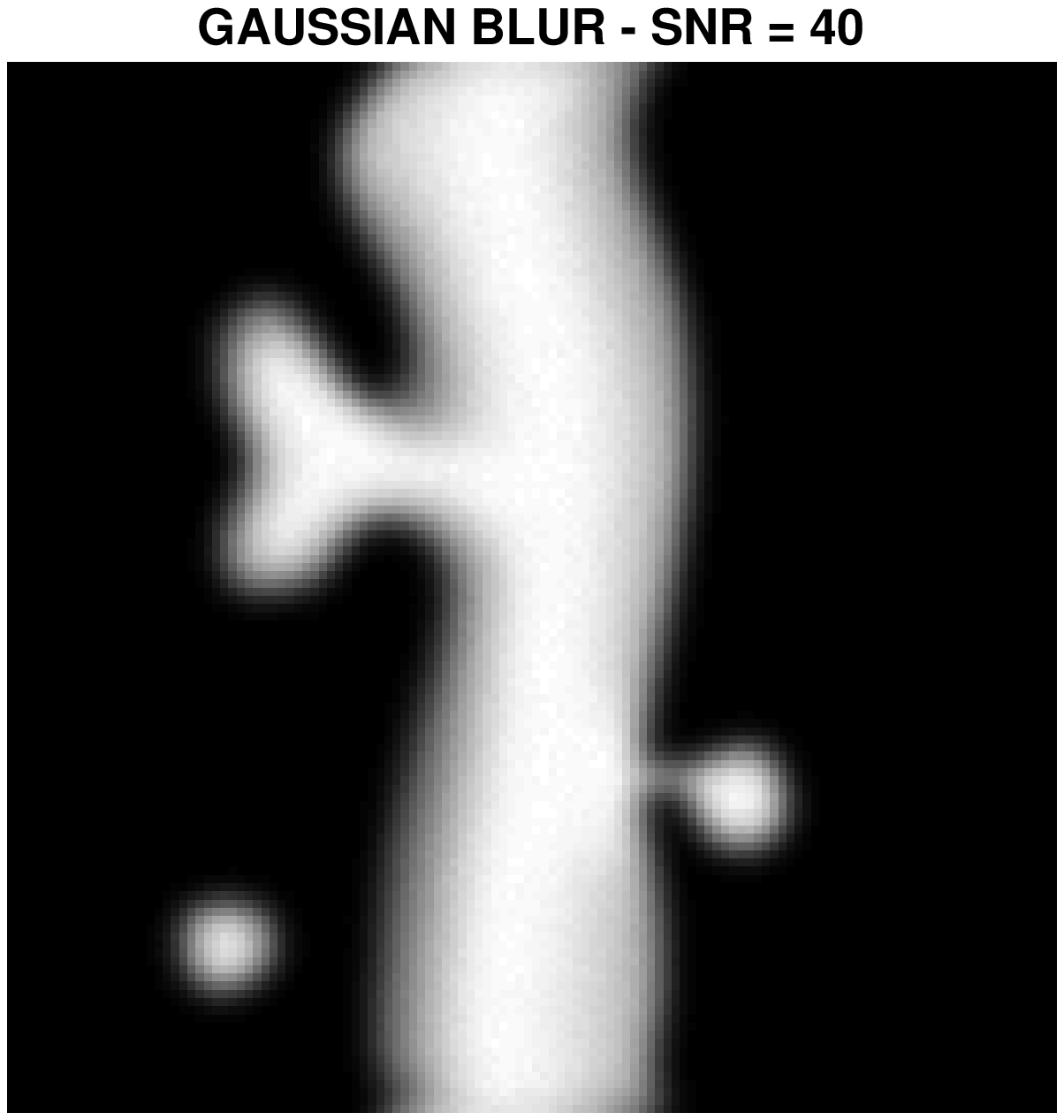}       &
			\hspace*{-1.5cm}
			\includegraphics[width=.47\textwidth]{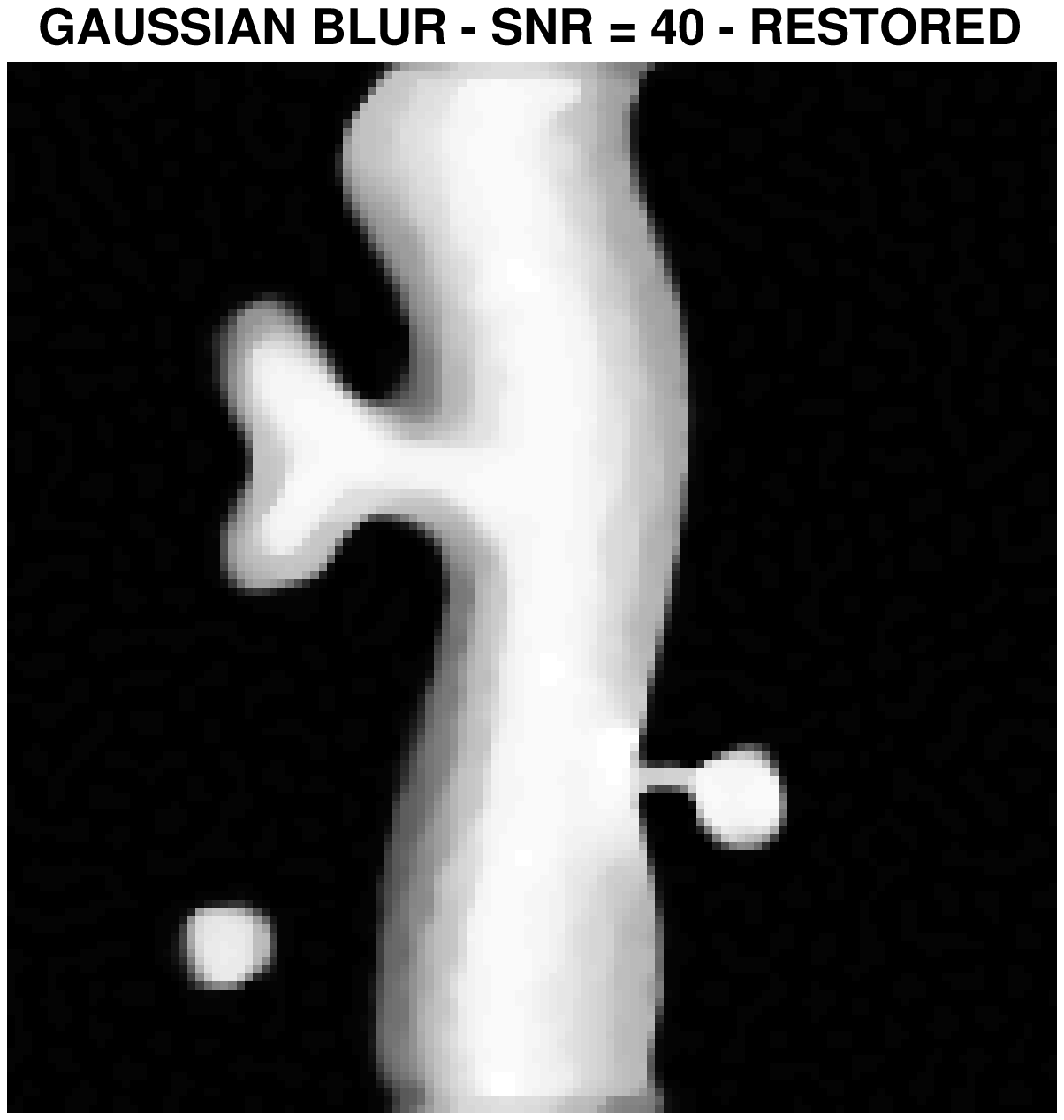}
		\end{tabular}
		\vspace*{-6mm}
		\caption{Micro: images corrupted by Gaussian blur and Poisson noise (left) and images
		 restored by ACQUIRE (right).  
		 \label{fig:micro}}
	\end{center}
\end{figure}

\clearpage

\begin{figure}[t!]
	\vspace*{2mm}
	\begin{center}
		\hspace*{-.55cm}
		\begin{tabular}{ccc}
			\includegraphics[width=.47\textwidth]{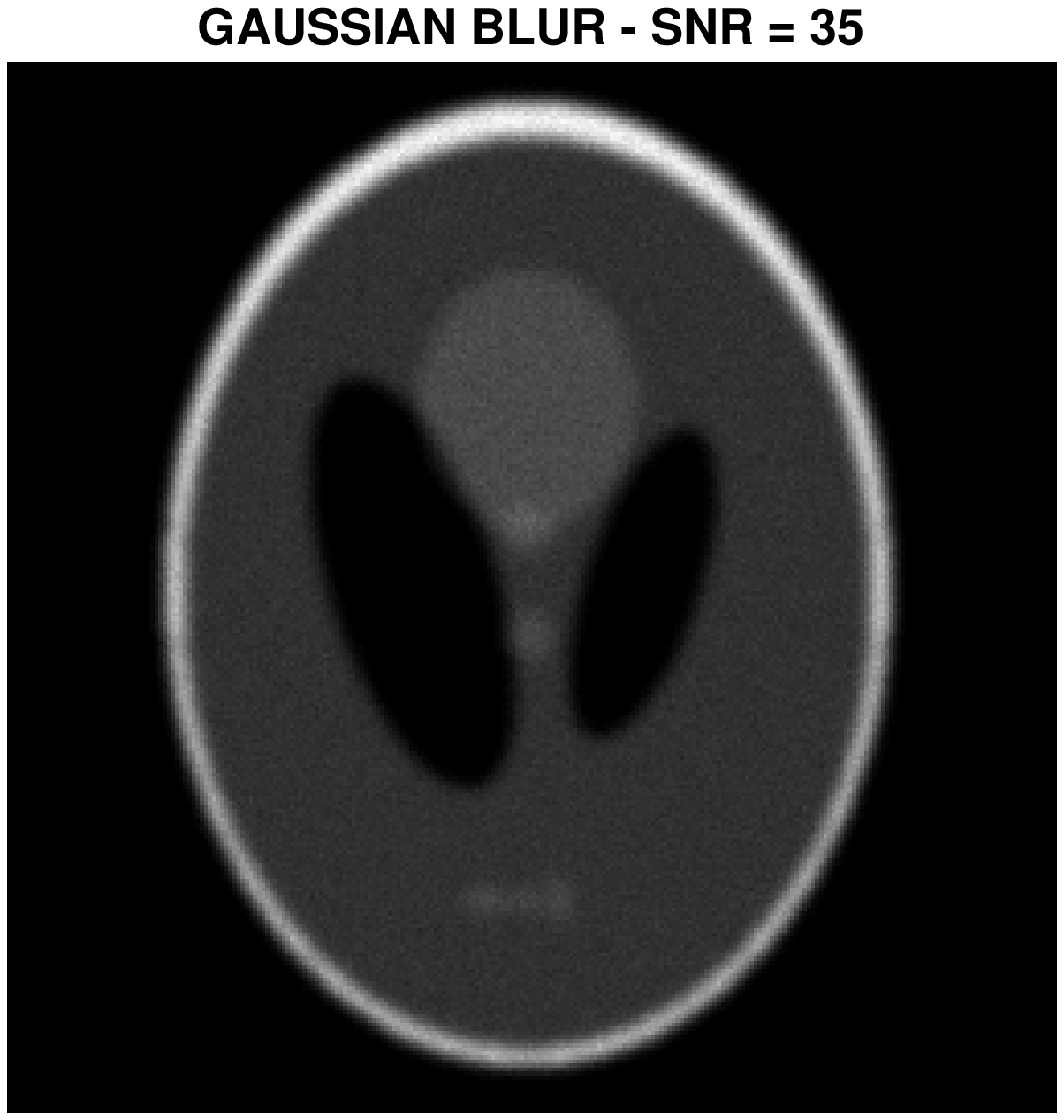}      &
			\hspace*{-1.5cm}
			\includegraphics[width=.47\textwidth]{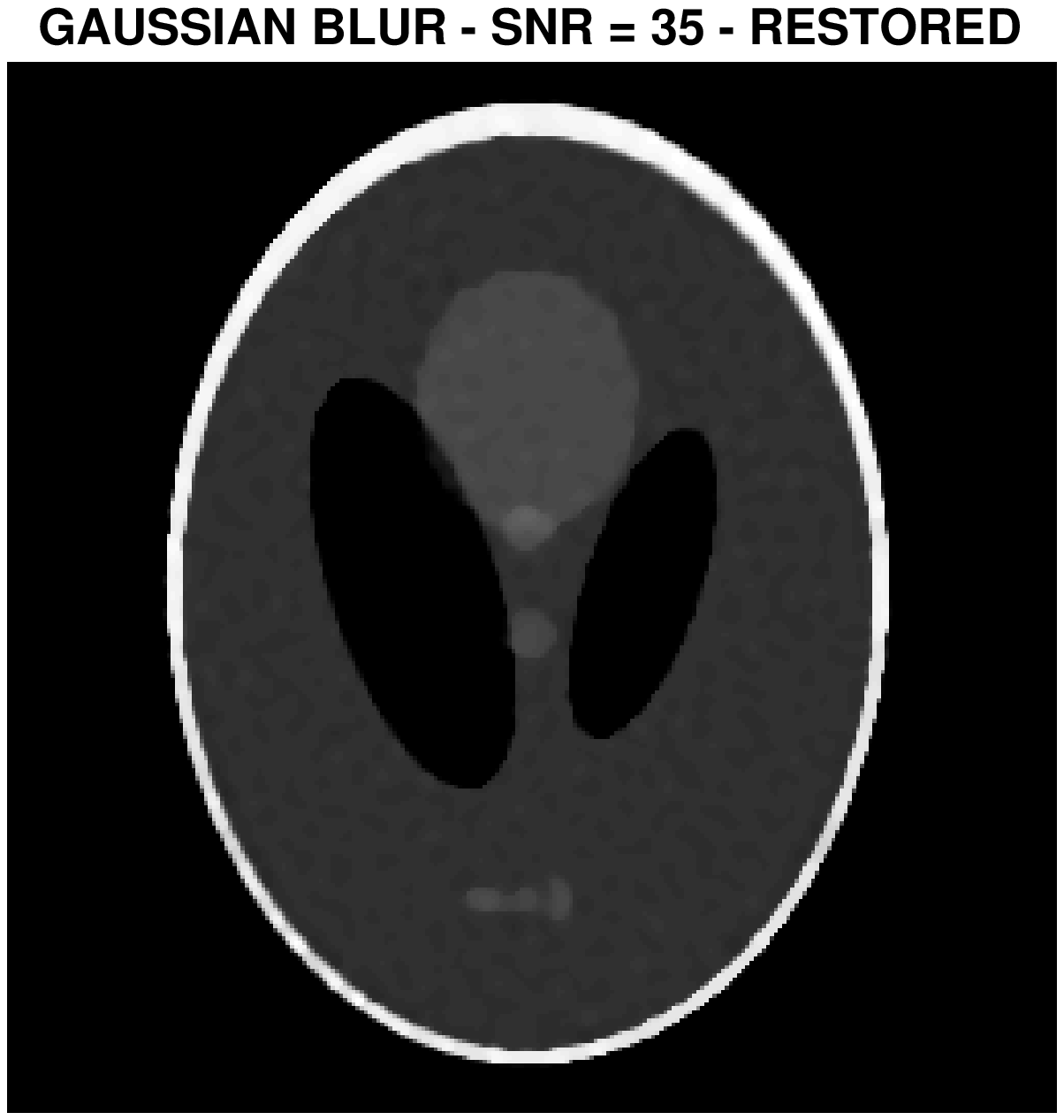}  \\[-4mm]
			\includegraphics[width=.47\textwidth]{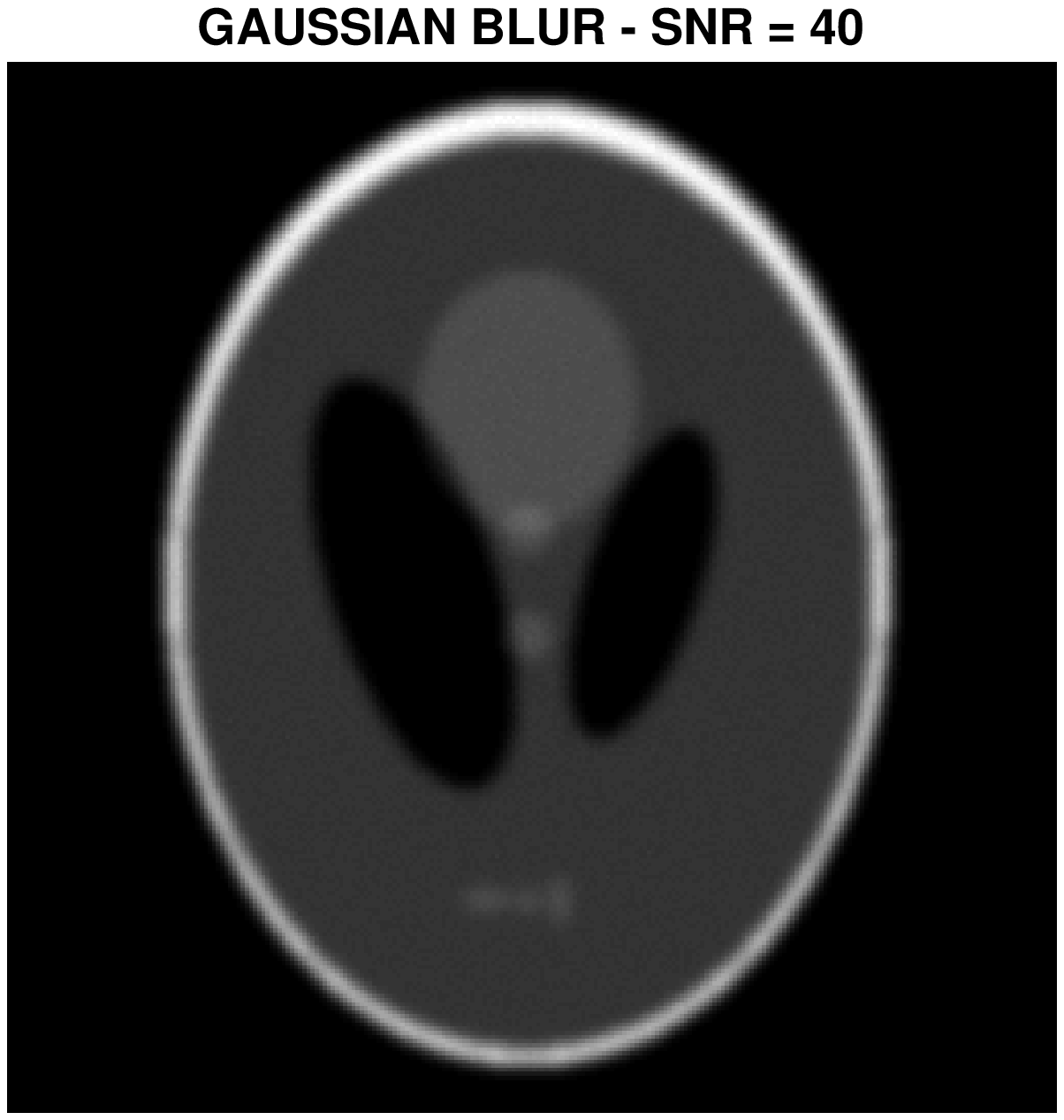}       &
			\hspace*{-1.5cm}
			\includegraphics[width=.47\textwidth]{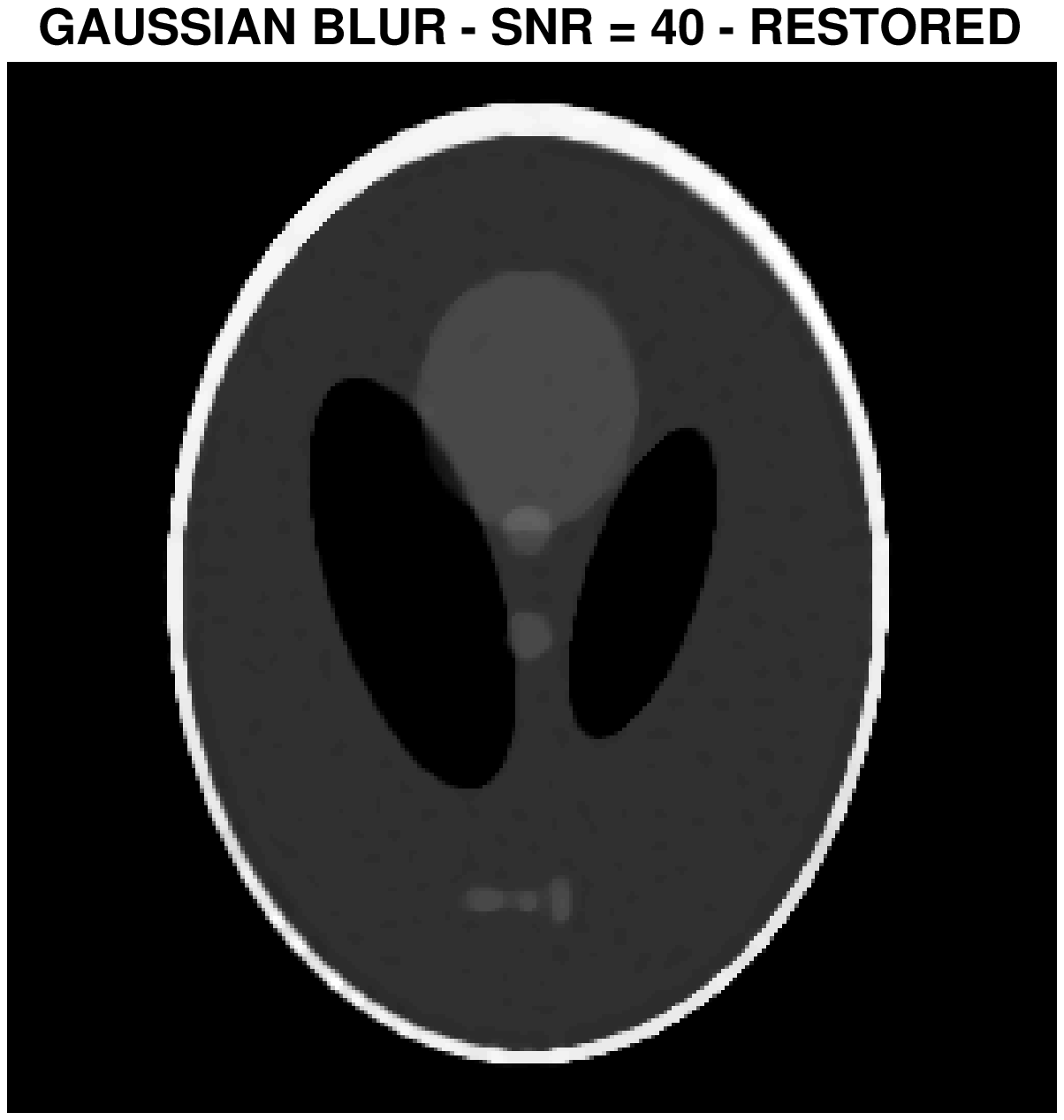}
		\end{tabular}
		\vspace*{-6mm}
		\caption{Phantom: images corrupted by Gaussian blur and Poisson noise (left) and images
		 restored by ACQUIRE (right). 
		 \label{fig:phantom}}
	\end{center}
\end{figure}

\begin{figure}[b!]
	\vspace*{2mm}
	\begin{center}
		\hspace*{-.55cm}
		\begin{tabular}{ccc}
			\includegraphics[width=.47\textwidth]{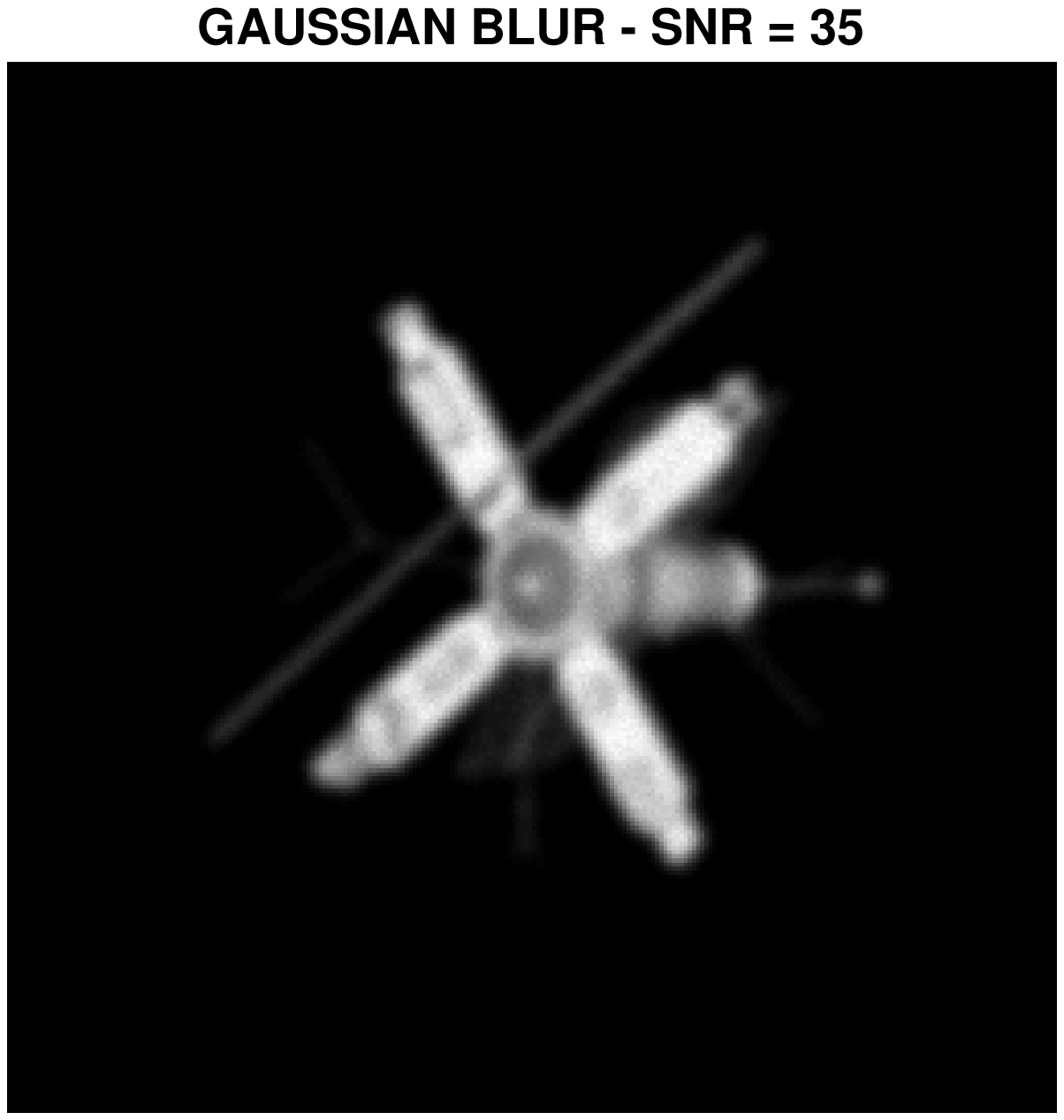}      &
			\hspace*{-1.5cm}
			\includegraphics[width=.47\textwidth]{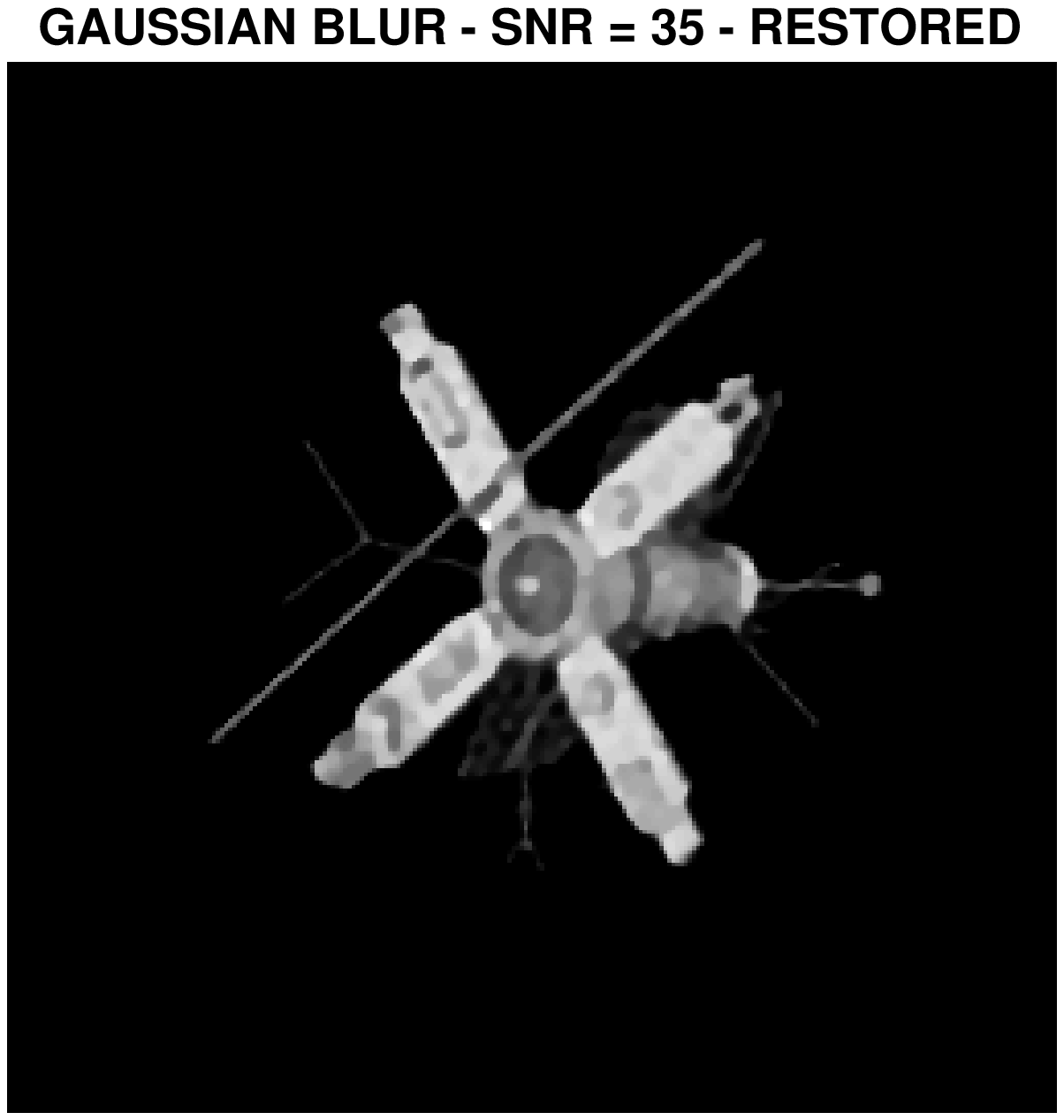}  \\[-4mm]
			\includegraphics[width=.47\textwidth]{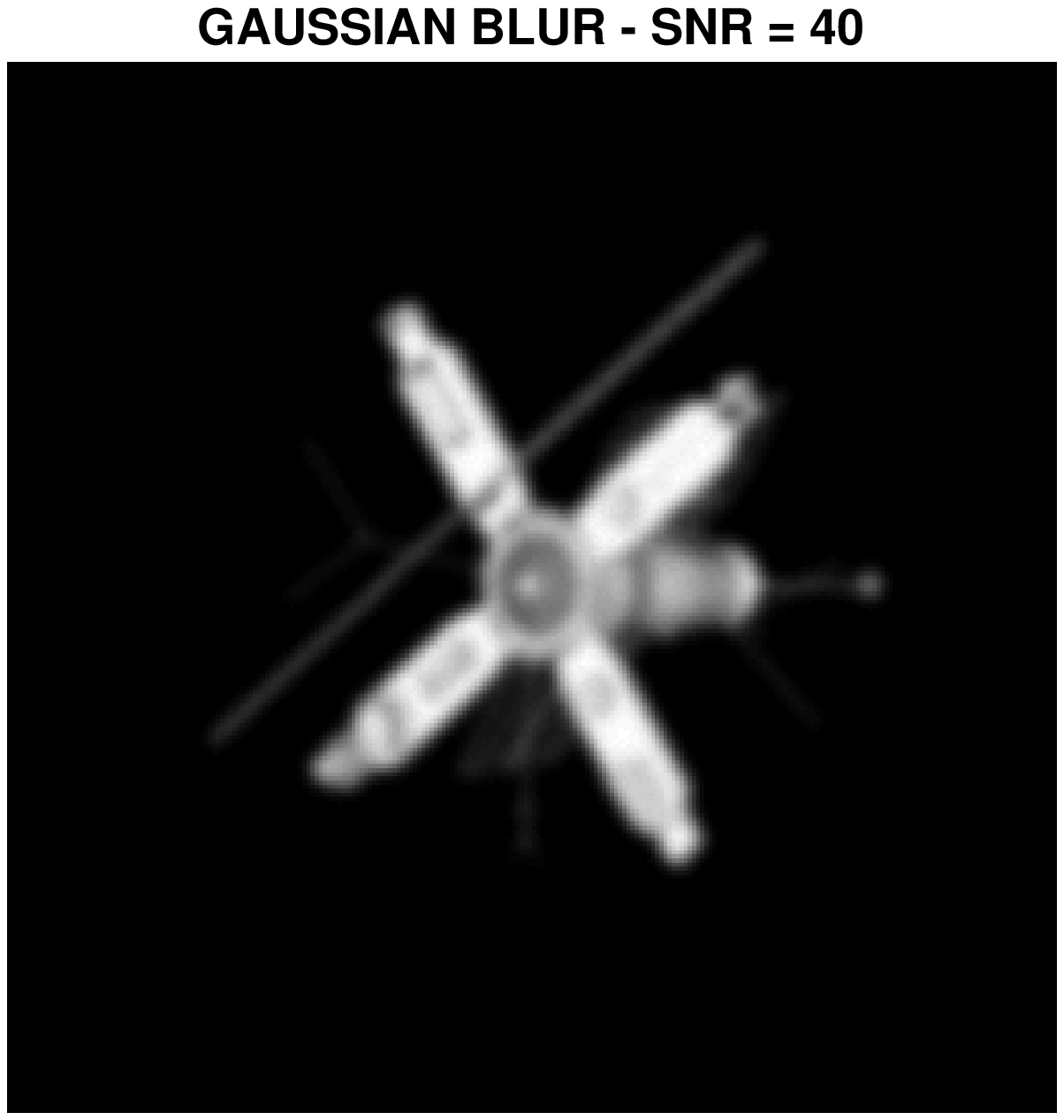}       &
			\hspace*{-1.5cm}
			\includegraphics[width=.47\textwidth]{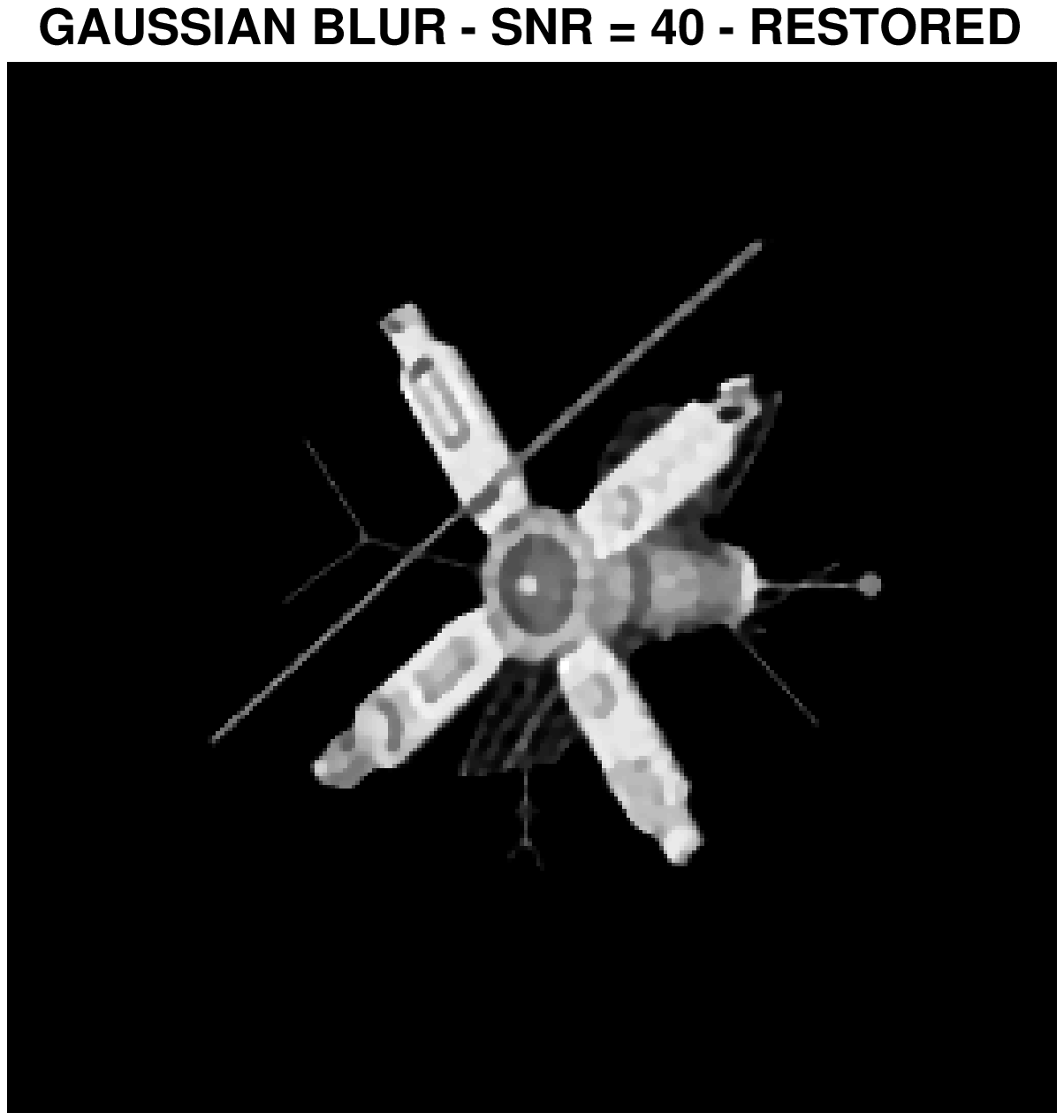}
		\end{tabular}
		\vspace*{-6mm}
		\caption{Satellite: images corrupted by Gaussian blur and Poisson noise (left) and images
		 restored by ACQUIRE (right). 
		 \label{fig:satellite}}
	\end{center}
\end{figure}

\clearpage

\begin{figure}[p!]
	\begin{center}
		\begin{tabular}{cc}
			\includegraphics[width=.44\textwidth]{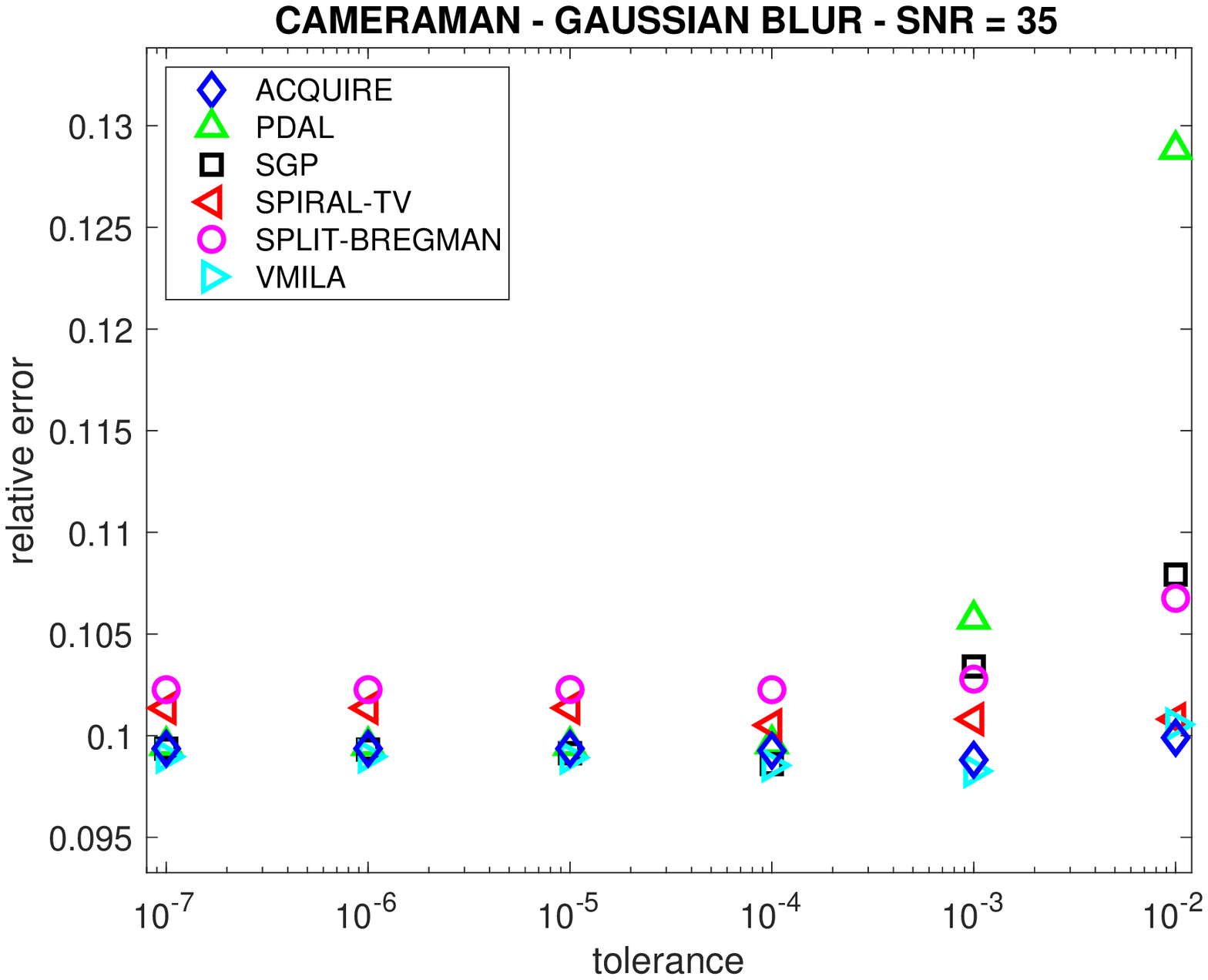}     &
			\includegraphics[width=.44\textwidth]{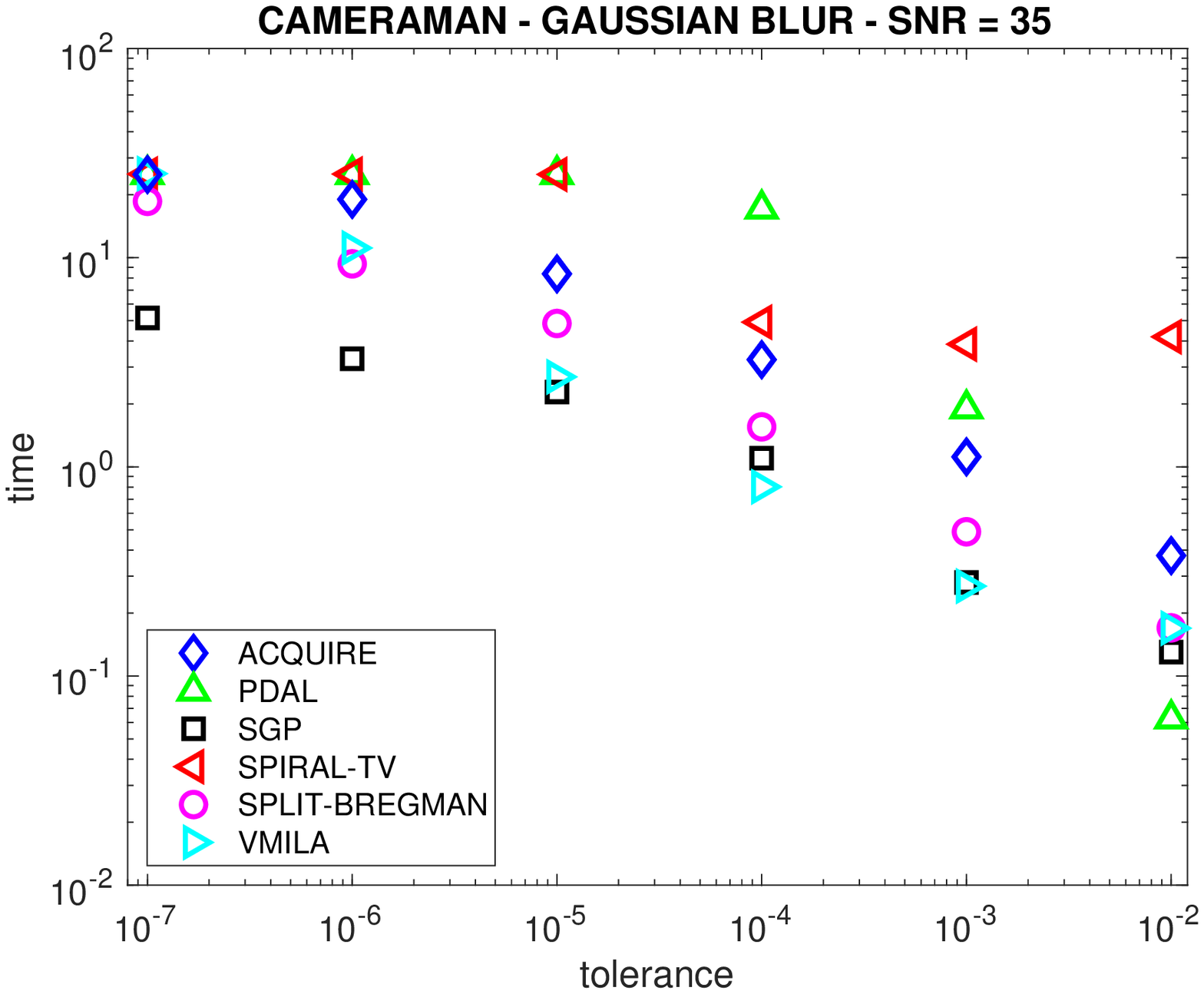}  \\[-1mm]
			\includegraphics[width=.44\textwidth]{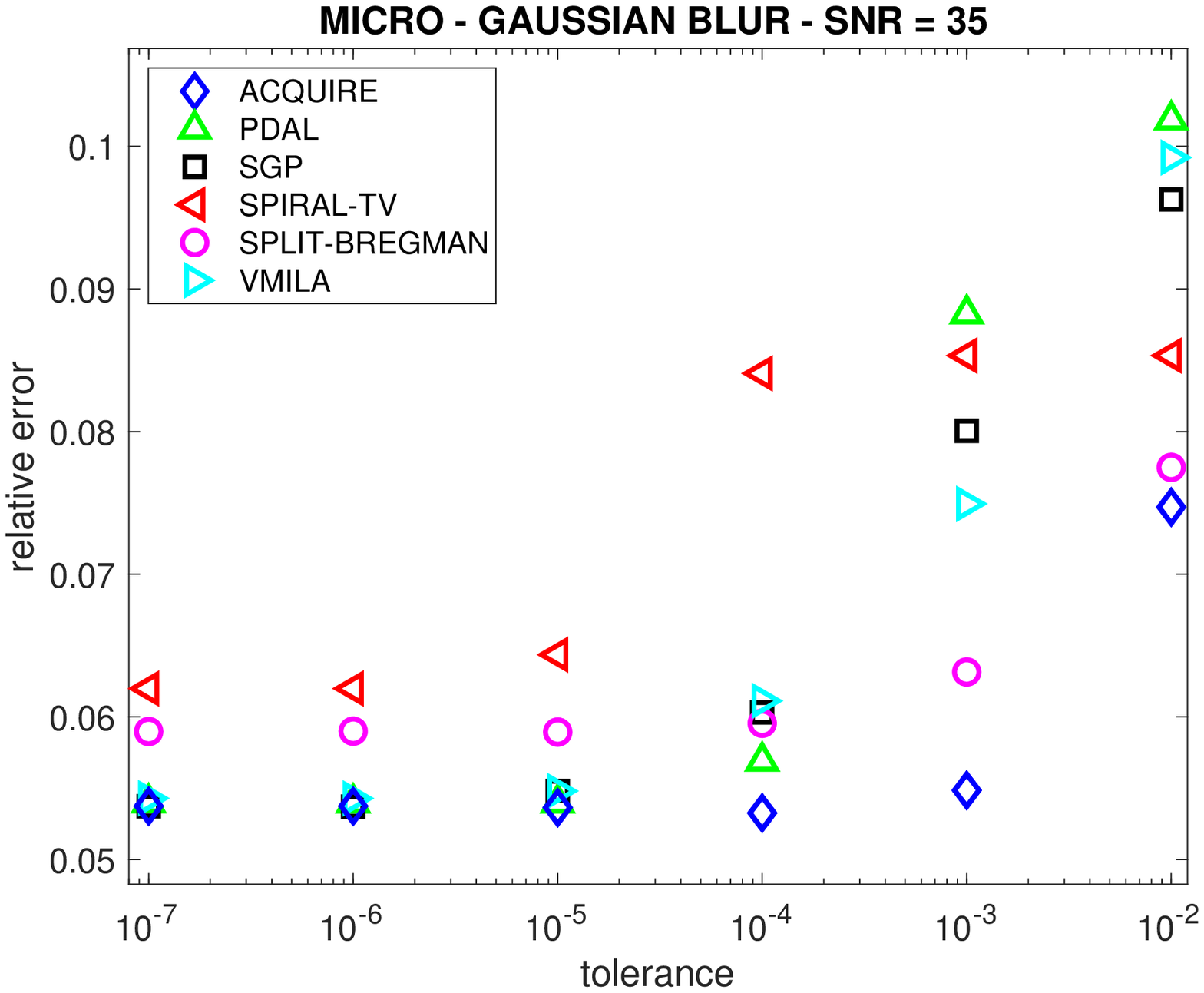}               &
			\includegraphics[width=.44\textwidth]{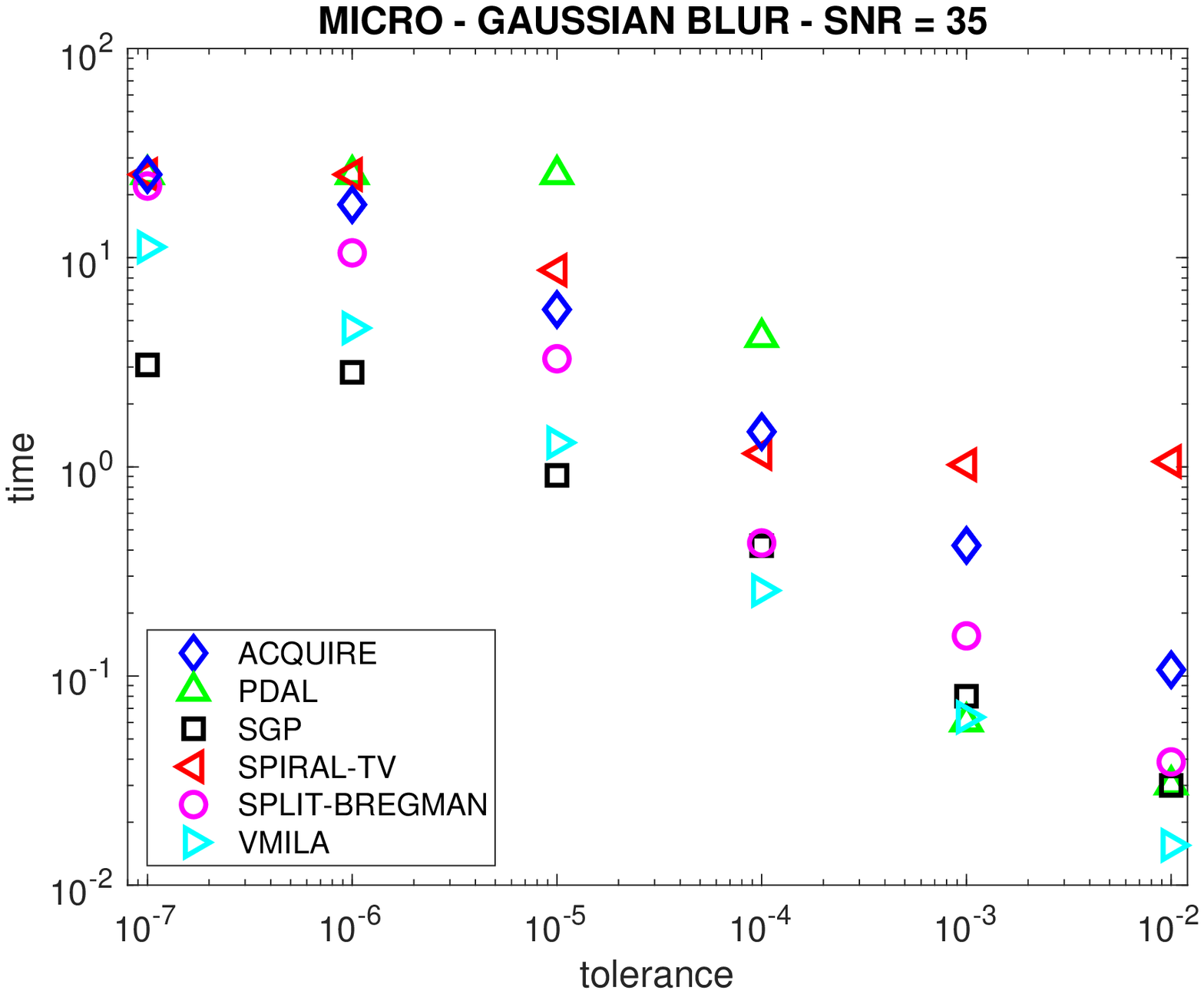}             \\[-1mm]
			\includegraphics[width=.44\textwidth]{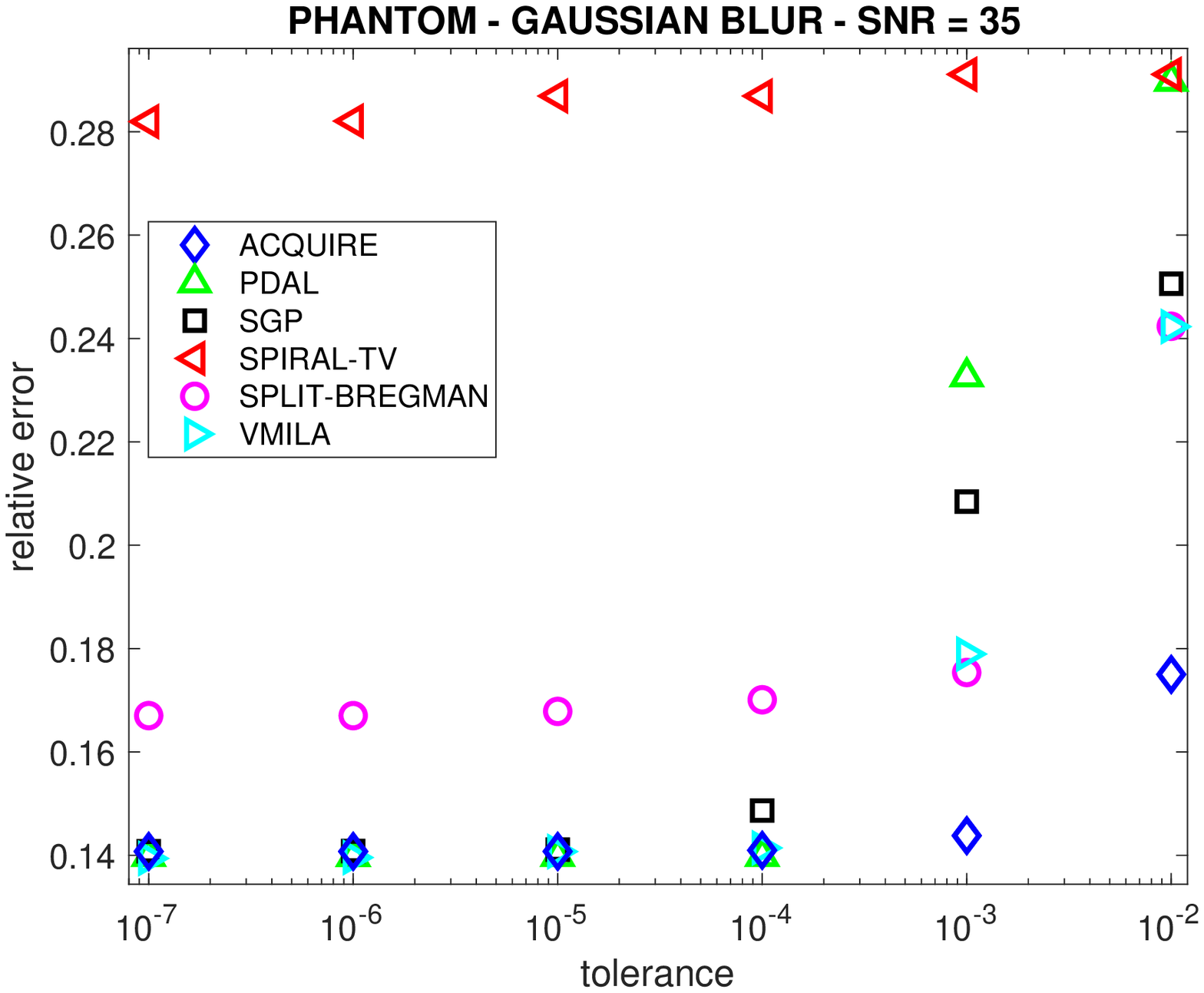}          &
			\includegraphics[width=.44\textwidth]{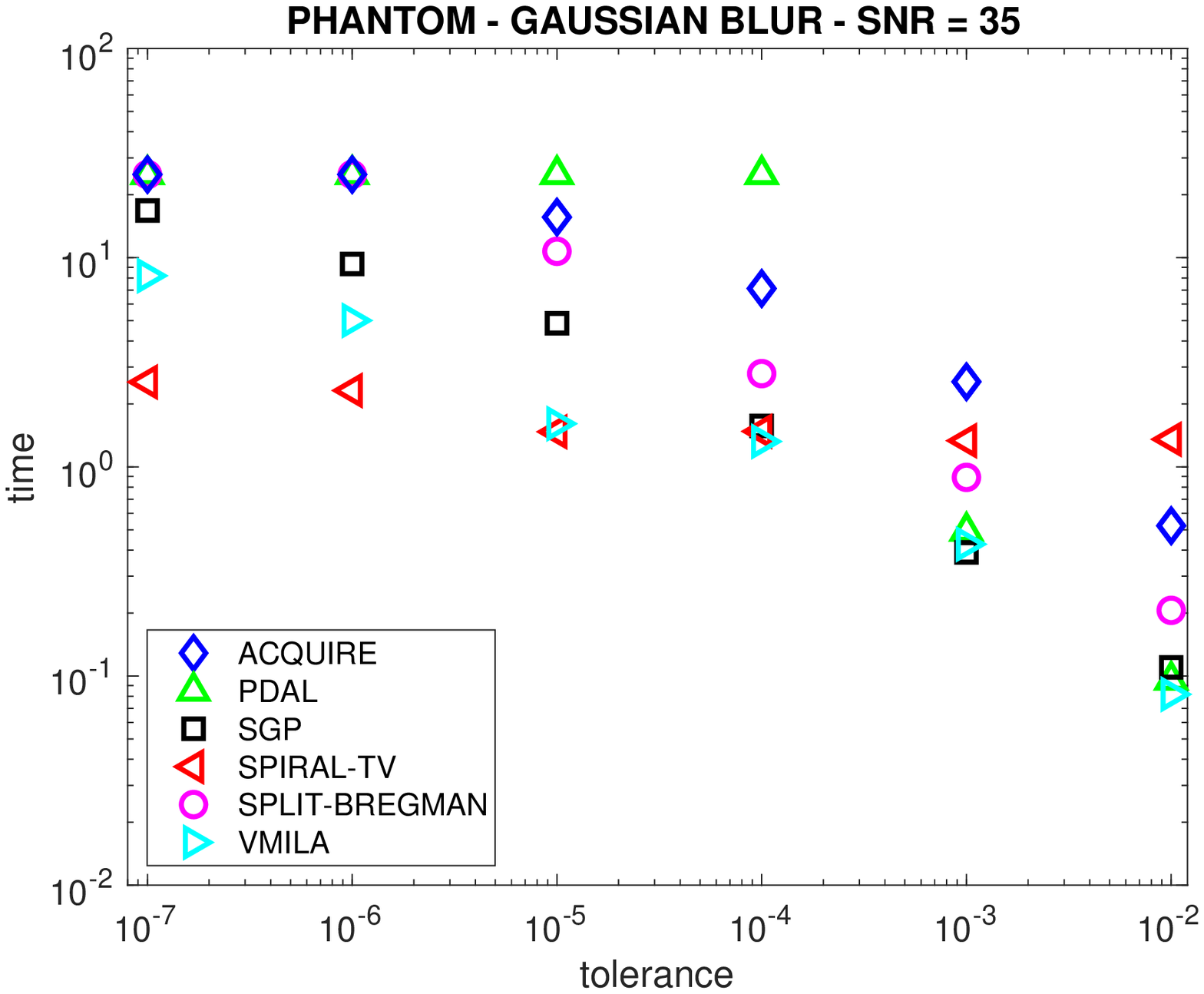}        \\[-1mm]
			\includegraphics[width=.44\textwidth]{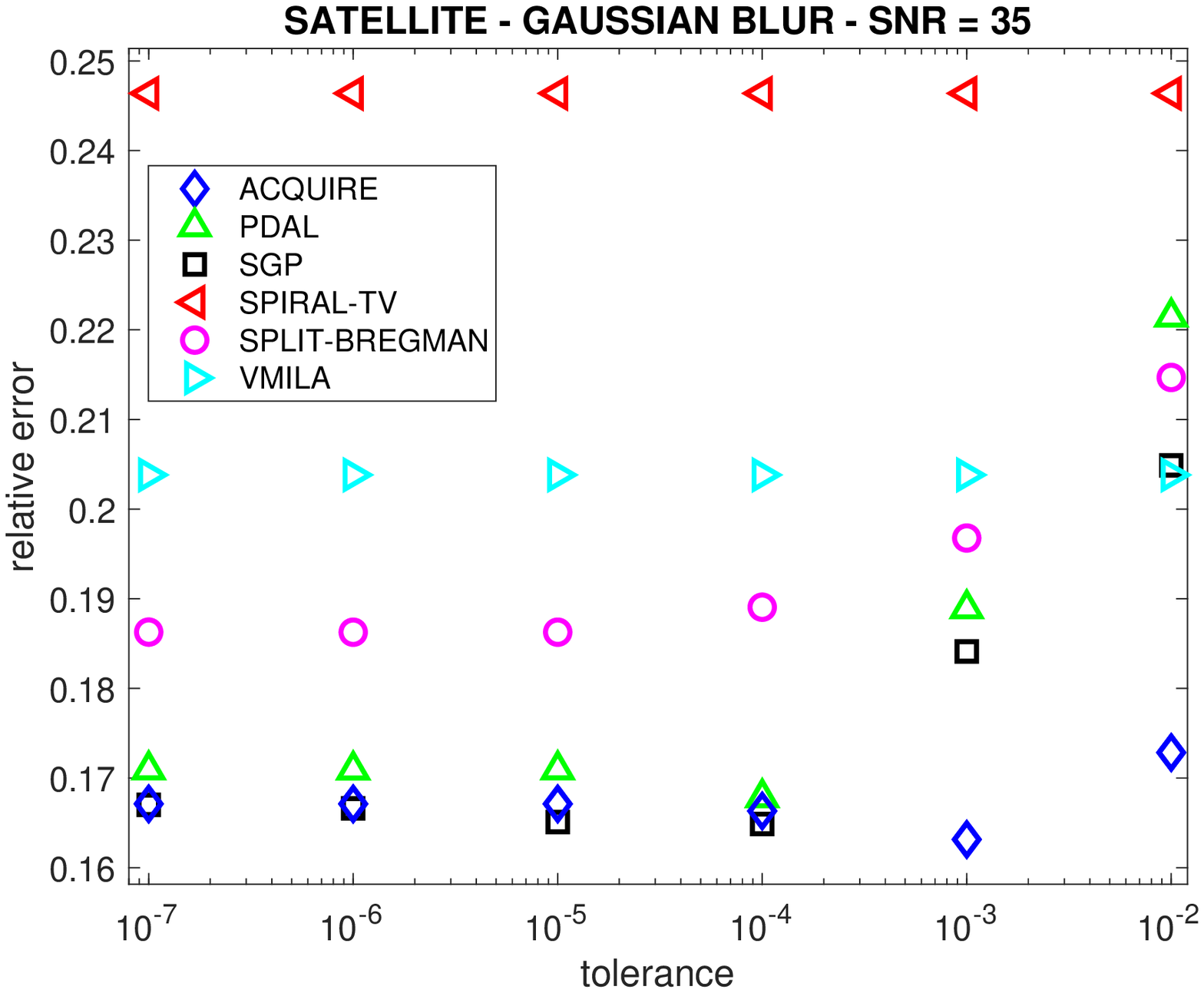}            &
			\includegraphics[width=.44\textwidth]{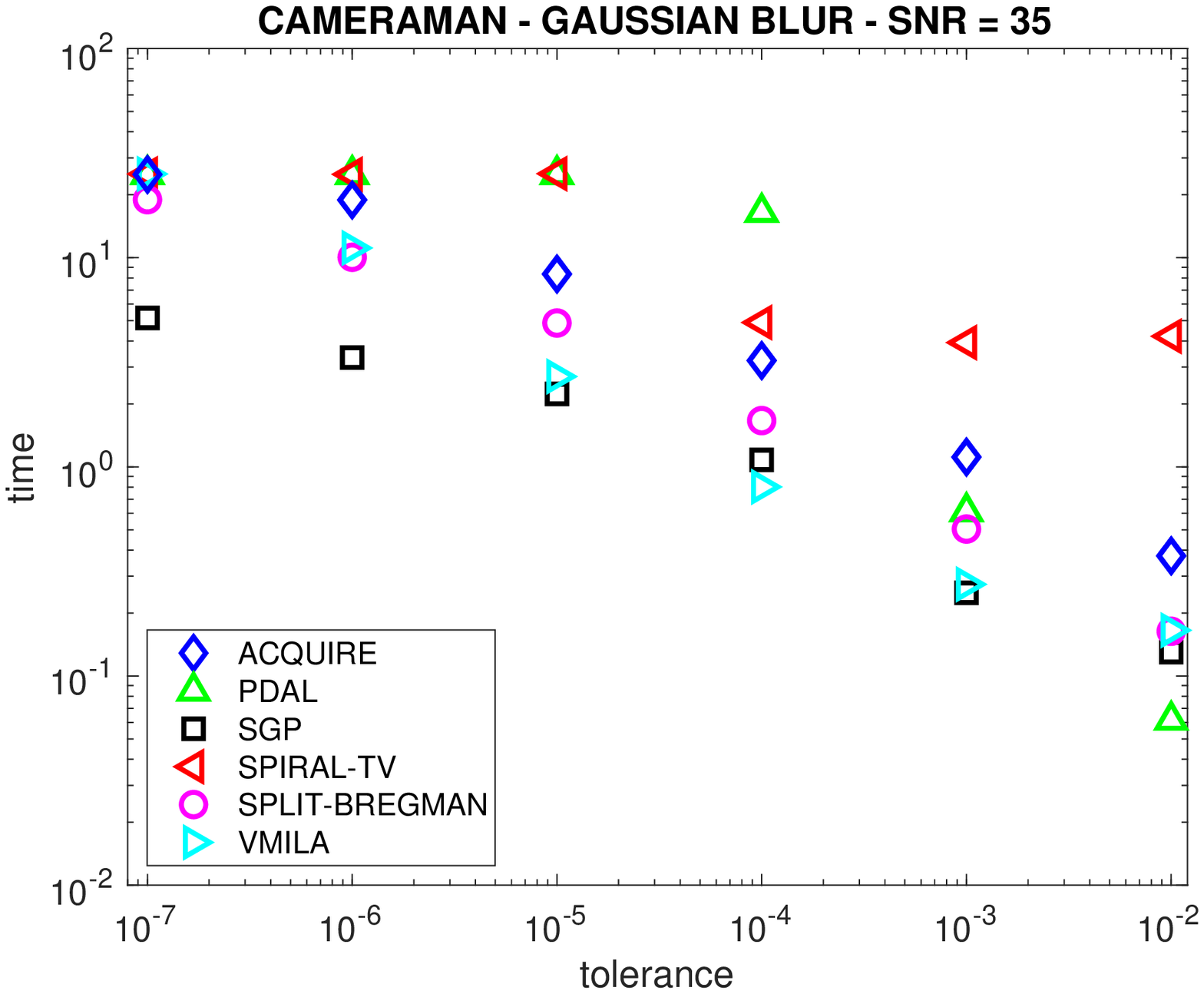}
		\end{tabular}
		\vspace*{-2mm}
		\caption{Test set T1, SNR = 35: relative error (left) and execution time (right) versus tolerance,
		for all the methods.\label{fig:compar35}}
	\end{center}
\end{figure}

\clearpage

\begin{figure}[p!]
	\begin{center}
		\begin{tabular}{cc}
			\includegraphics[width=.44\textwidth]{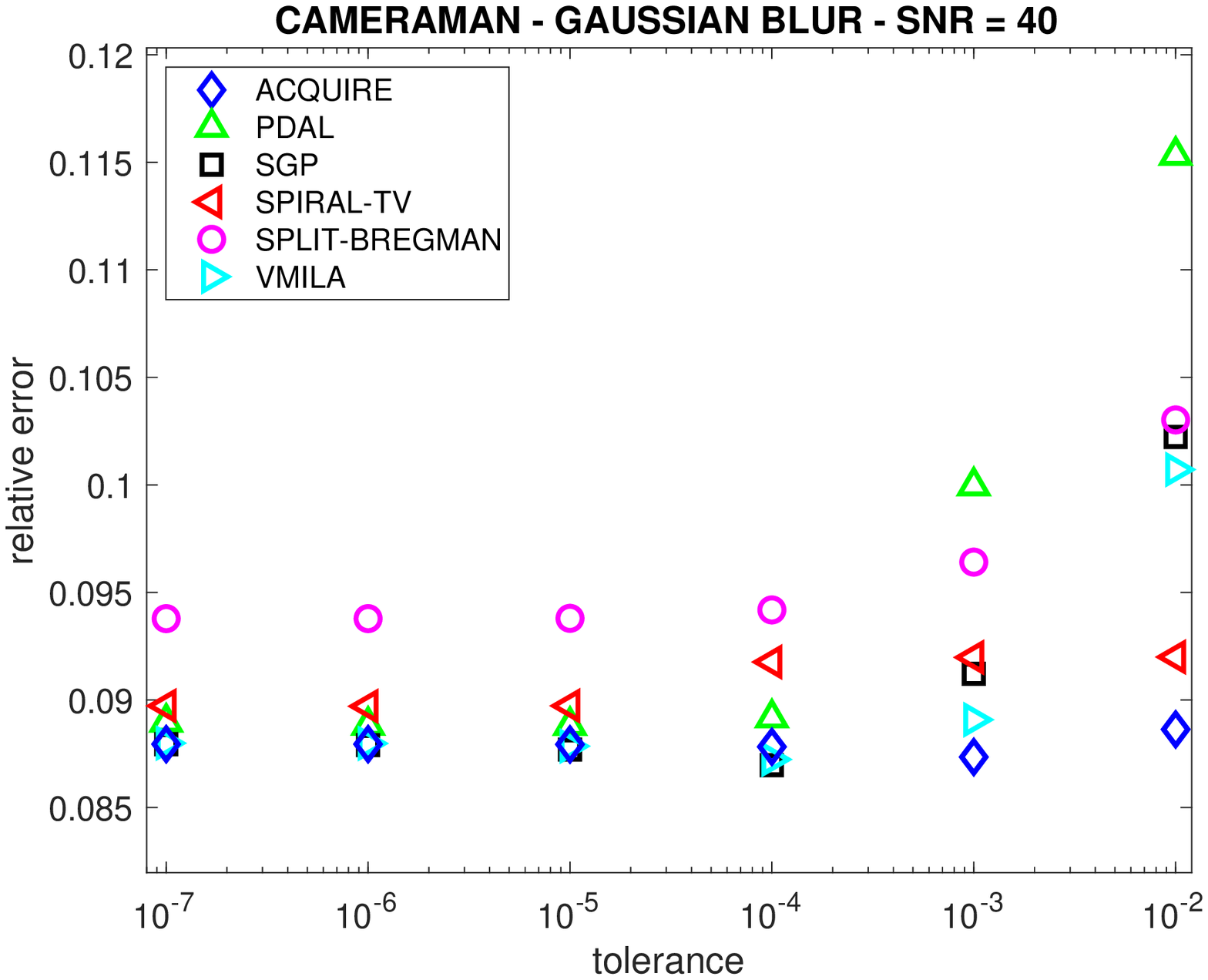}     &
			\includegraphics[width=.44\textwidth]{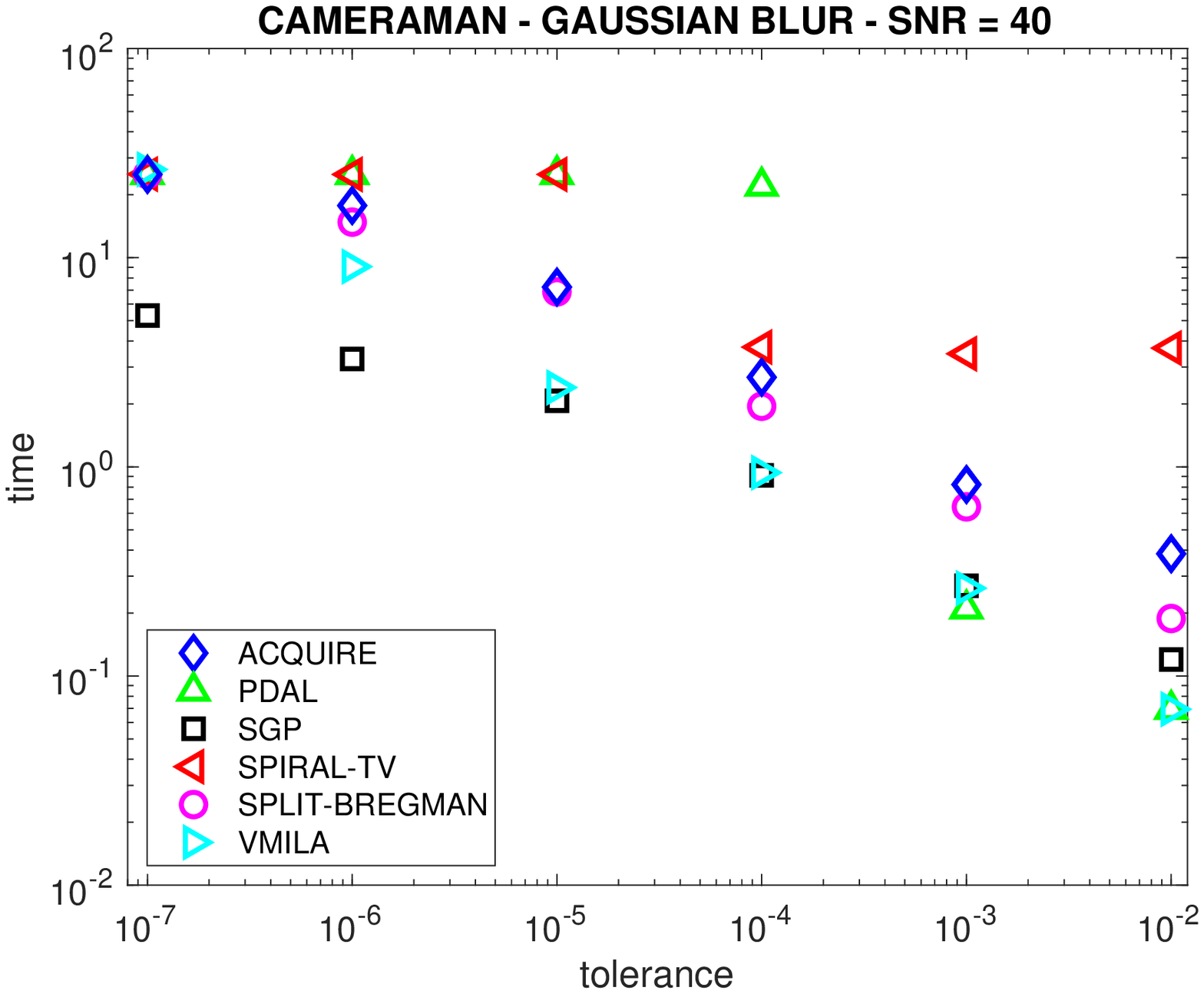}  \\[-1mm]
			\includegraphics[width=.44\textwidth]{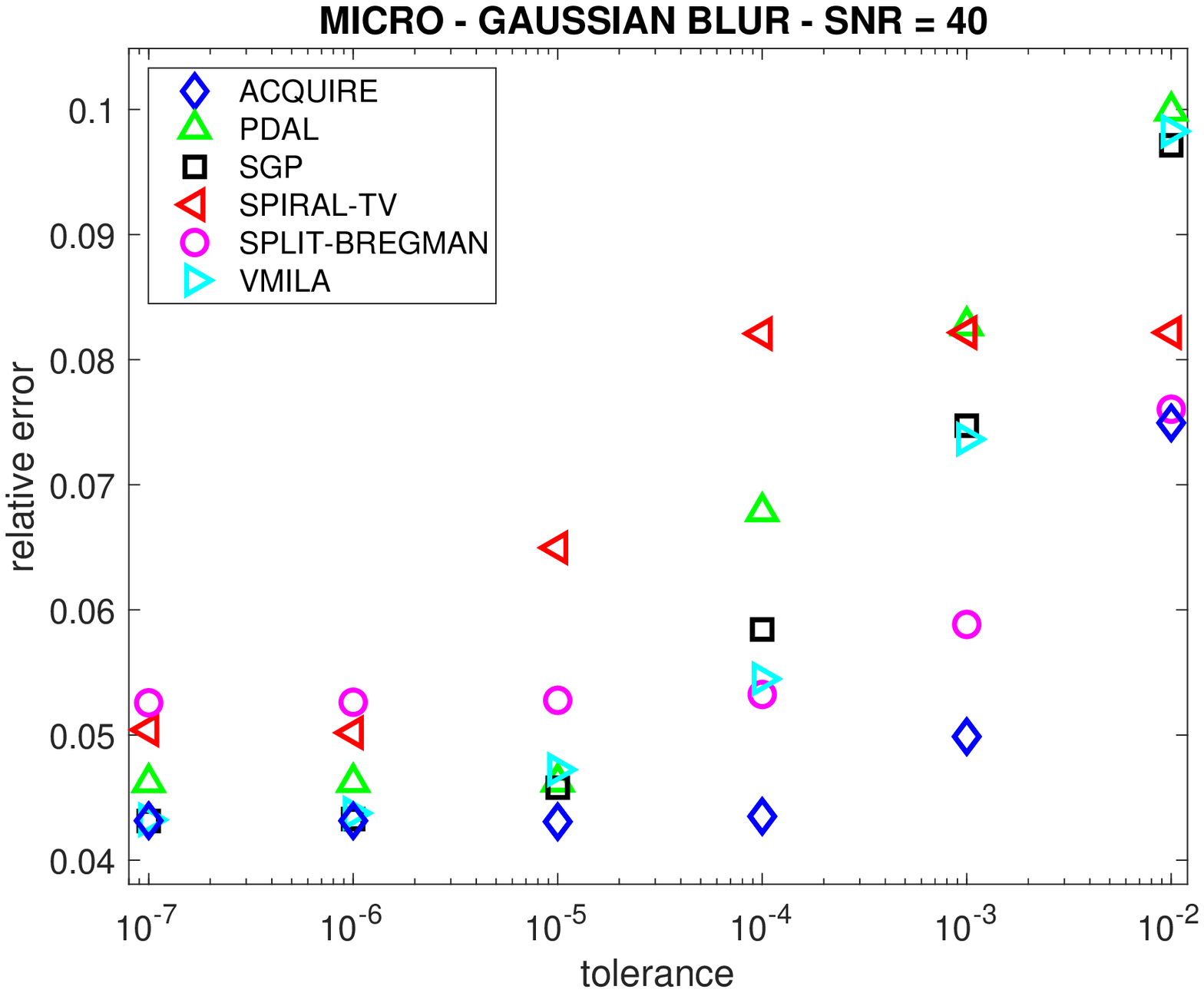}               &
			\includegraphics[width=.44\textwidth]{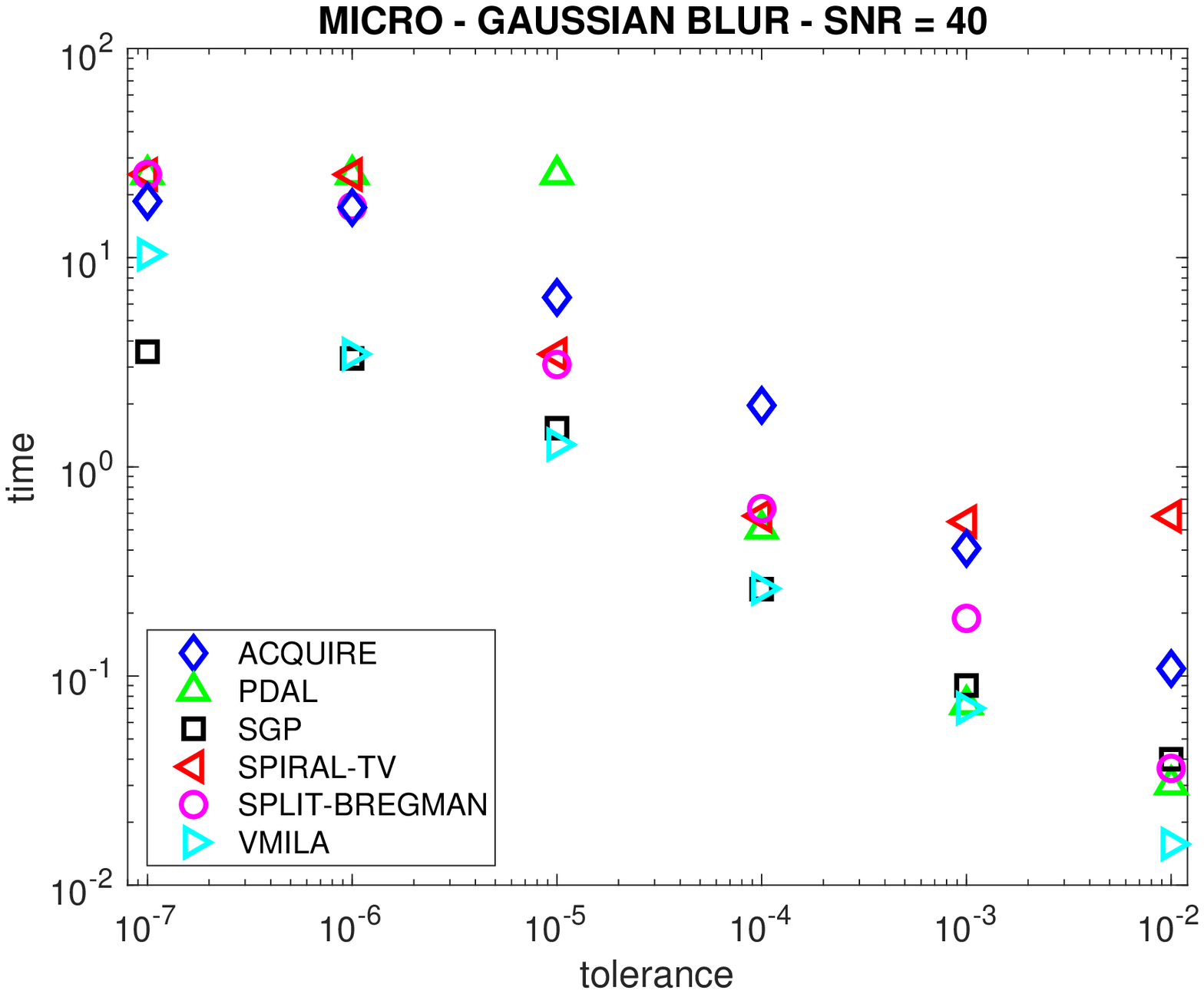}             \\[-1mm]
			\includegraphics[width=.44\textwidth]{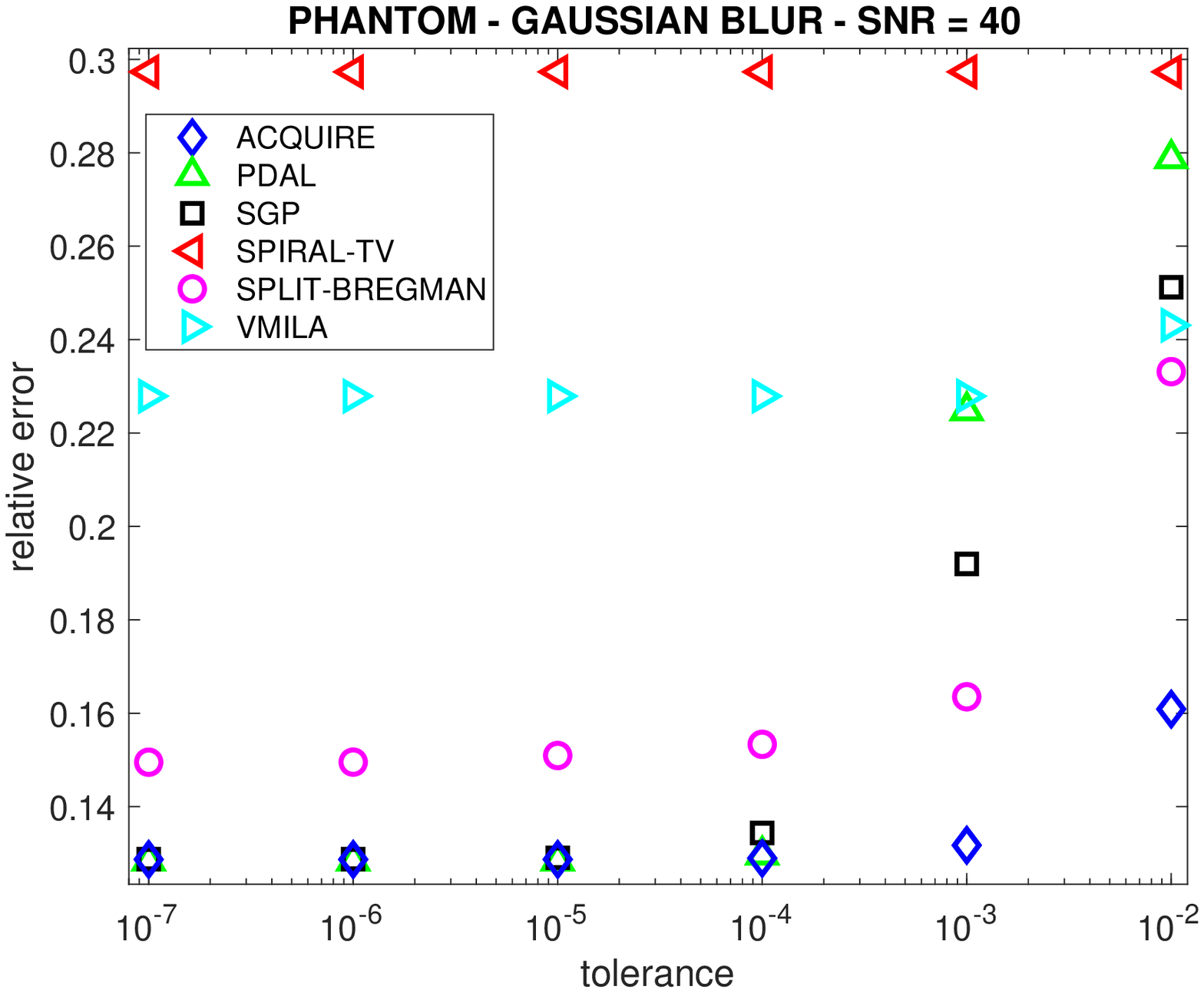}          &
			\includegraphics[width=.44\textwidth]{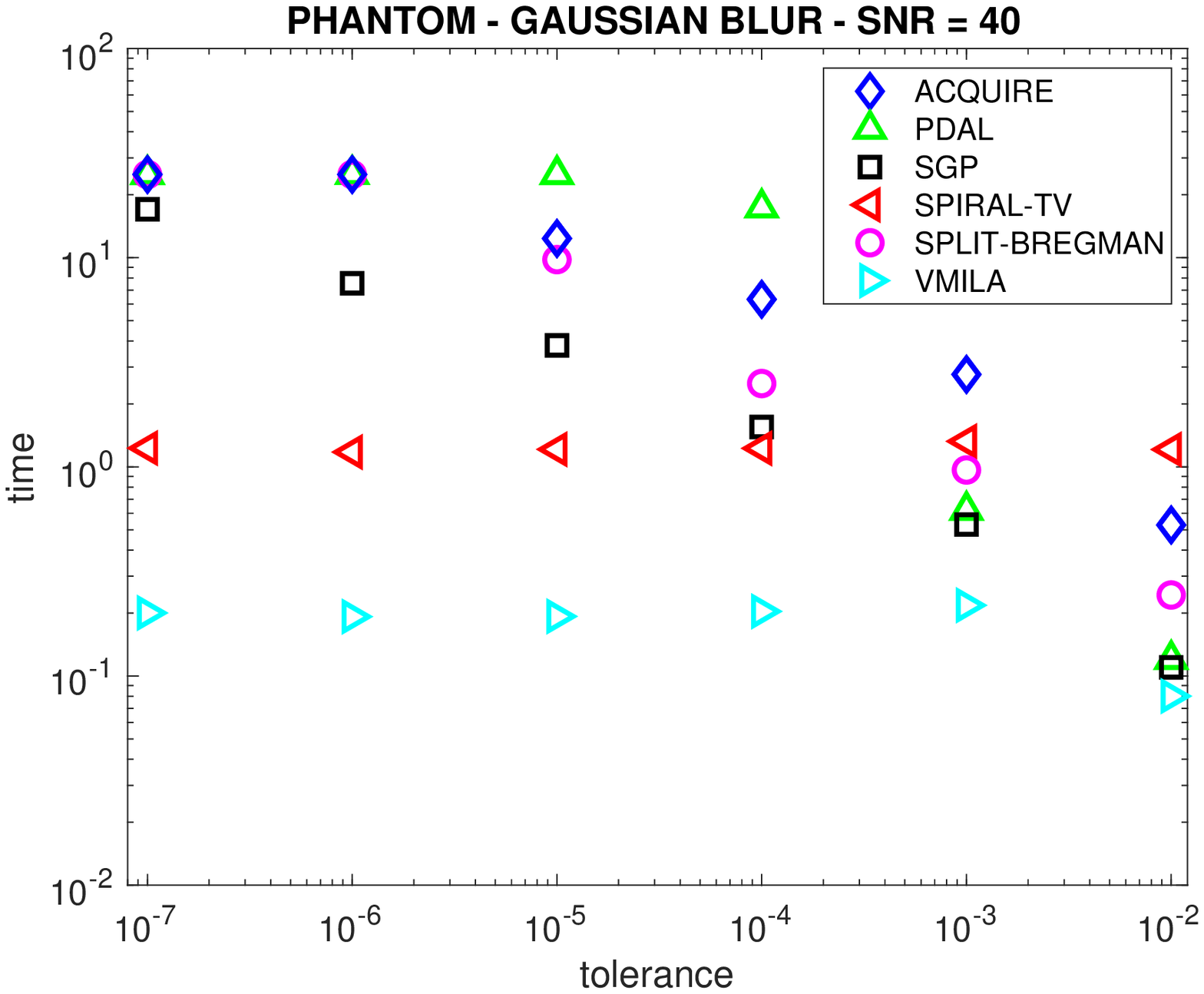}        \\[-1mm]
			\includegraphics[width=.44\textwidth]{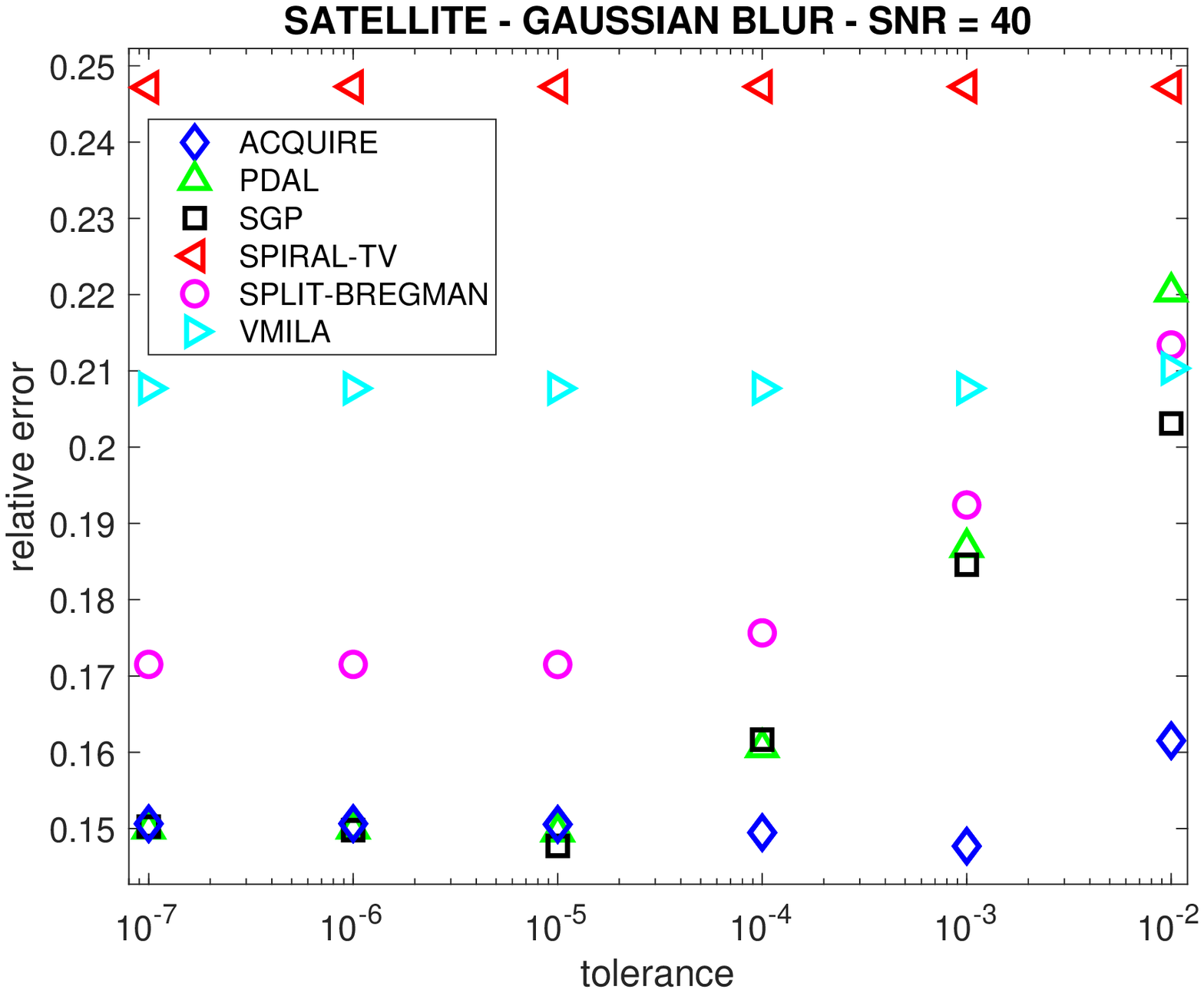}            &
			\includegraphics[width=.44\textwidth]{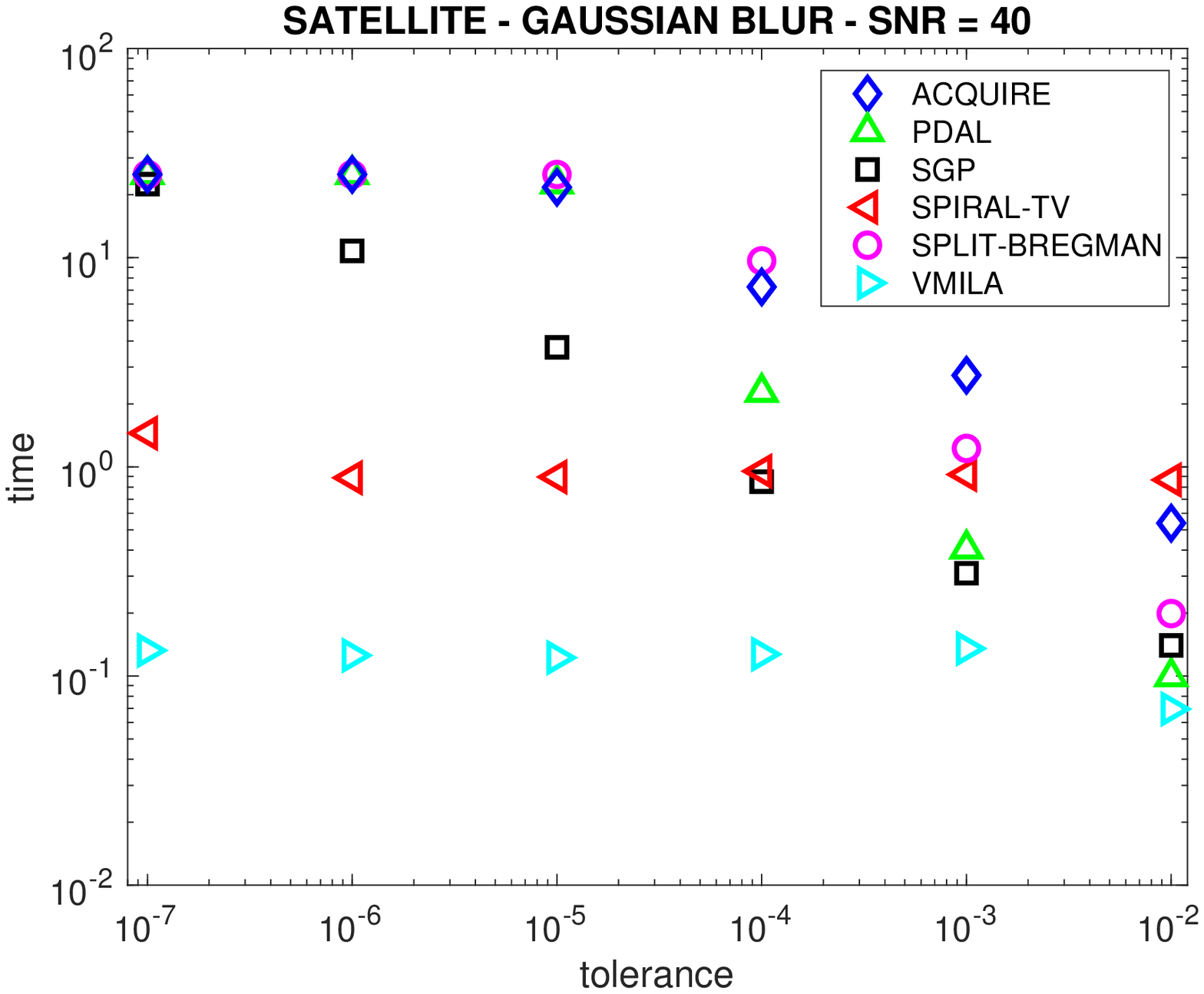}
		\end{tabular}
		\vspace*{-2mm}
		\caption{Test set T1, SNR = 40: relative error (left) and execution time (right) versus tolerance,
		for all the methods.\label{fig:compar40}}
	\end{center}
\end{figure}

\begin{figure}[t!]
	\vspace*{2mm}
	\begin{center}
		\hspace*{-.55cm}
		\begin{tabular}{ccc}
			\includegraphics[width=.47\textwidth]{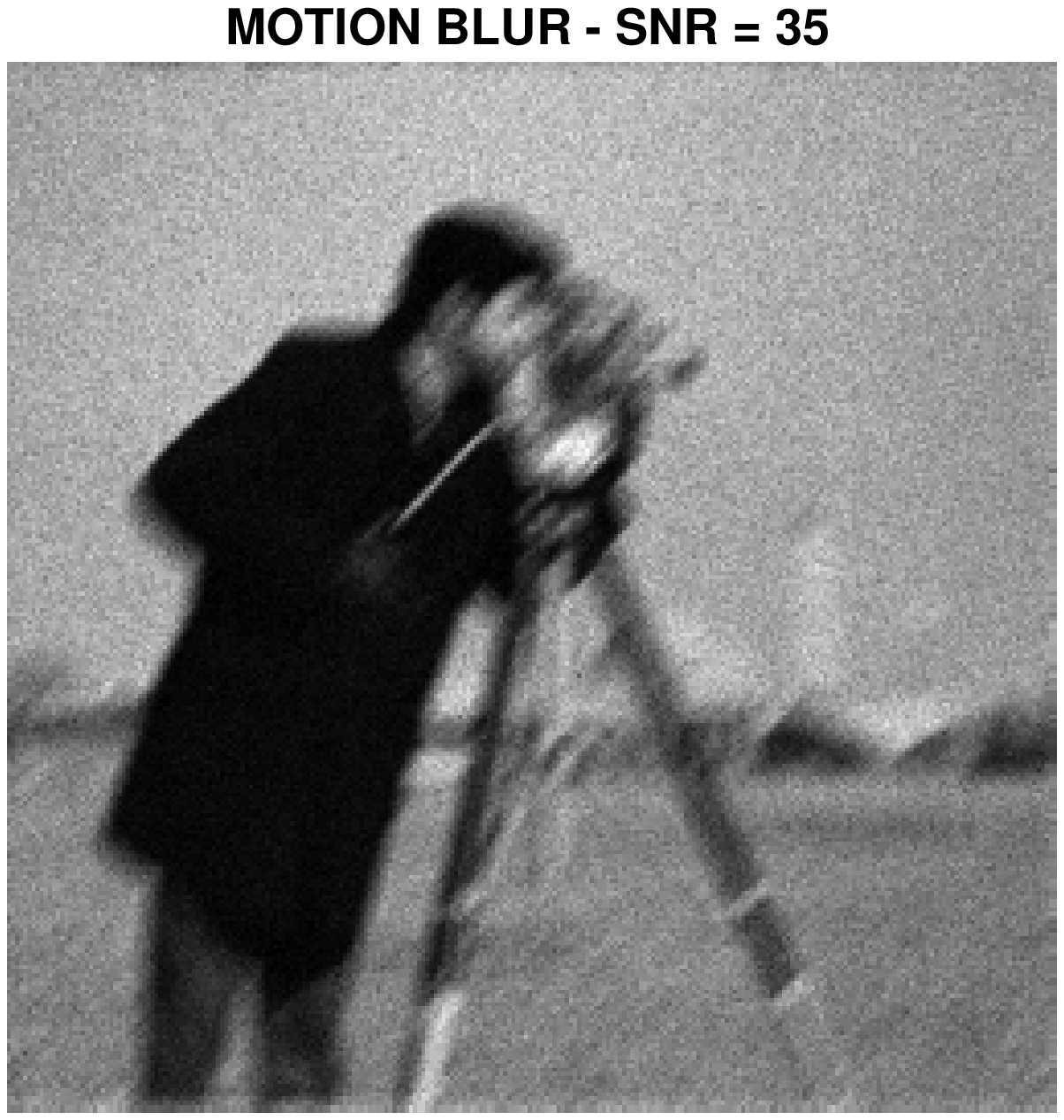}      &
			\hspace*{-1.5cm}
			\includegraphics[width=.47\textwidth]{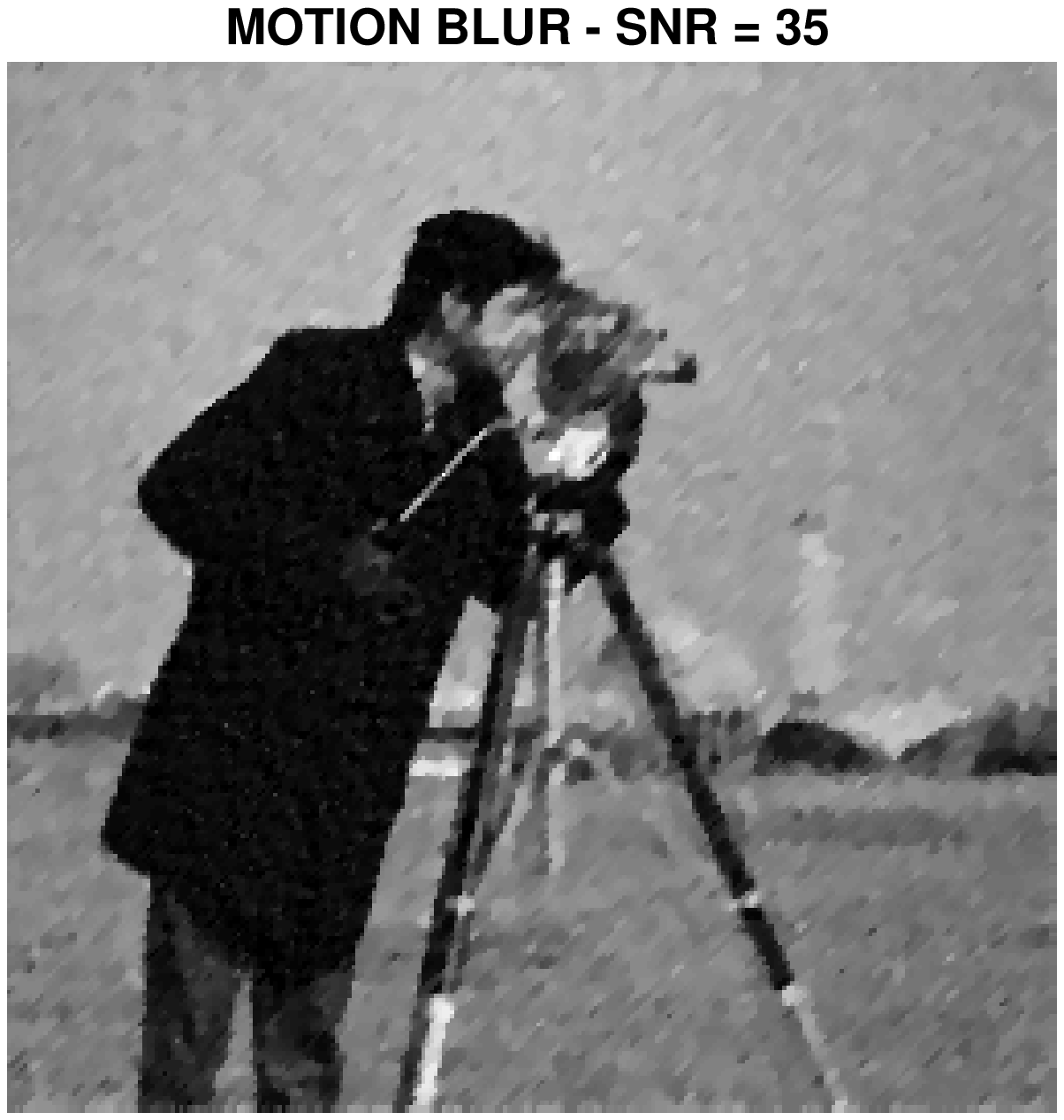}  \\[-4mm]
			\includegraphics[width=.47\textwidth]{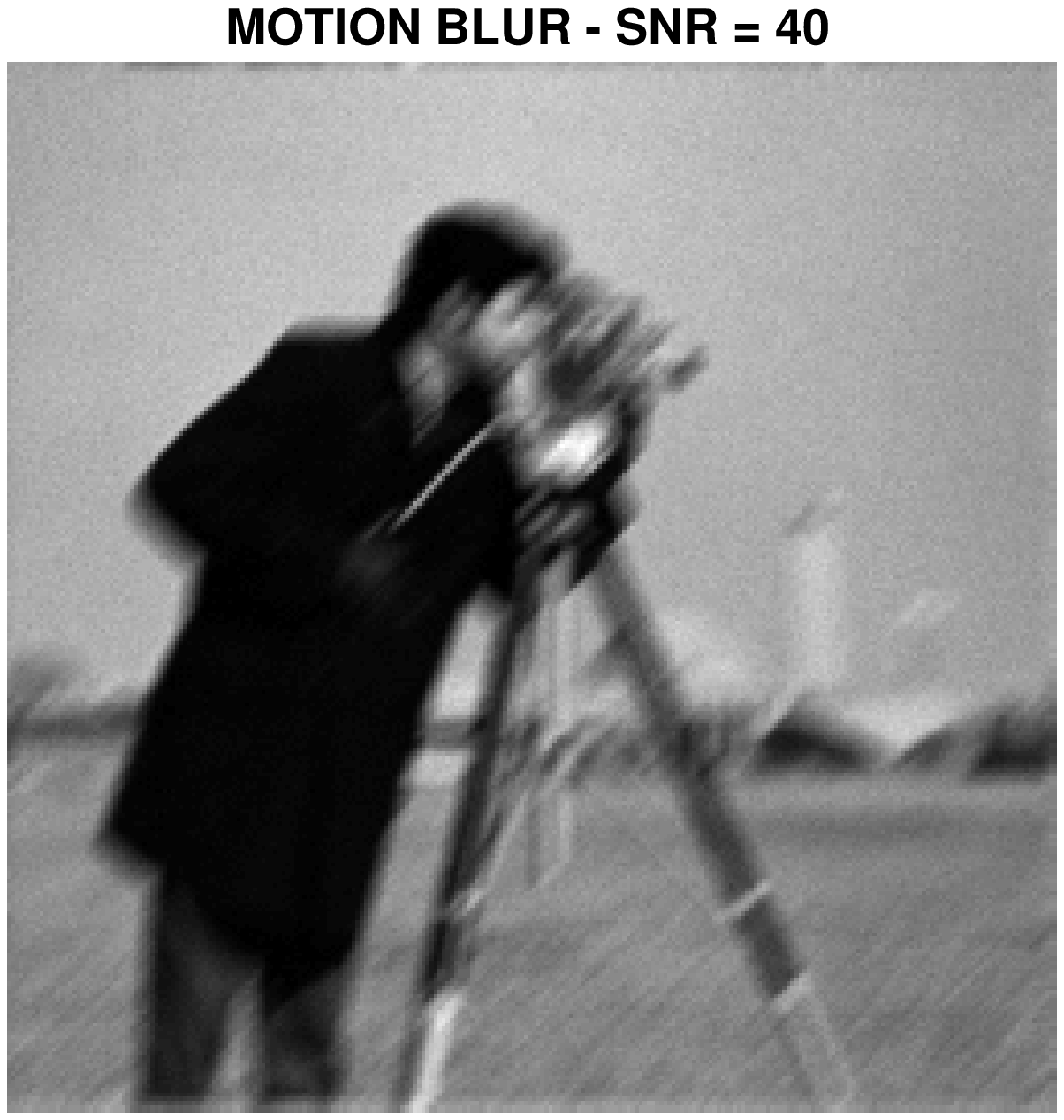}       &
			\hspace*{-1.5cm}
			\includegraphics[width=.47\textwidth]{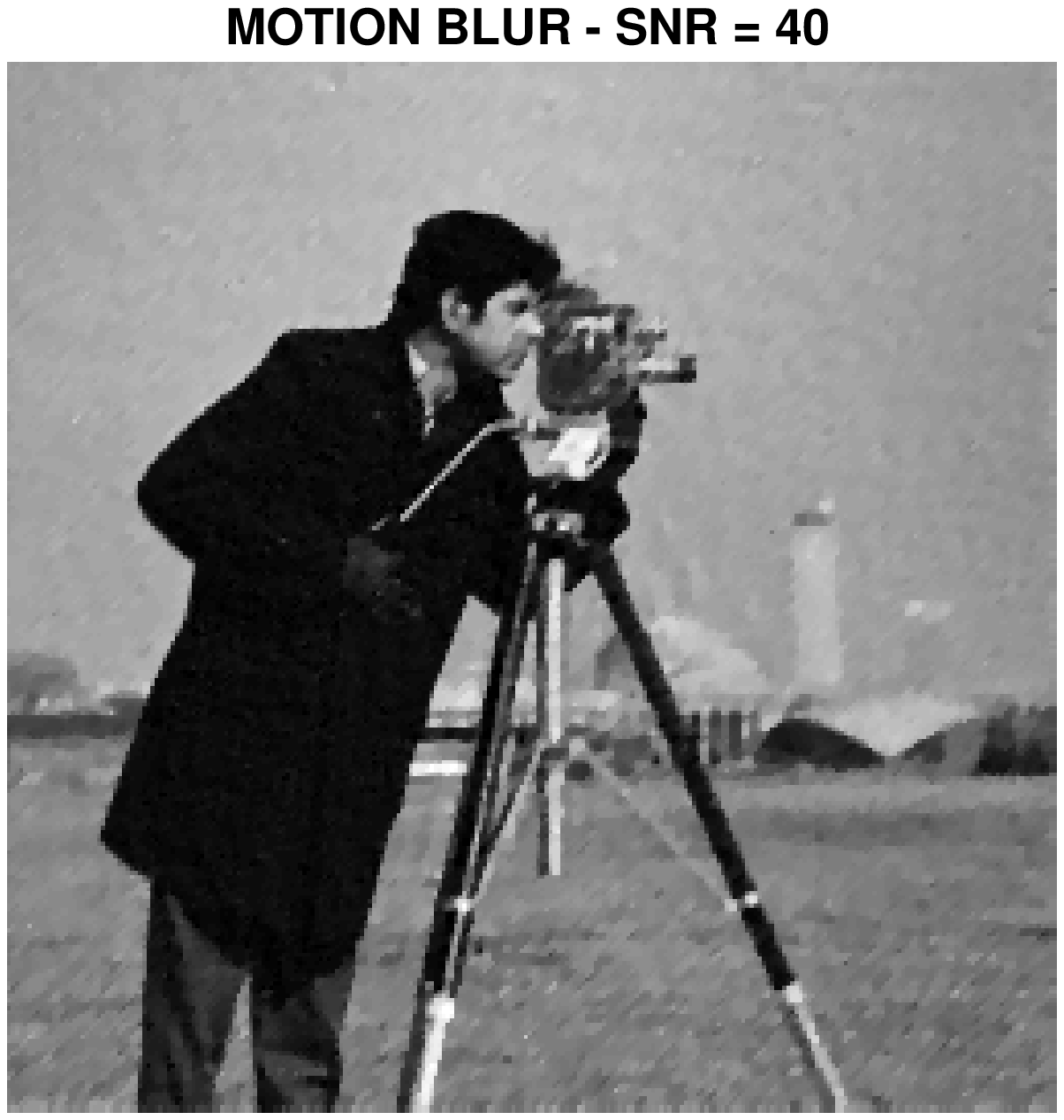}
		\end{tabular}
		\vspace*{-6mm}
		\caption{Cameraman: images corrupted by motion blur and Poisson noise (left) and images
		 restored by ACQUIRE (right).  
		 \label{fig:cameraman_mot}}
	\end{center}
\end{figure}

\begin{figure}[b!]
	\vspace*{2mm}
	\begin{center}
		\hspace*{-.55cm}
		\begin{tabular}{ccc}
			\includegraphics[width=.47\textwidth]{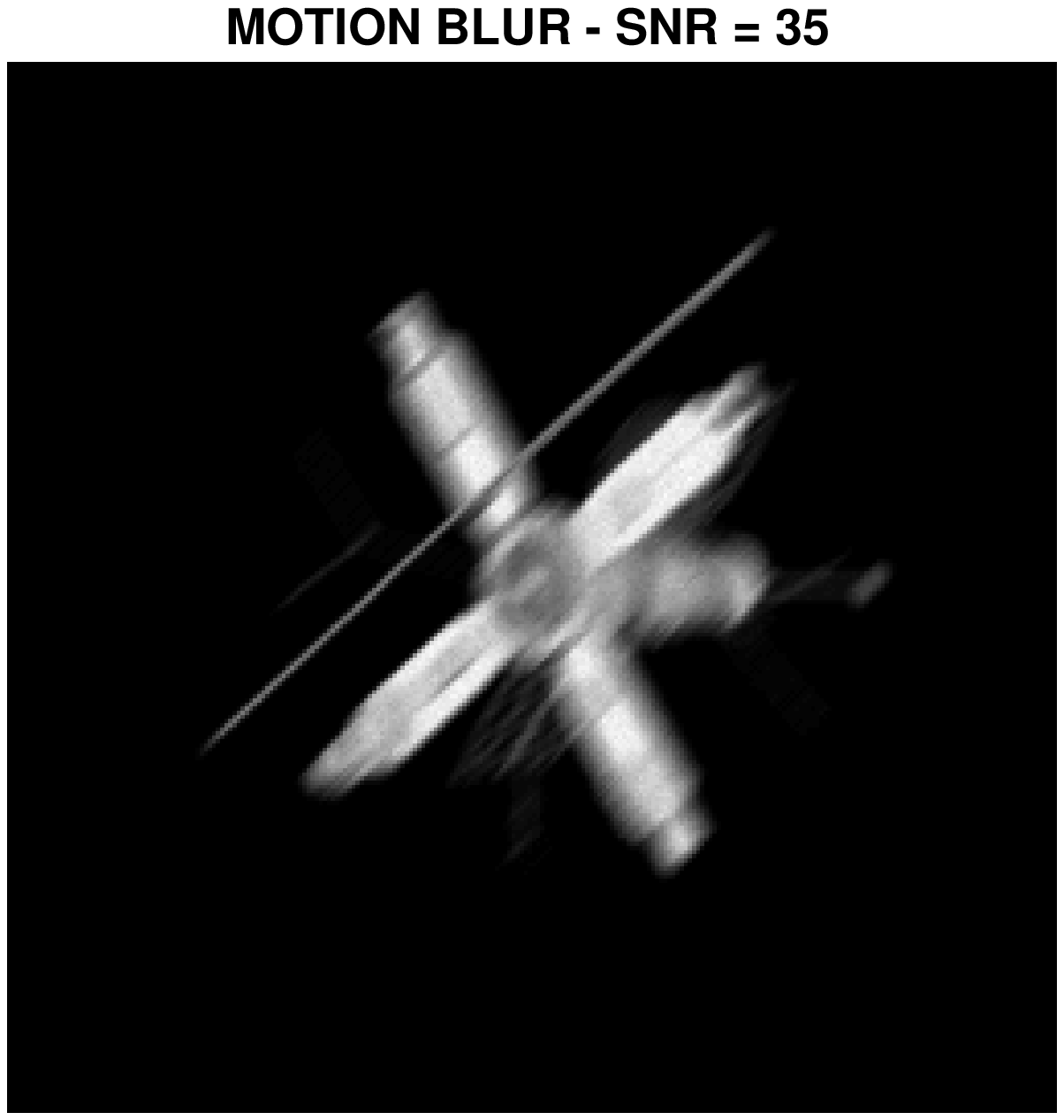}      &
			\hspace*{-1.5cm}
			\includegraphics[width=.47\textwidth]{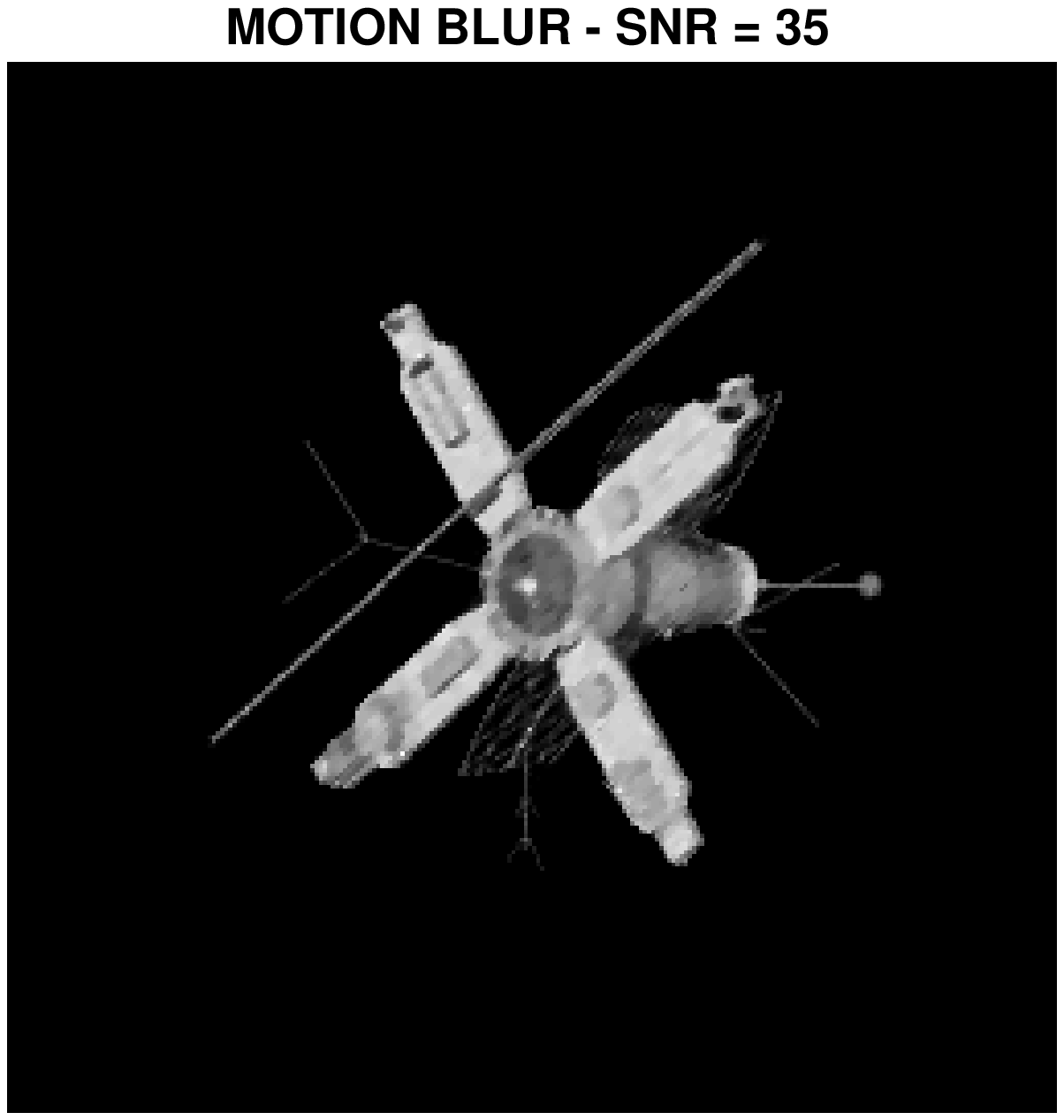}  \\[-4mm]
			\includegraphics[width=.47\textwidth]{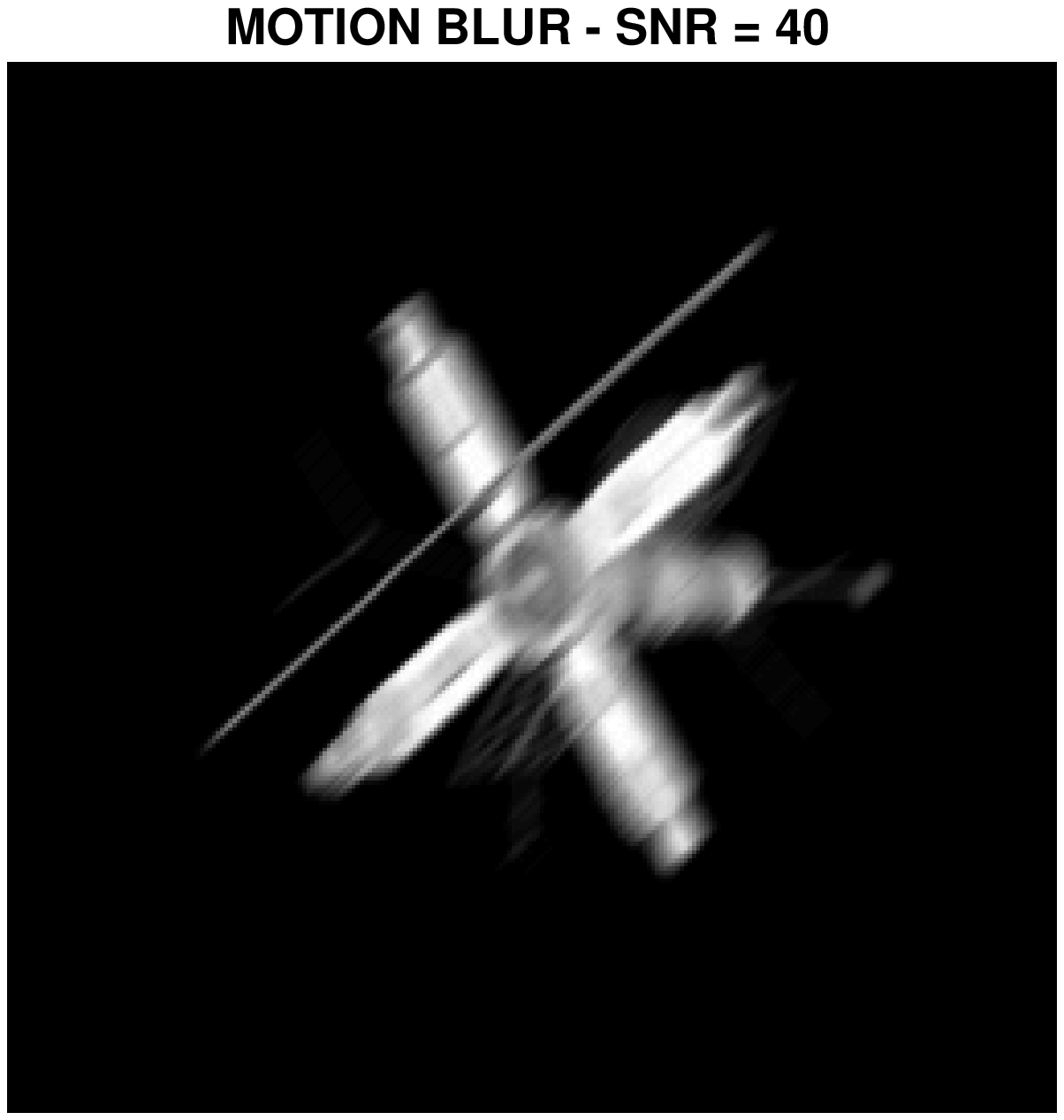}       &
			\hspace*{-1.5cm}
			\includegraphics[width=.47\textwidth]{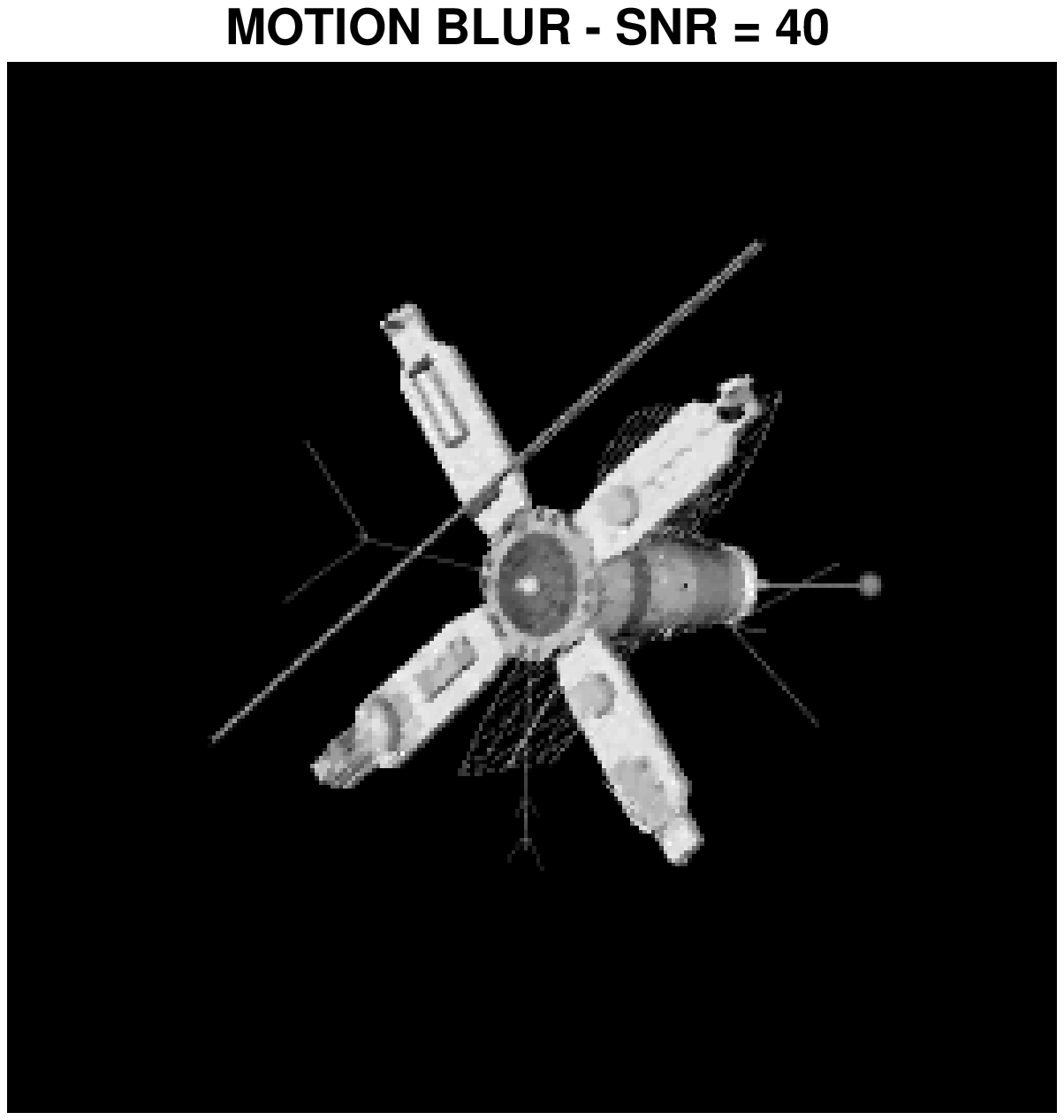}
		\end{tabular}
		\vspace*{-6mm}
		\caption{Satellite: images corrupted by motion blur and Poisson noise (left) and images
		 restored by ACQUIRE (right).  
		 \label{fig:satellite_mot}}
	\end{center}
\end{figure}

\clearpage

\begin{figure}[t!]
	\vspace*{2mm}
	\begin{center}
		\hspace*{-.55cm}
		\begin{tabular}{ccc}
			\includegraphics[width=.47\textwidth]{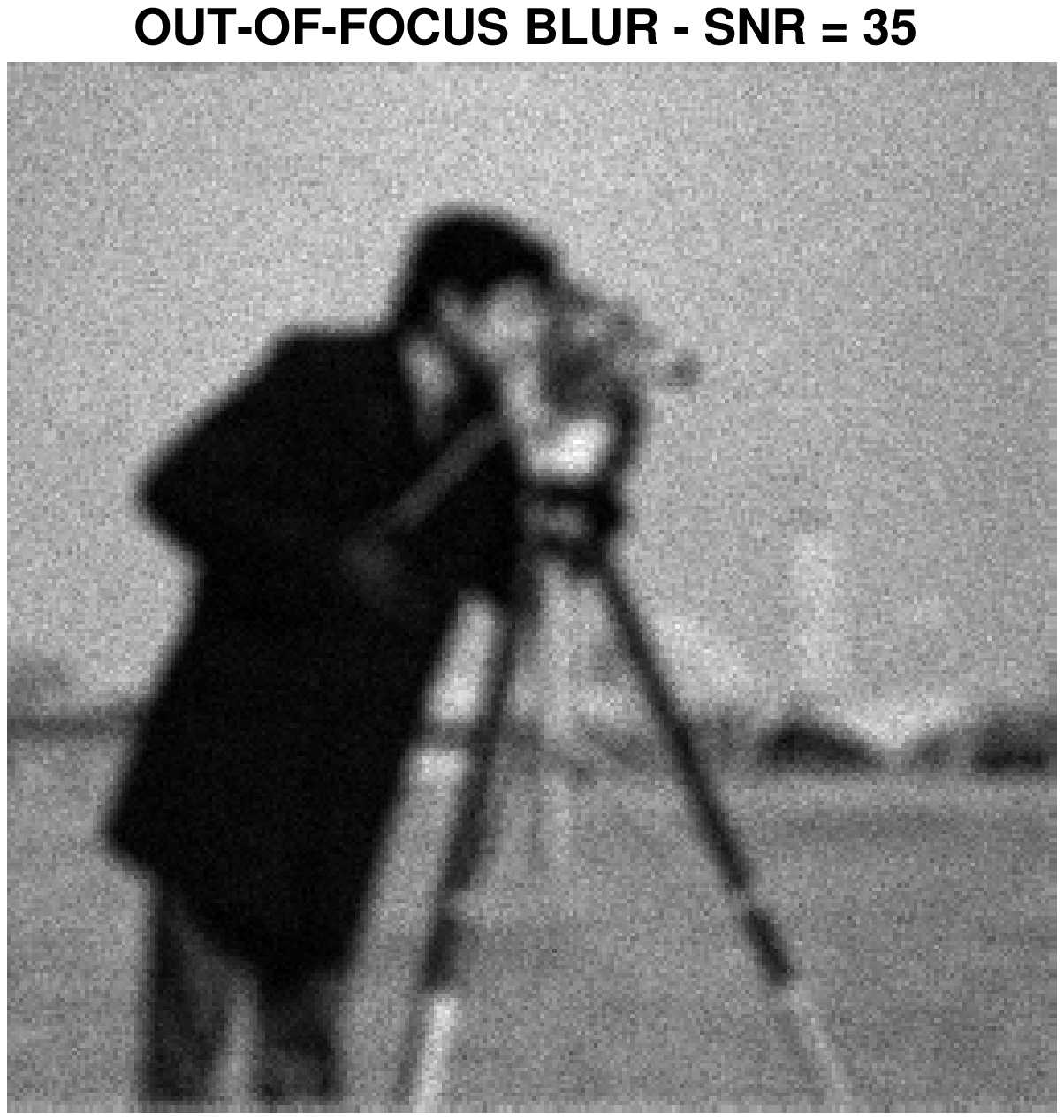}      &
			\hspace*{-1.5cm}
			\includegraphics[width=.47\textwidth]{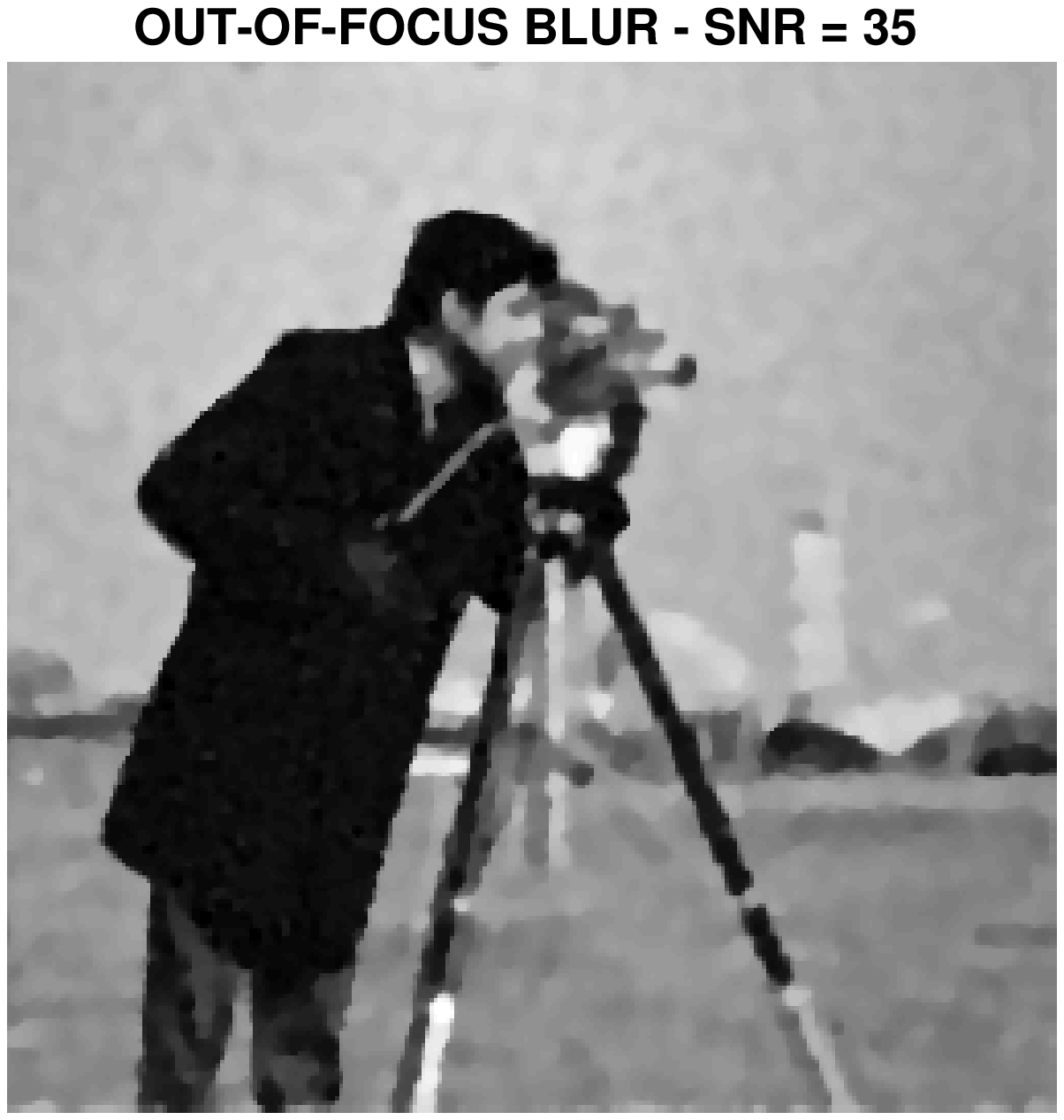}  \\[-4mm]
			\includegraphics[width=.47\textwidth]{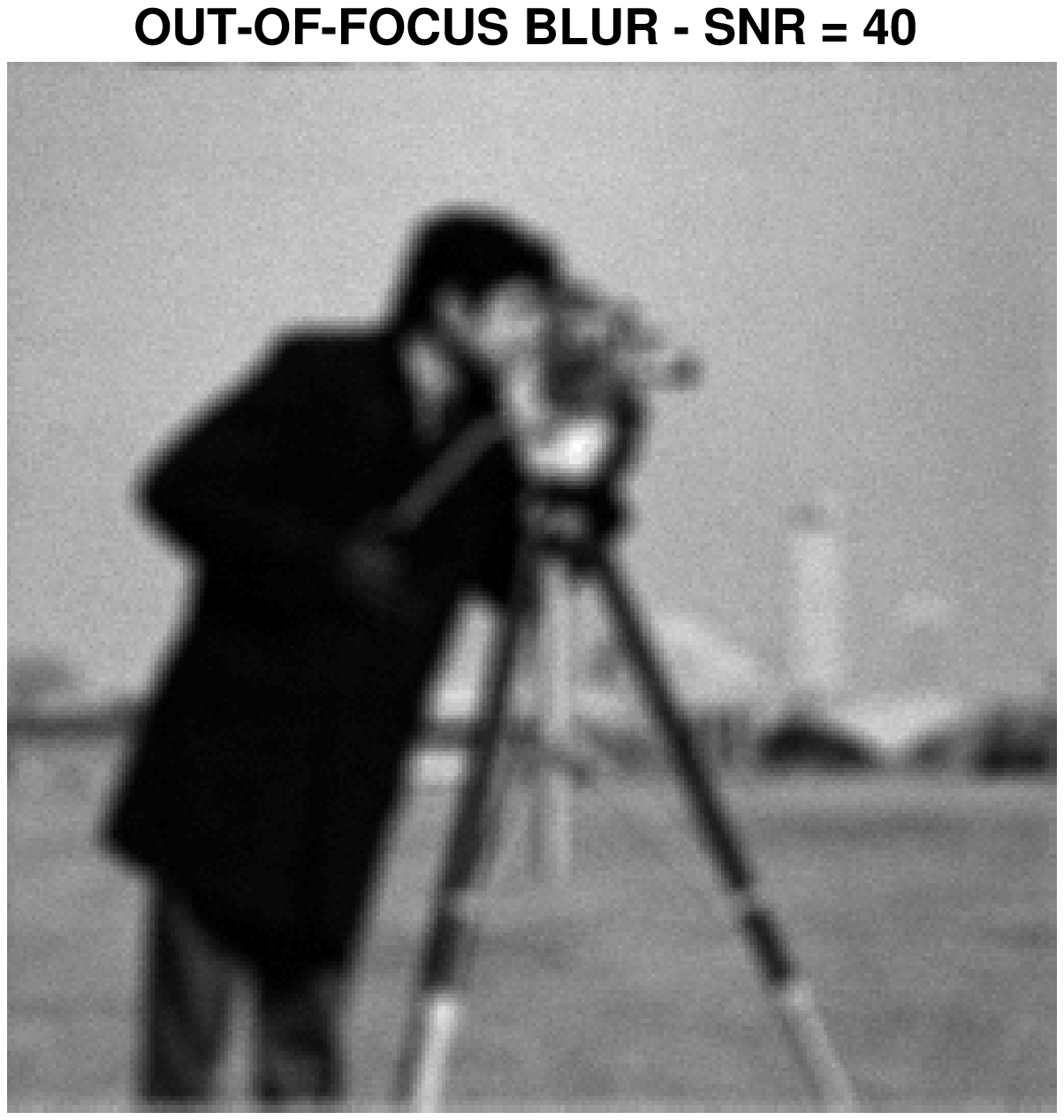}       &
			\hspace*{-1.5cm}
			\includegraphics[width=.47\textwidth]{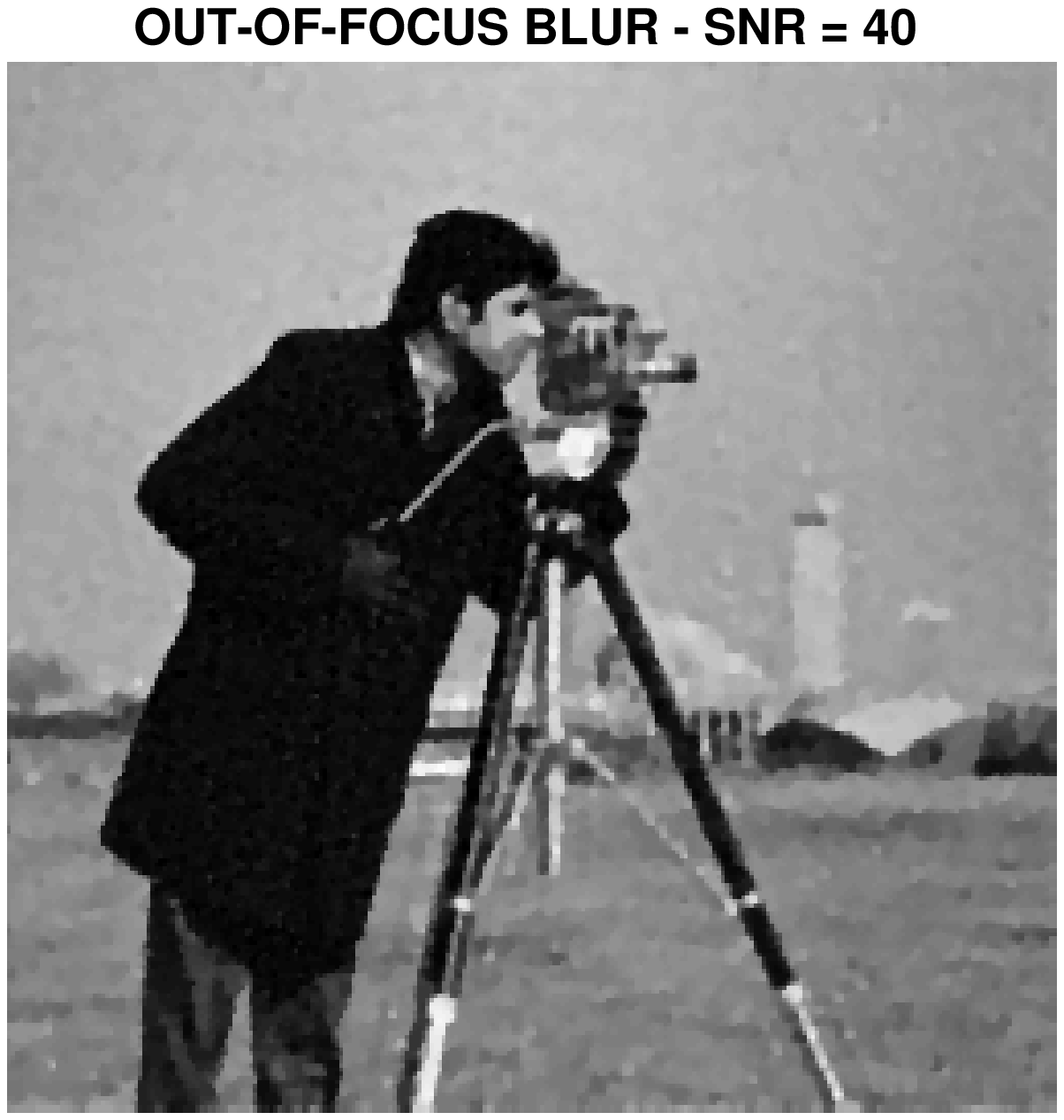}
		\end{tabular}
		\vspace*{-6mm}
		\caption{Cameraman: images corrupted by out-of-focus blur and Poisson noise (left) and images
		 restored by ACQUIRE (right).  
		 \label{fig:cameraman_of}}
	\end{center}
\end{figure}

\begin{figure}[b!]
	\vspace*{2mm}
	\begin{center}
		\hspace*{-.55cm}
		\begin{tabular}{ccc}
			\includegraphics[width=.47\textwidth]{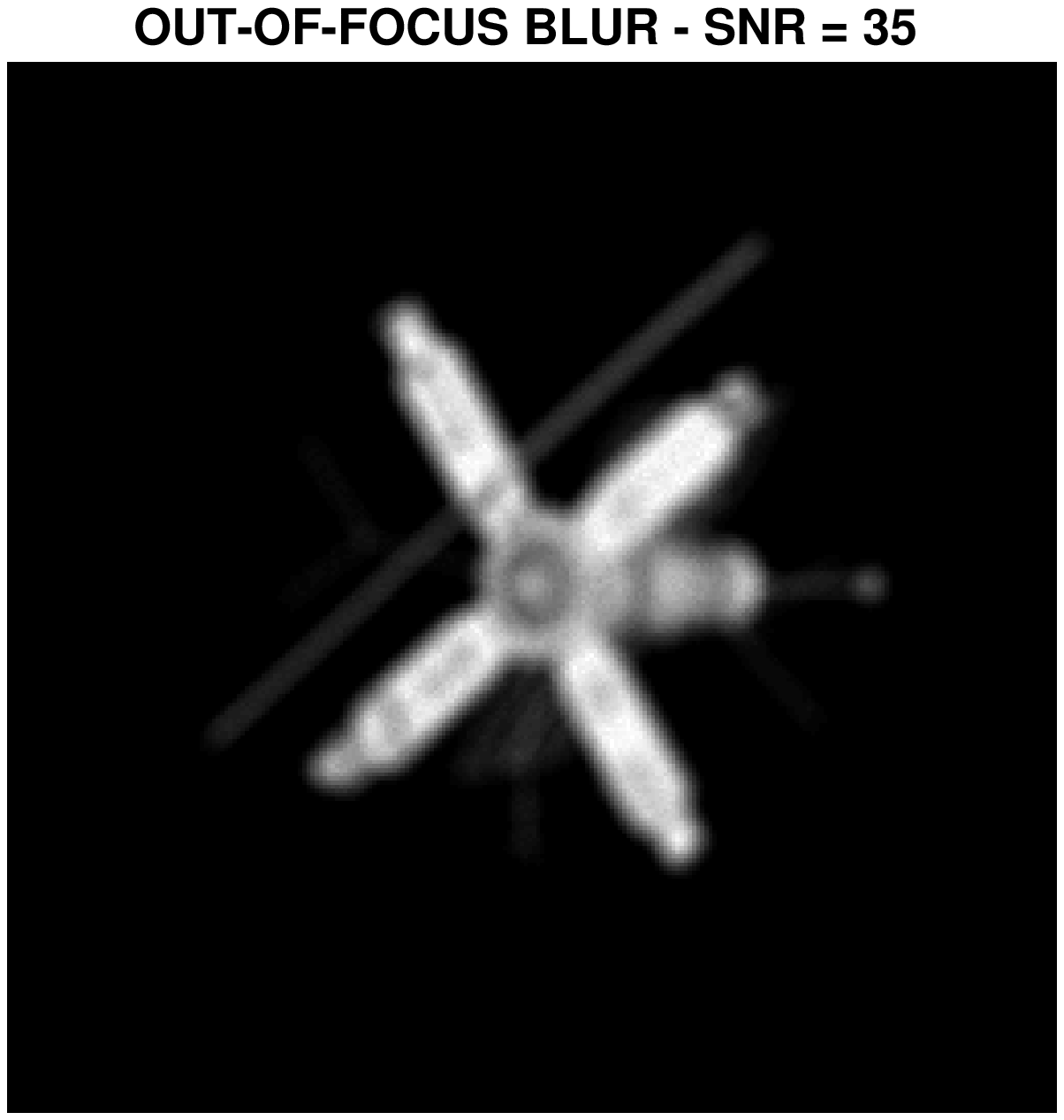}      &
			\hspace*{-1.5cm}
			\includegraphics[width=.47\textwidth]{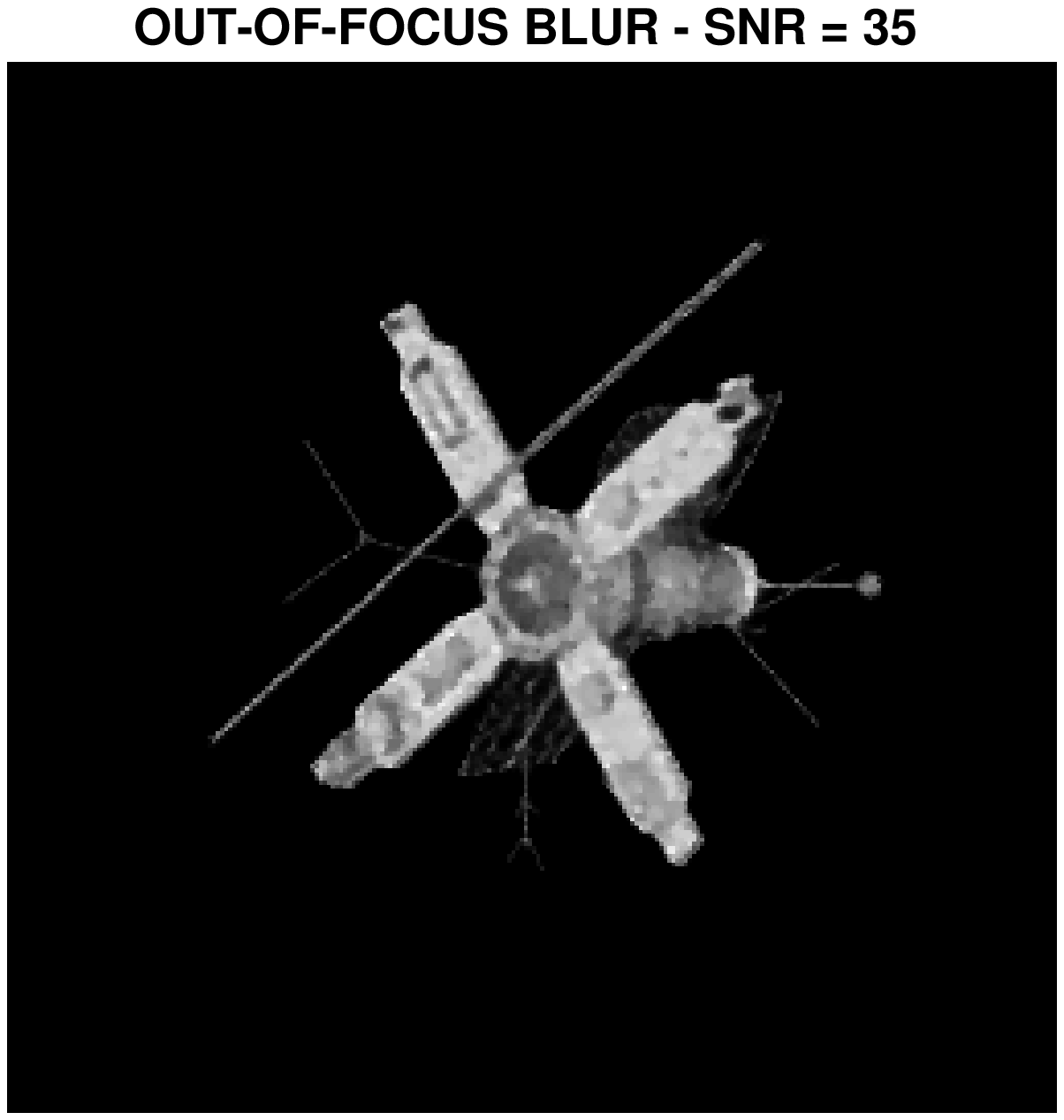}  \\[-4mm]
			\includegraphics[width=.47\textwidth]{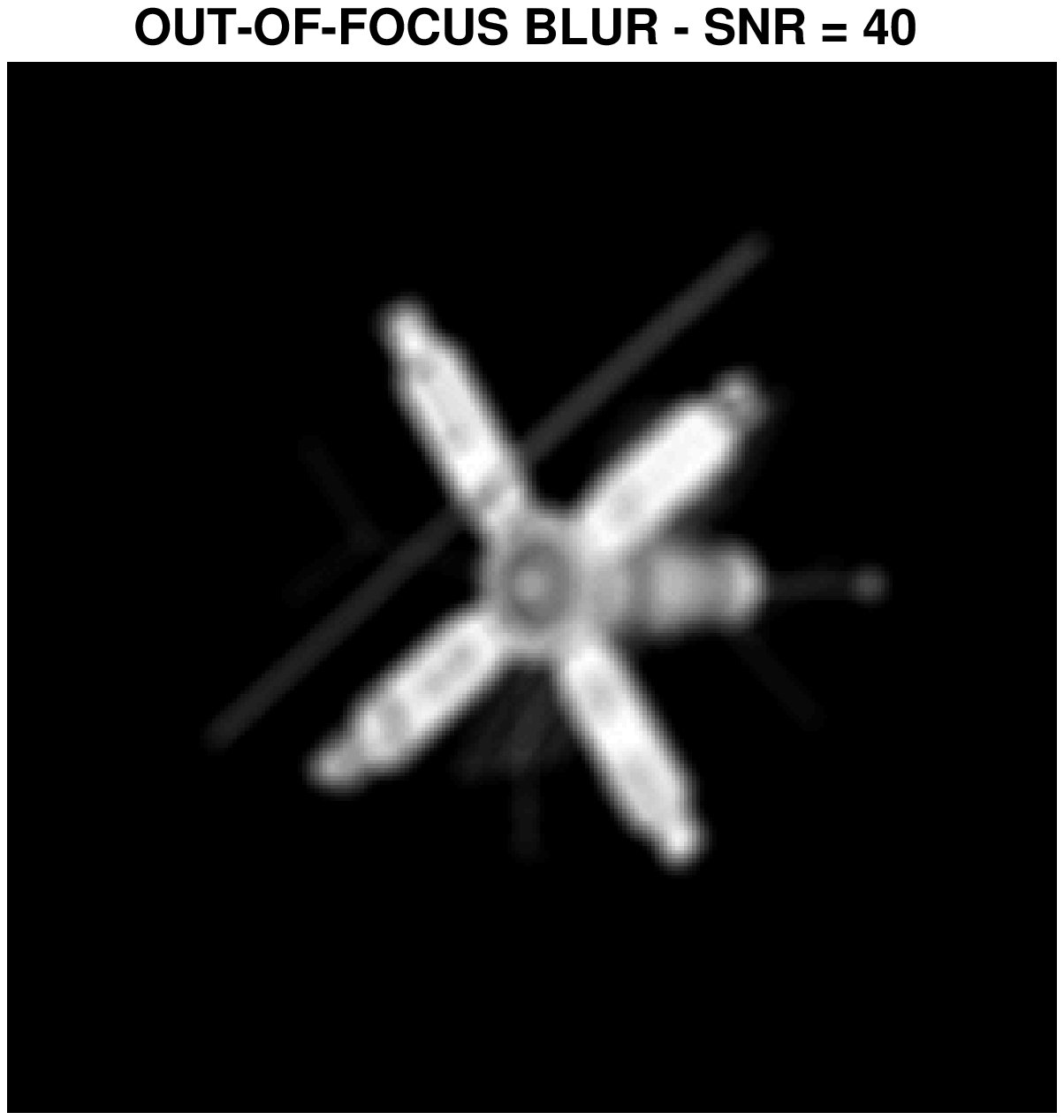}       &
			\hspace*{-1.5cm}
			\includegraphics[width=.47\textwidth]{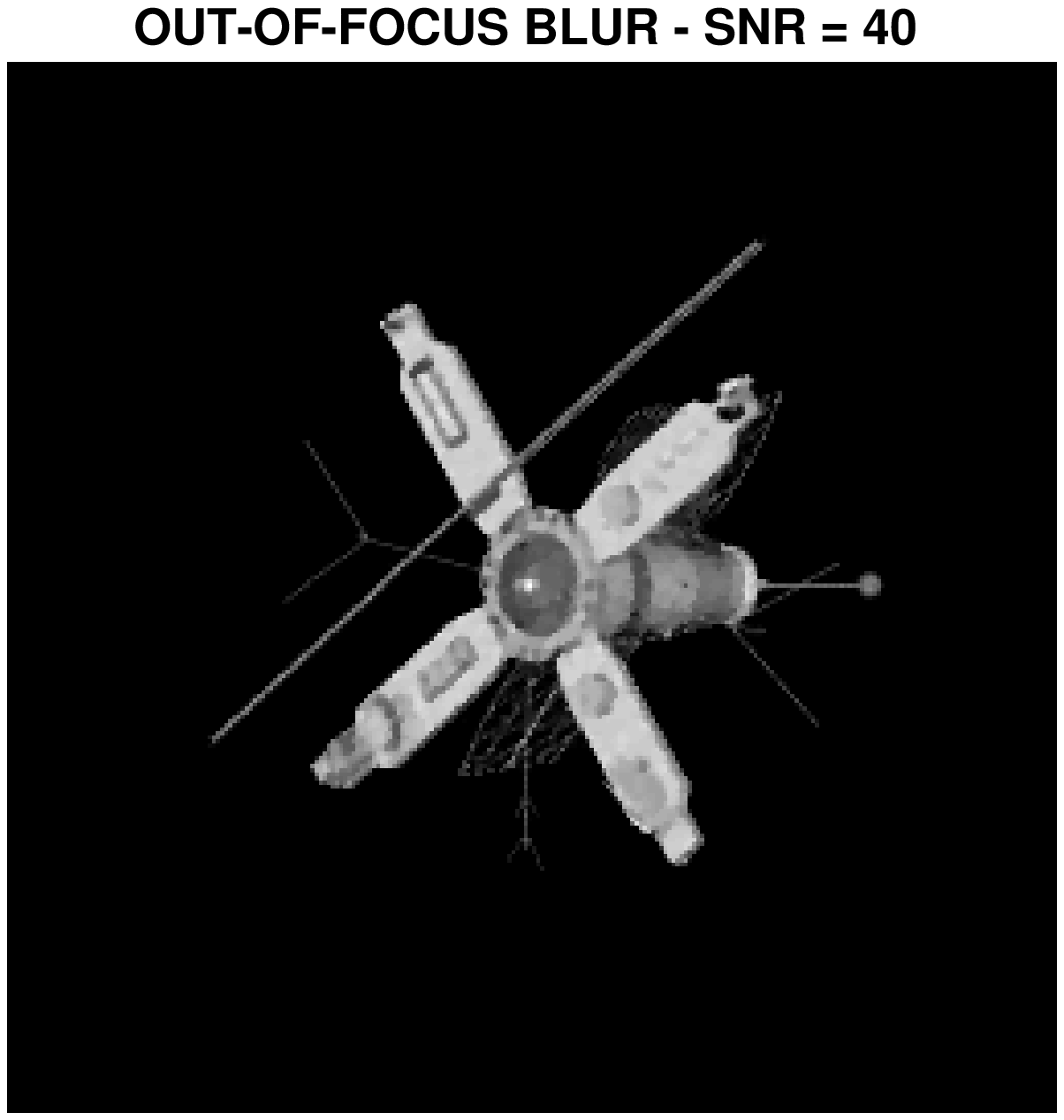}
		\end{tabular}
		\vspace*{-6mm}
		\caption{Satellite: images corrupted by out-of-focus blur and Poisson noise (left) and images
		 restored by ACQUIRE (right).  
		 \label{fig:satellite_of}}
	\end{center}
\end{figure}

\clearpage

\begin{figure}[p!]
	\begin{center}
		\begin{tabular}{cc}
			\includegraphics[width=.44\textwidth]{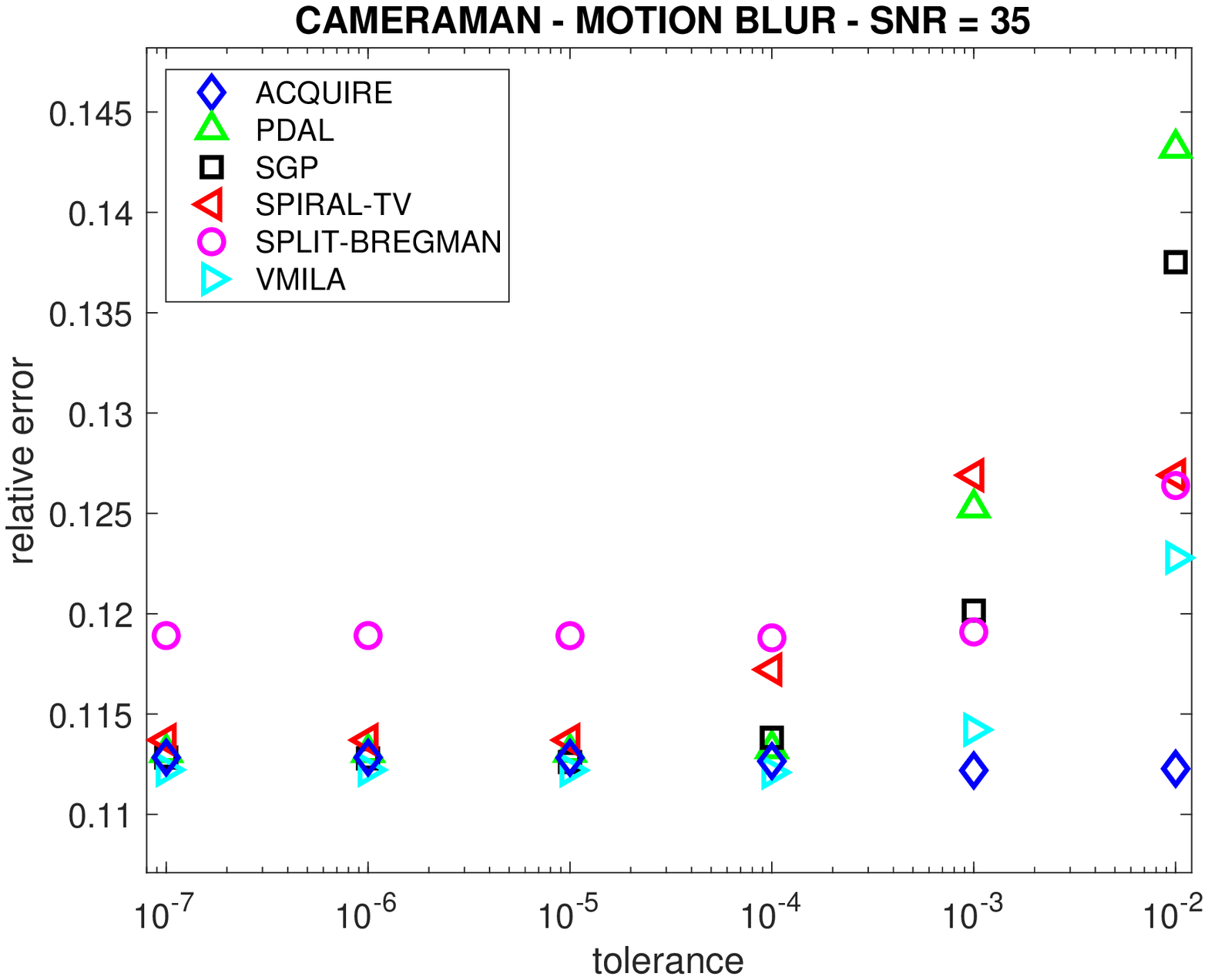}     &
			\includegraphics[width=.44\textwidth]{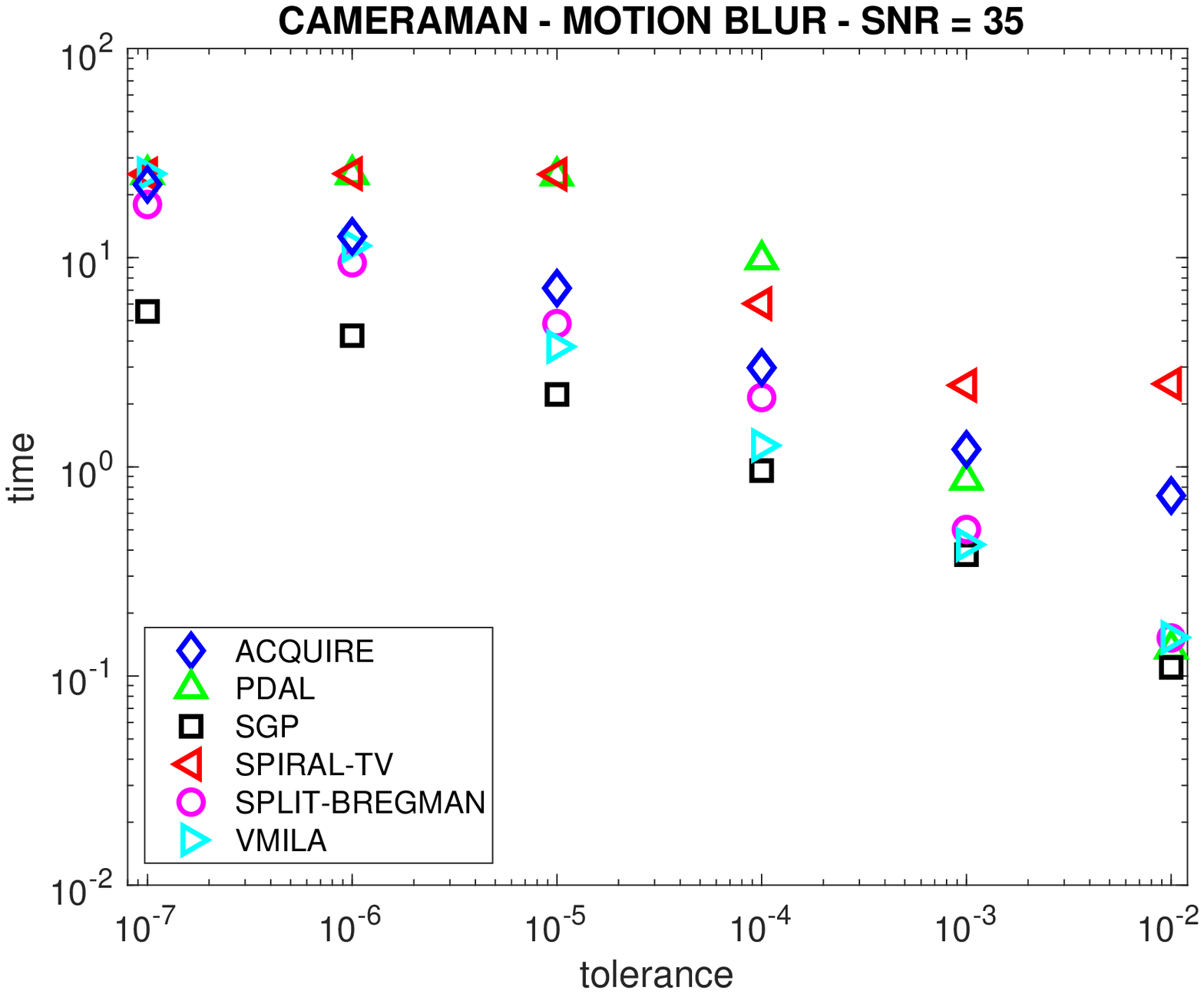}  \\[-1mm]
			\includegraphics[width=.44\textwidth]{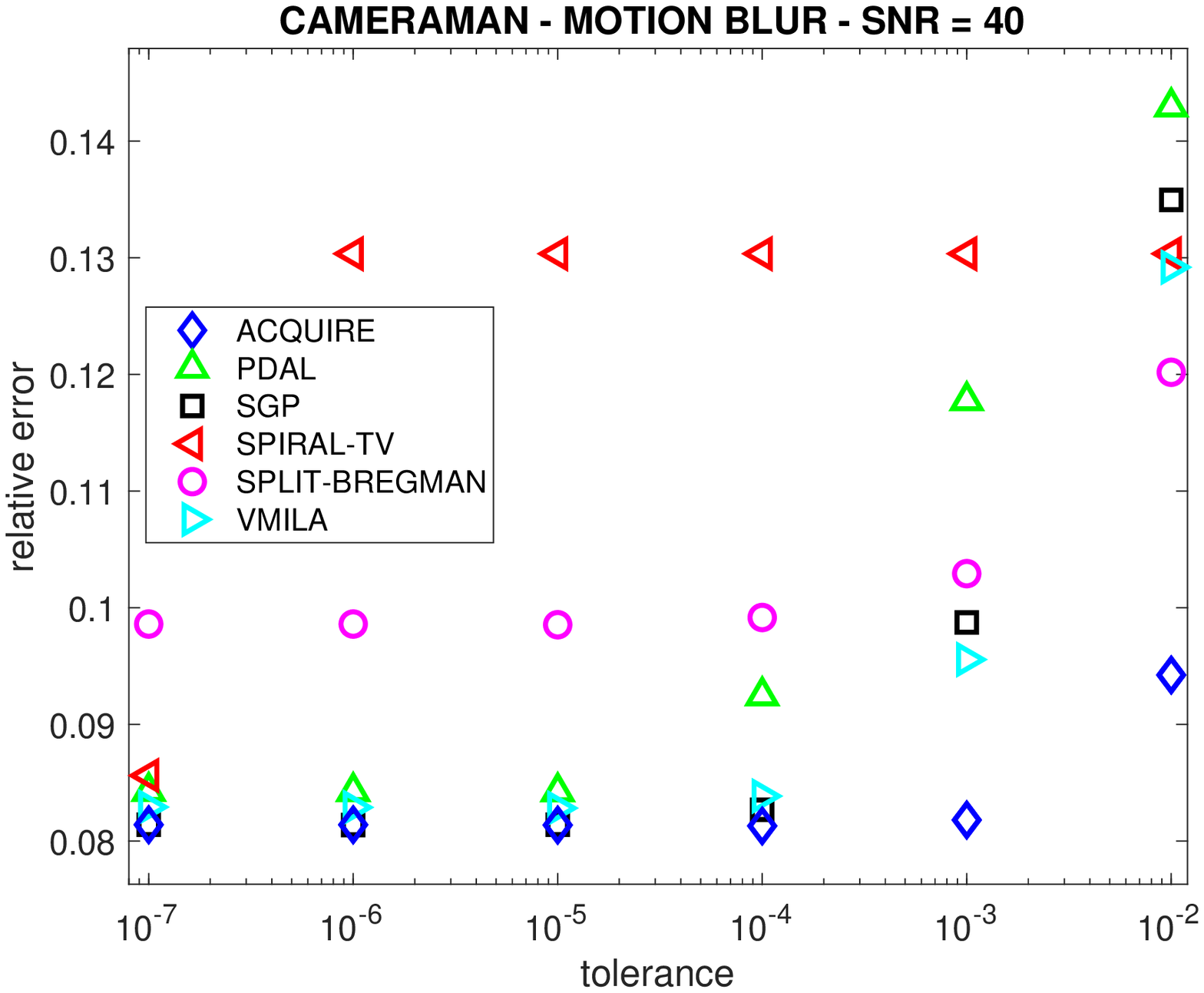}    &
			\includegraphics[width=.44\textwidth]{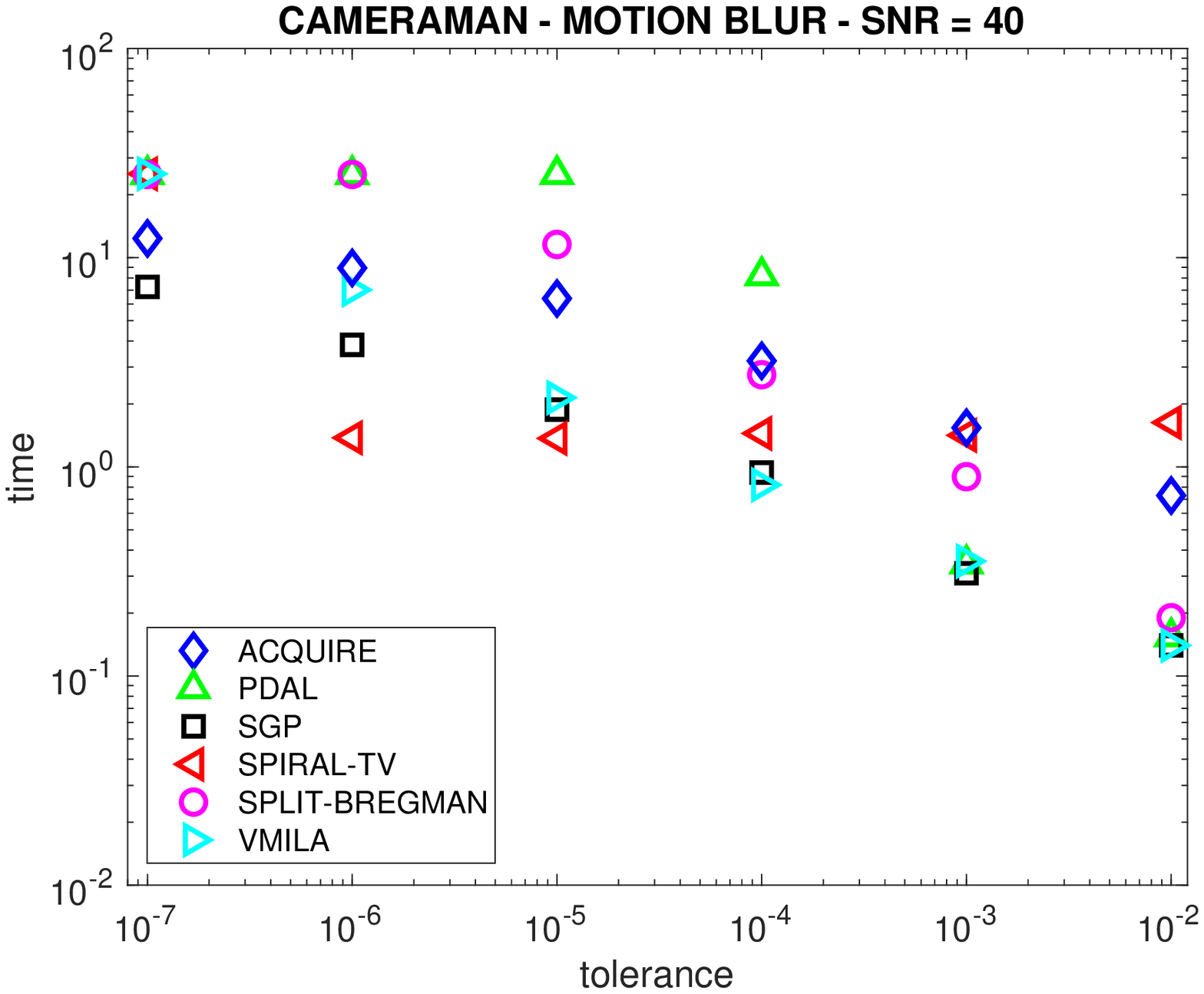}  \\[-1mm]
			\includegraphics[width=.44\textwidth]{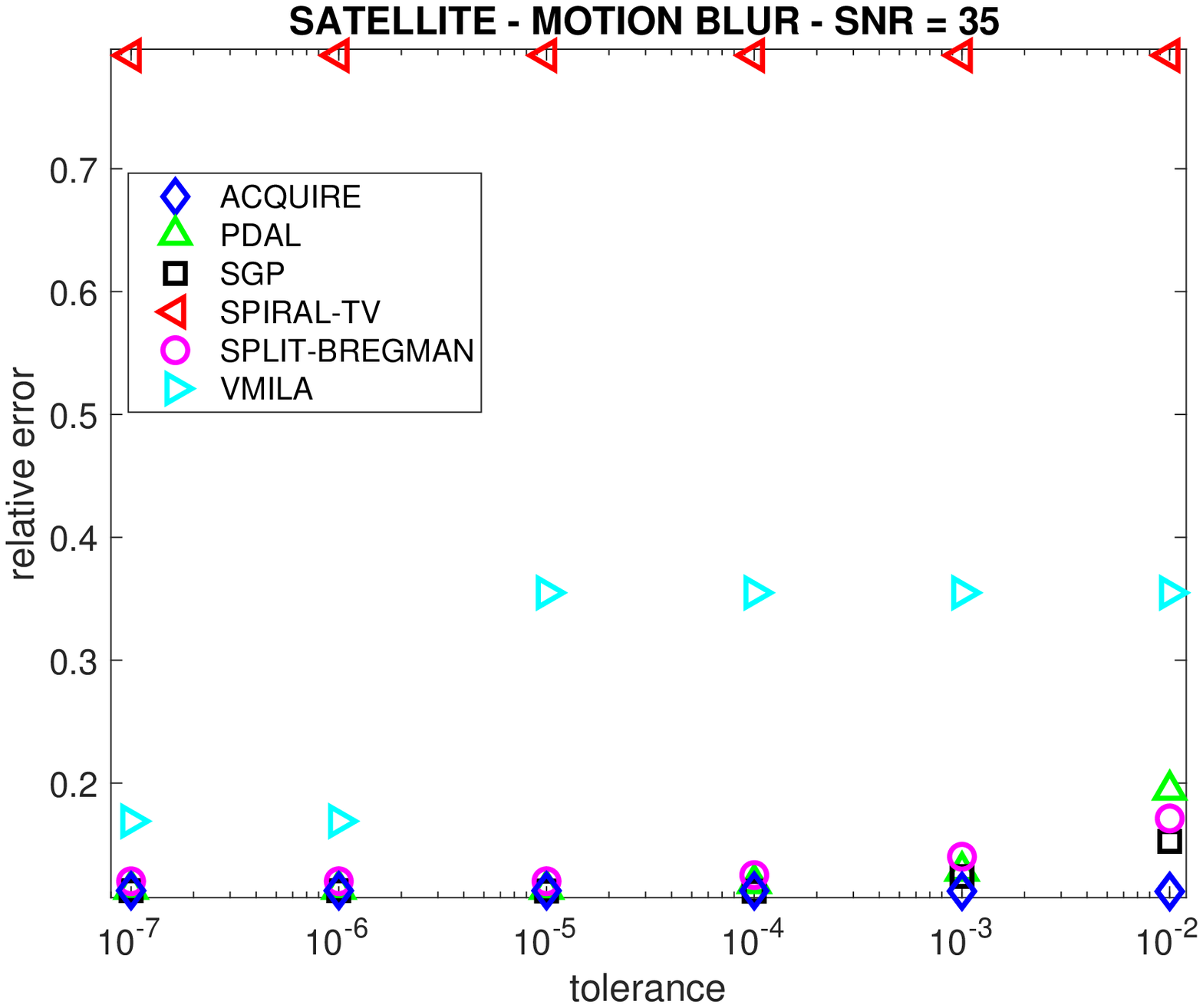}     &
			\includegraphics[width=.44\textwidth]{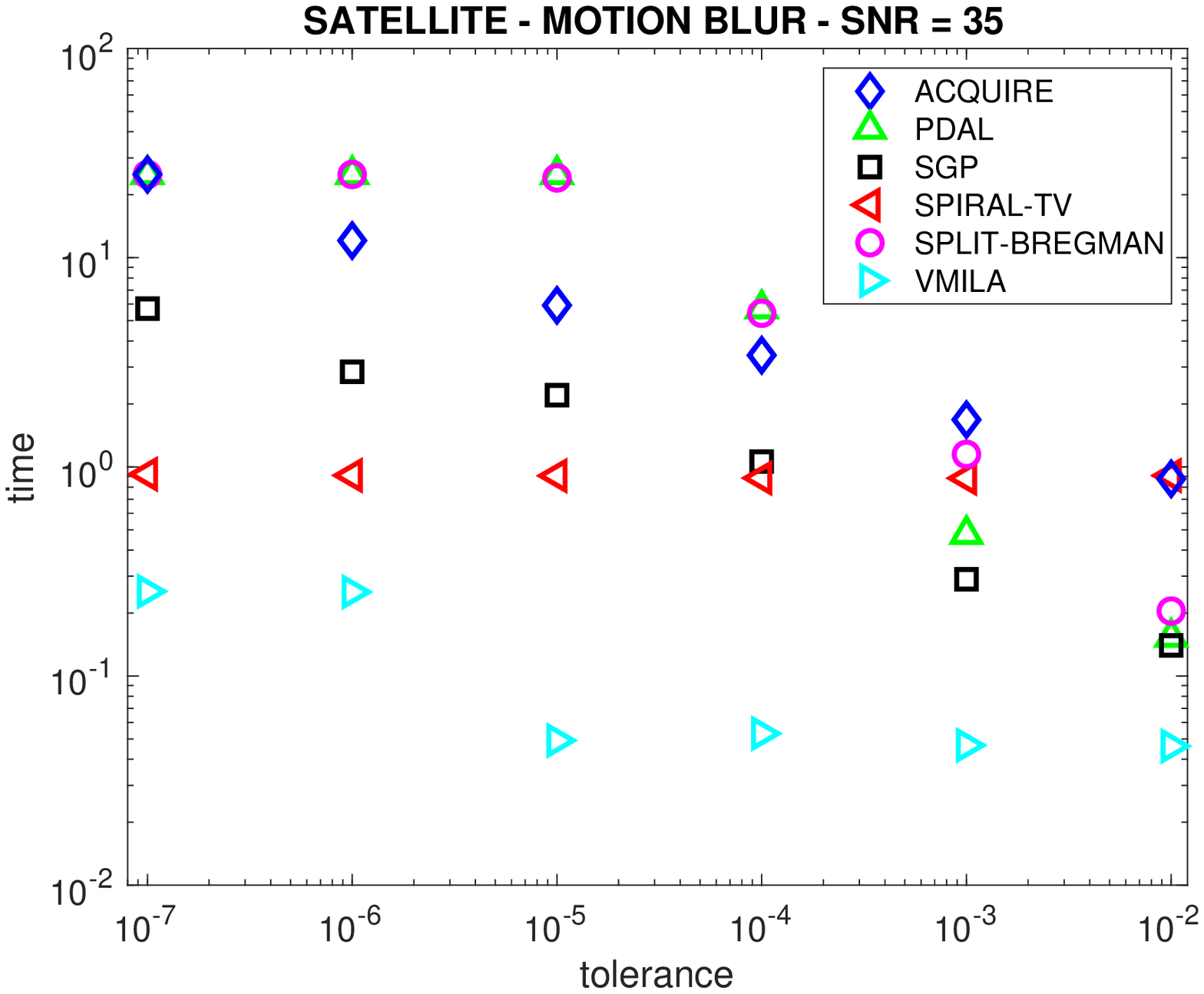}  \\[-1mm]
			\includegraphics[width=.44\textwidth]{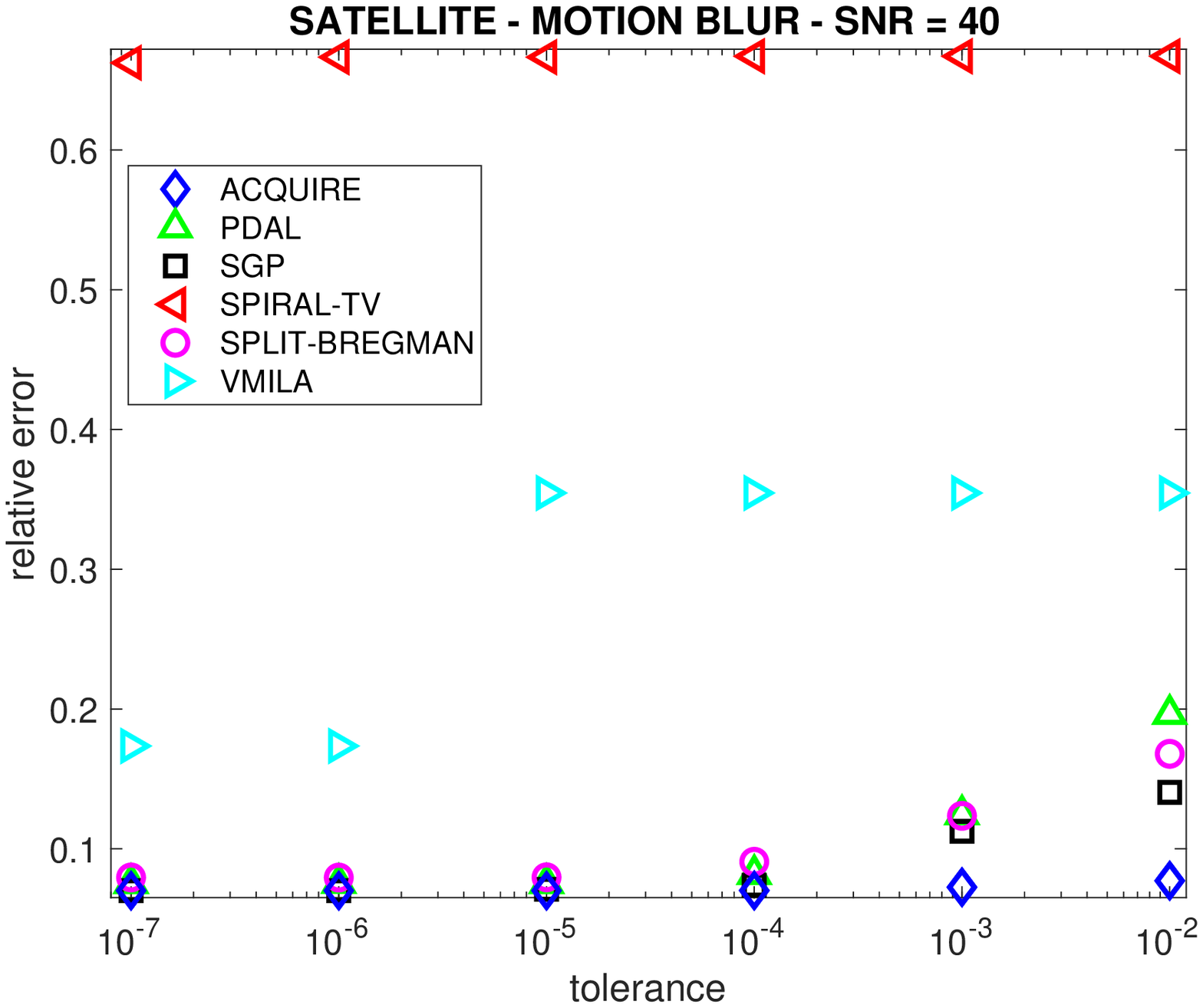}    &
			\includegraphics[width=.44\textwidth]{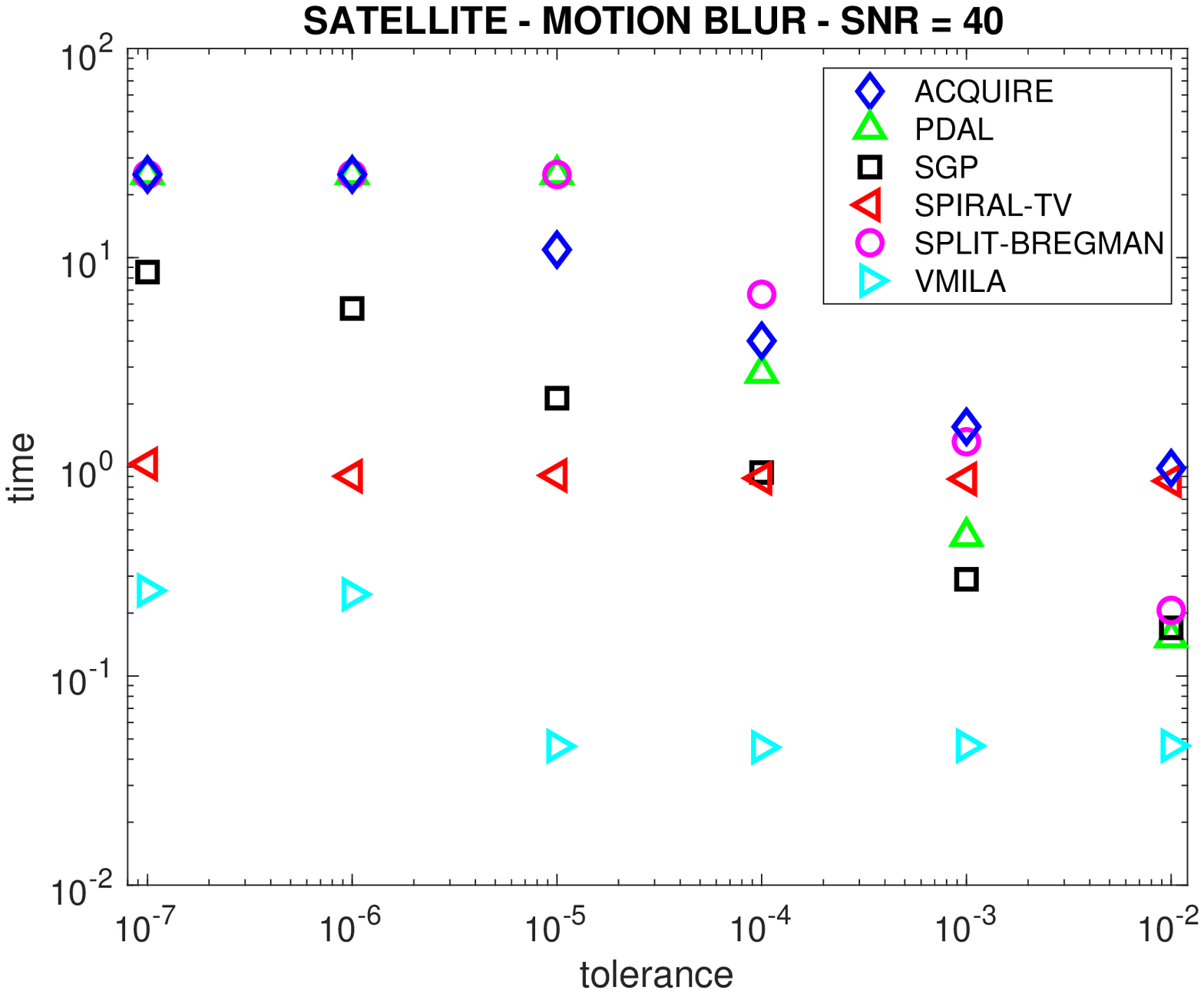}  \\[-1mm]
		\end{tabular}
		\vspace*{-2mm}
		\caption{Test set T2, motion blur, SNR $= 35,40$: relative error (left) and execution time (right) versus tolerance,
		for all the methods.\label{fig:compar_mot}}	
	\end{center}
\end{figure}

\clearpage

\begin{figure}[p!]
	\begin{center}
		\begin{tabular}{cc}
			\includegraphics[width=.44\textwidth]{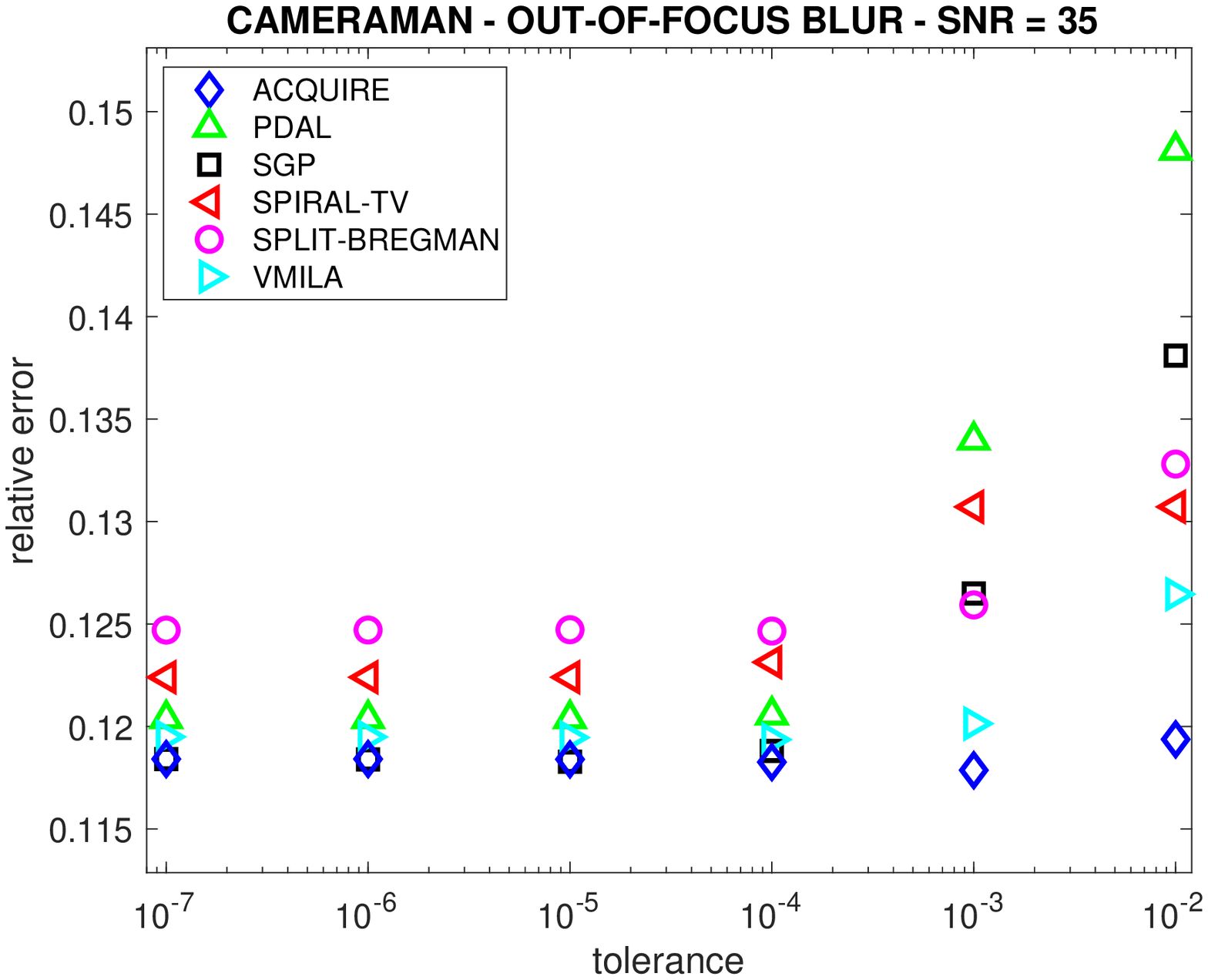}     &
			\includegraphics[width=.44\textwidth]{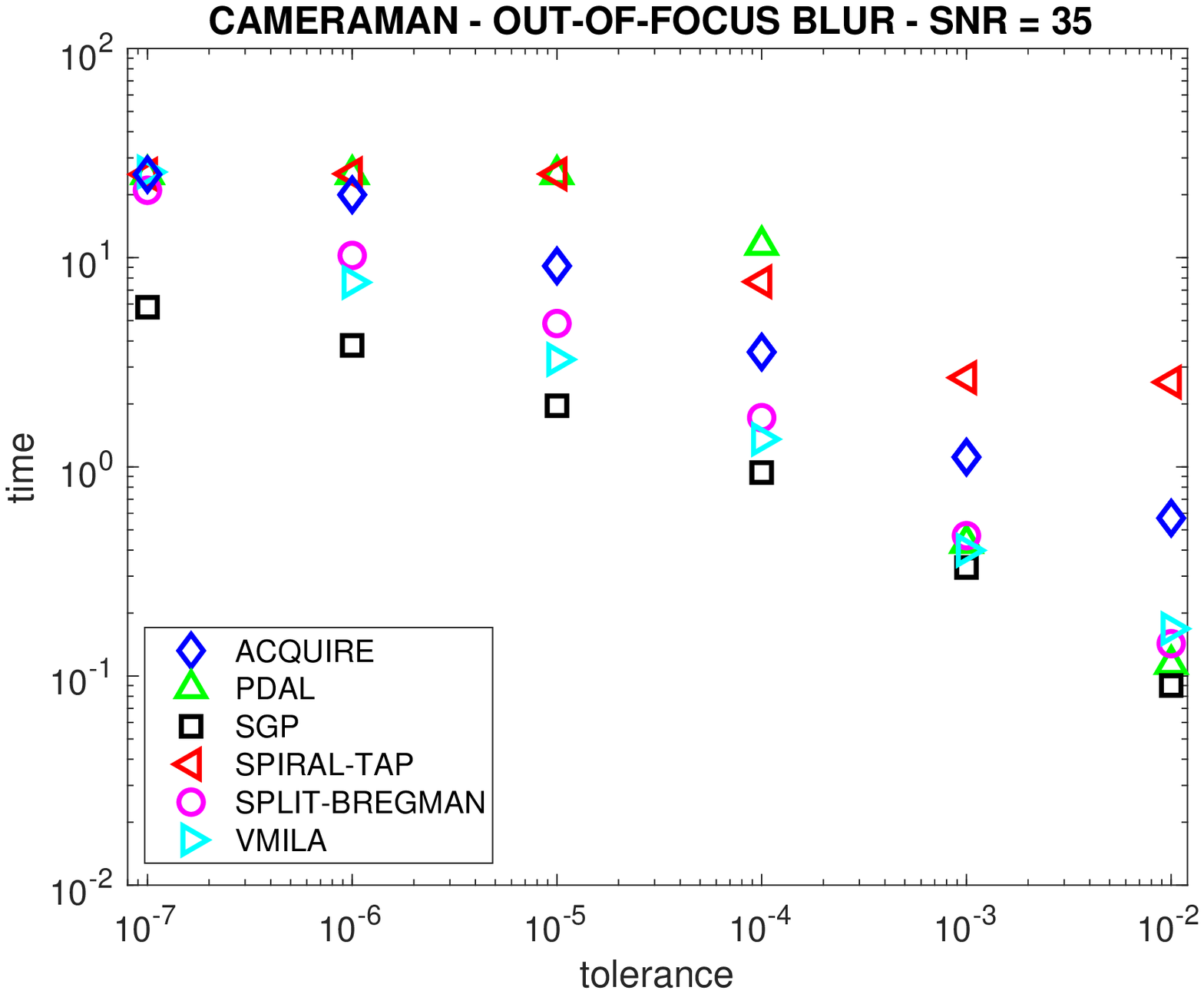}  \\[-1mm]
			\includegraphics[width=.44\textwidth]{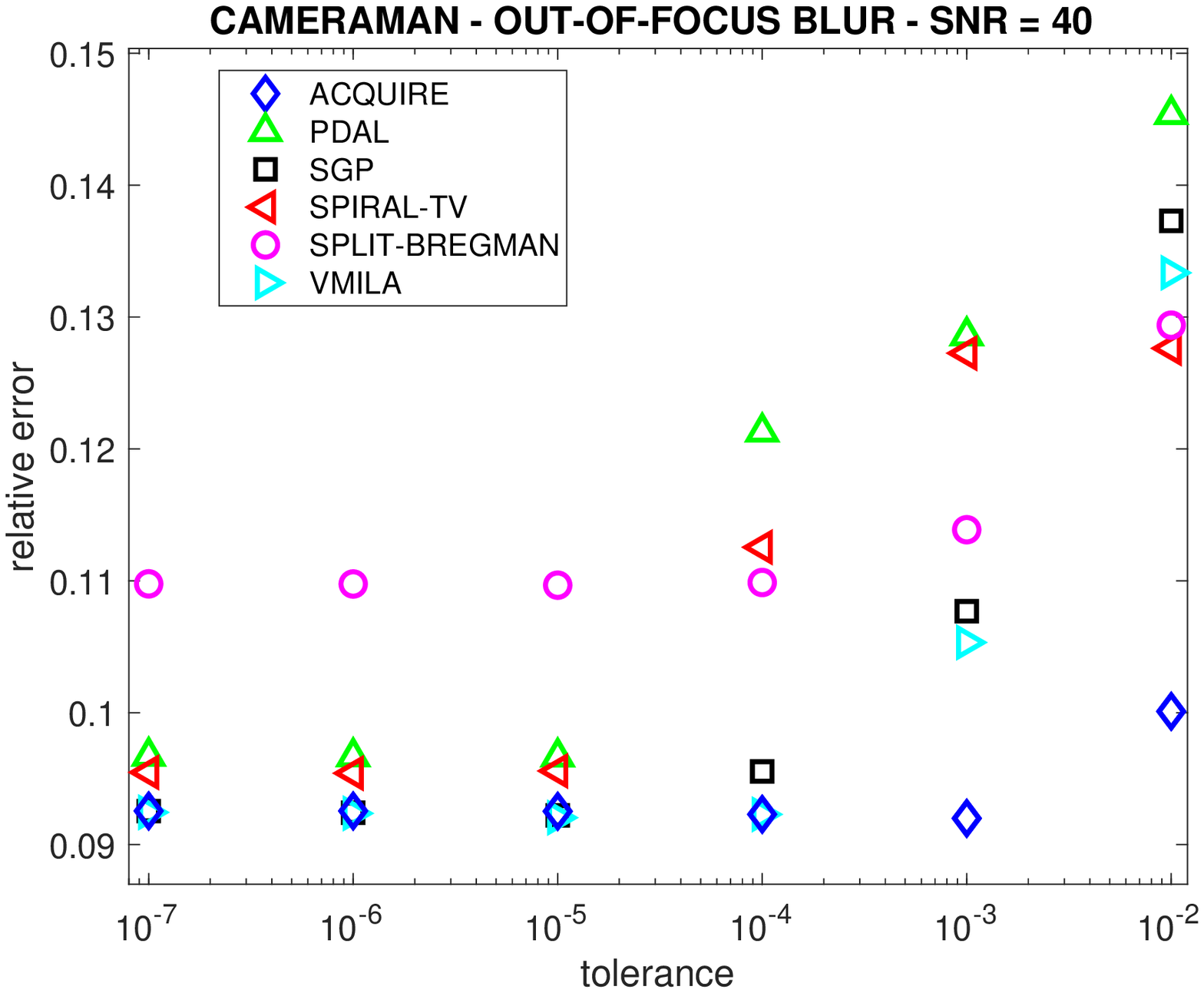}    &
			\includegraphics[width=.44\textwidth]{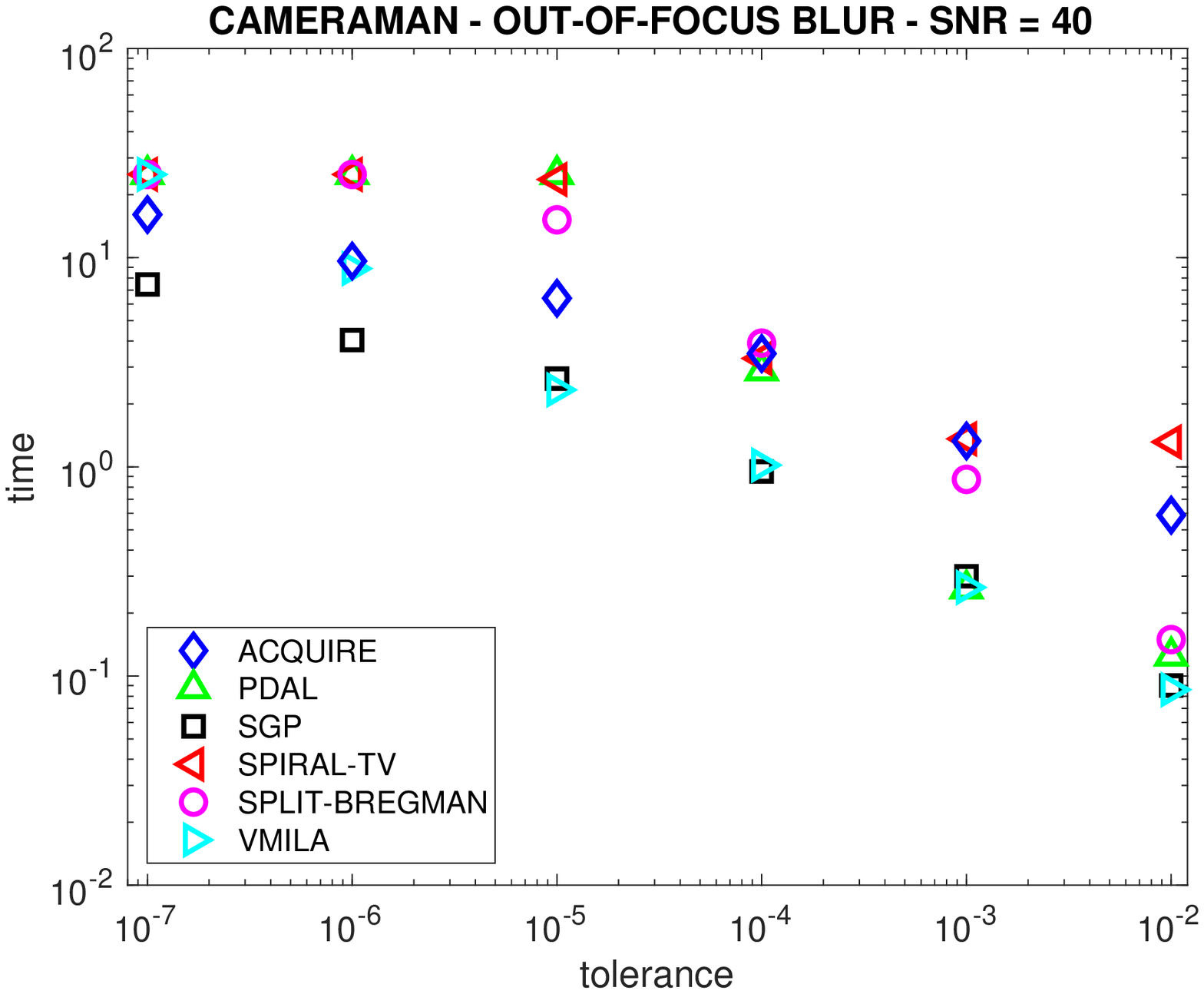}  \\[-1mm]
			\includegraphics[width=.44\textwidth]{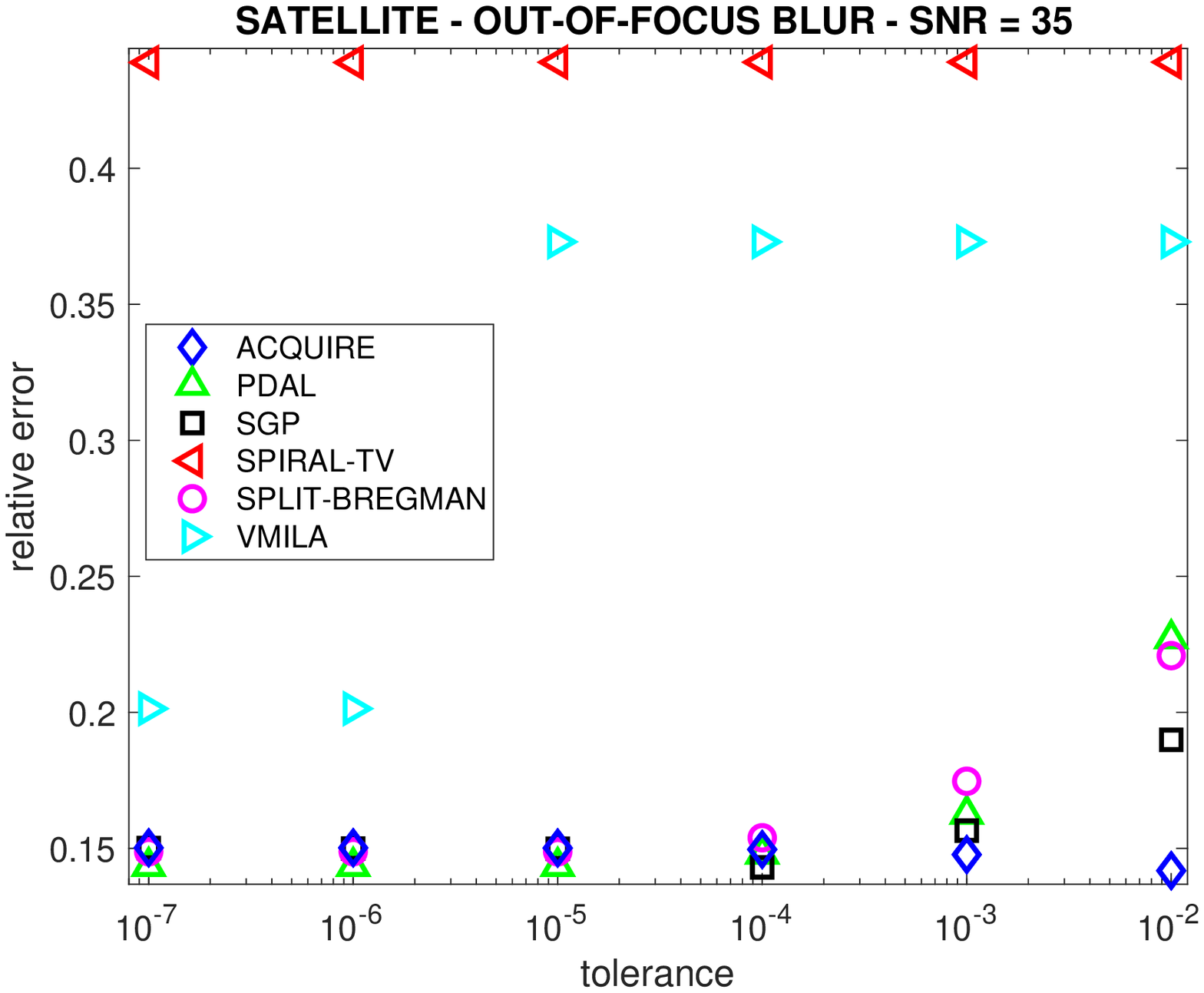}     &
			\includegraphics[width=.44\textwidth]{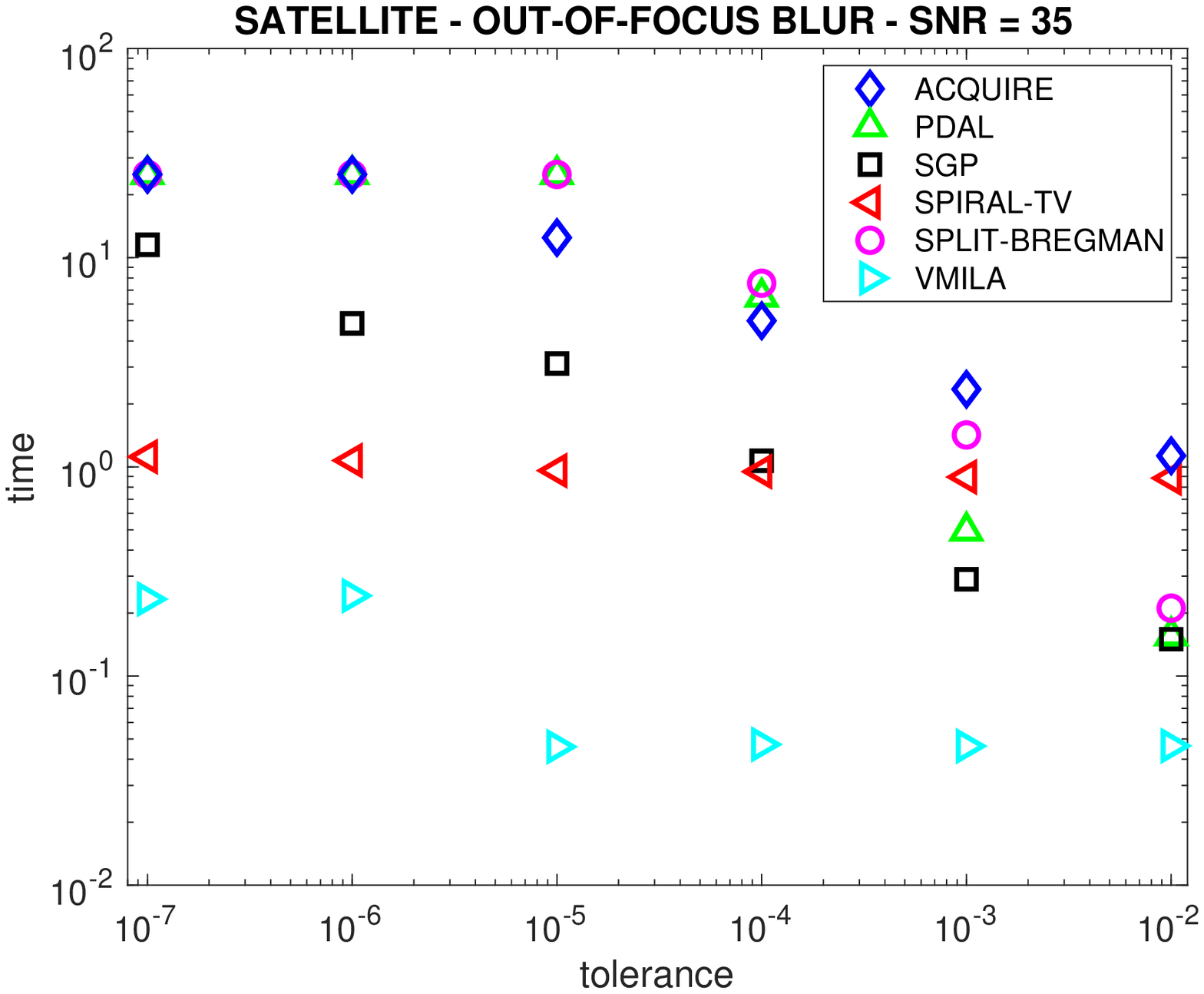}  \\[-1mm]
			\includegraphics[width=.44\textwidth]{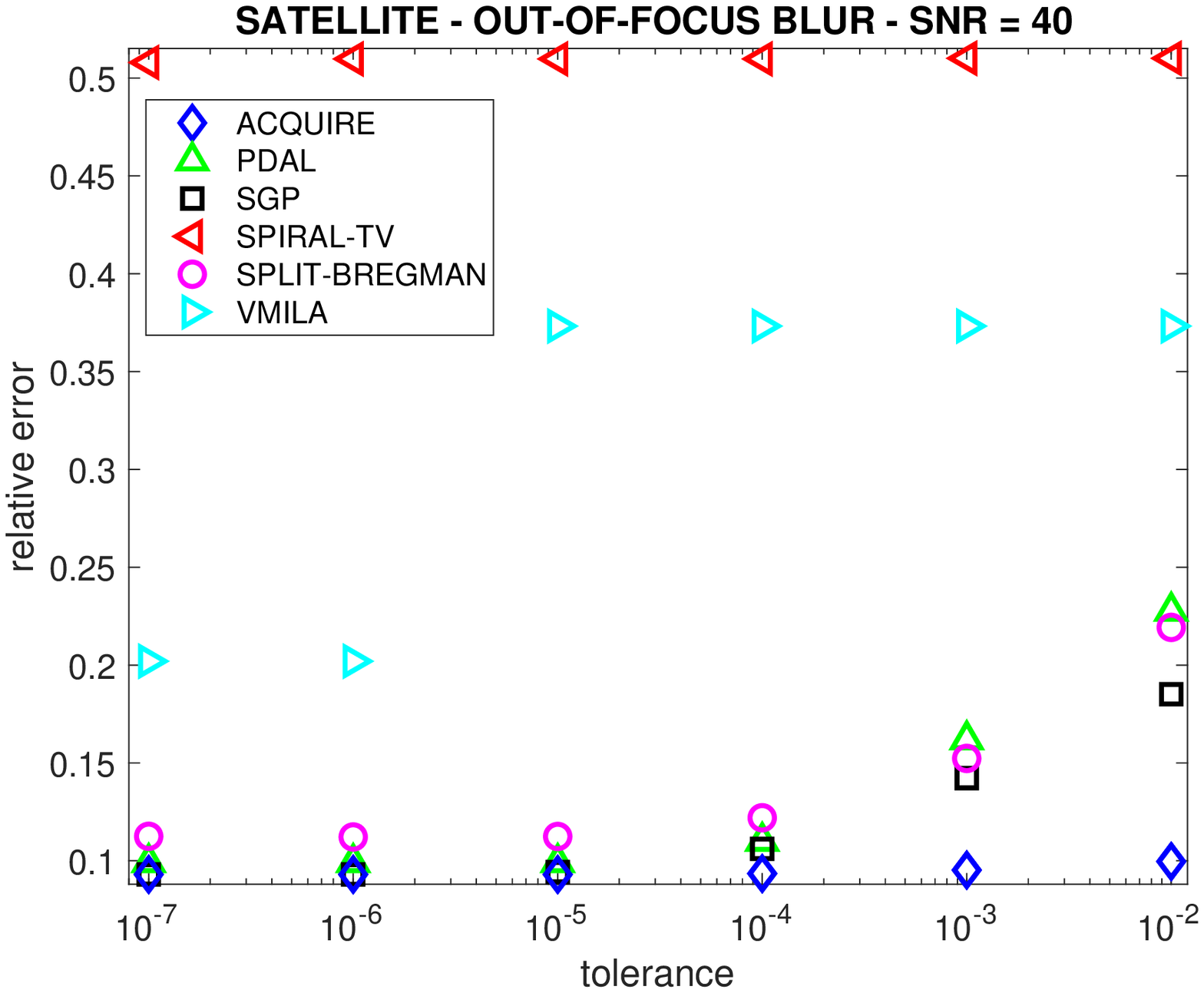}    &
			\includegraphics[width=.44\textwidth]{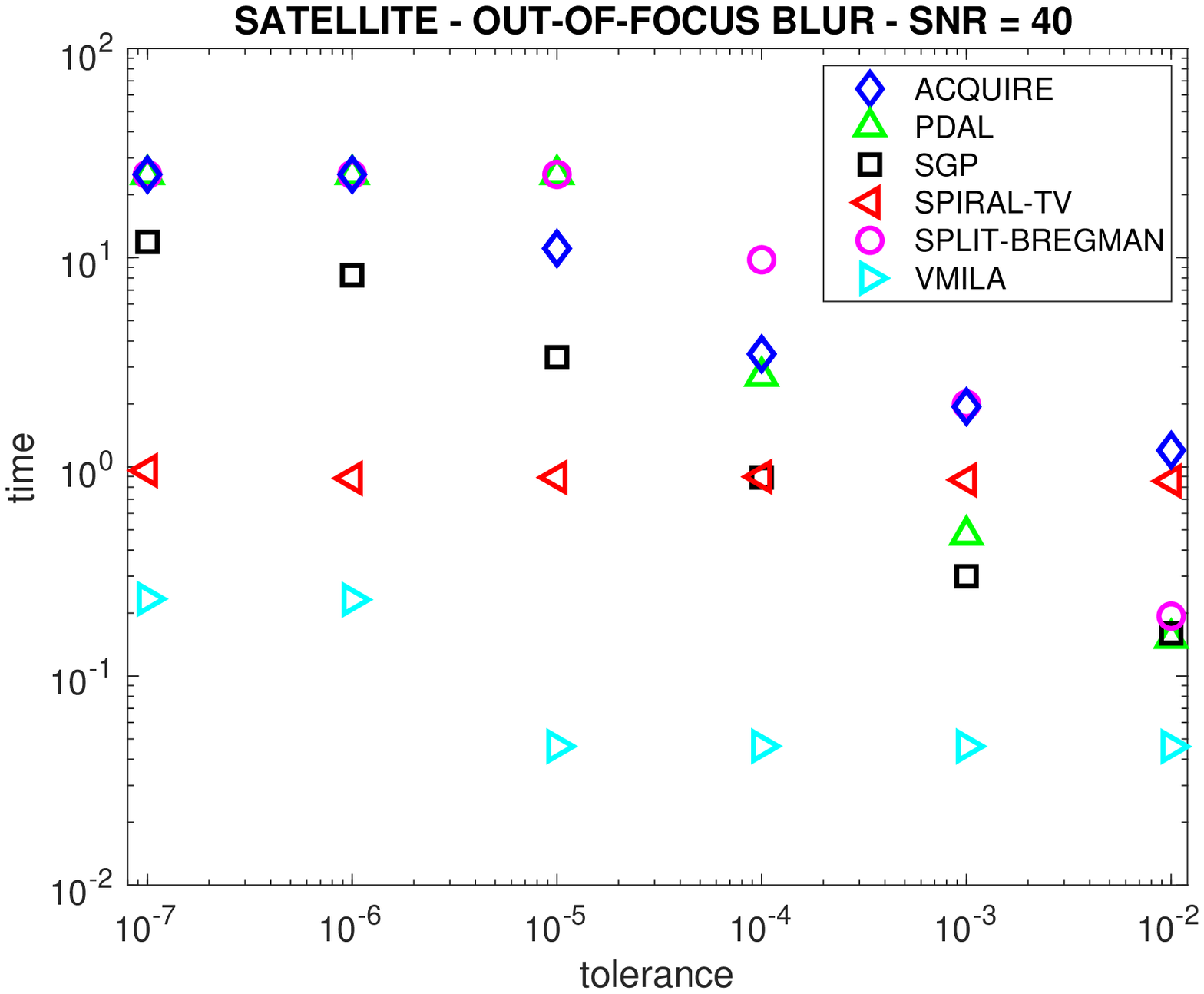}  \\[-1mm]
		\end{tabular}
		\vspace*{-2mm}
		\caption{Test set T2, out-of-focus blur, SNR $= 35,40$: relative error (left) and execution time (right) versus tolerance,
		for all the methods.\label{fig:compar_of}}	
	\end{center}
\end{figure}

\clearpage


\begin{table}[h!]
	\vspace*{1cm}
	{\small
		\begin{center}
			\begin{tabular}{|l|c|c|c|c|}
				\hline
				Problem       & $\sigma$ & SNR& \multicolumn{2}{|c|}{$\lambda$} \\ \cline{4-5}
				                    &                 &         & $TV$     &  $TV_\mu$ \\
				\hline
				cameraman &1.4            & 35     & 1.55e$-$2  & 1.55e$-$2 \\
				                    &                 & 40     & 5.00e$-$3  & 5.00e$-$3 \\
				\hline
				micro           & 2.0           & 35     & 4.50e$-$3   & 4.50e$-$3 \\
				                    &                 & 40     & 1.00e$-$3   & 1.00e$-$3 \\
				\hline
				phantom      & 2.0           & 35     & 6.00e$-$3   & 6.00e$-$3 \\
				                    &                 & 40     & 4.00e$-$3   & 4.00e$-$3 \\
				\hline
				satellite        & 2.0           & 35     & 7.00e$-$4   & 9.00e$-$4 \\
				                    &                 & 40     & 1.50e$-$4   & 9.00e$-$5 \\
				\hline
			\end{tabular}
			\vspace*{3mm}
			\caption{Details of test set T1.\label{tab:test_pbs}}
	     \end{center}
	}
\end{table}

\begin{table}[h!]
	\vspace*{1cm}
	{\small
		\begin{center}
			\begin{tabular}{|l|c|c|c|c|}
				\hline
				\multicolumn{5}{|c|}{motion blur} \\
				\hline
				Problem       & (\texttt{len}, $\varphi$) & SNR & \multicolumn{2}{|c|}{$\lambda$} \\ 
				\cline{4-5}
				                     &                &          & $TV$          &  $TV_\mu$ \\
				\hline
				cameraman   &(11,45)    & 35     & 1.50e$-$2   & 0.75e$-$2 \\
				                     &                 & 40     & 2.50e$-$3  & 1.75e$-$3 \\
				\hline
				satellite        & (11,45)     & 35     & 2.00e$-$3   & 1.50e$-$3 \\
				                    &                 & 40     & 3.50e$-$4   &  2.25e$-$4 \\
				\hline
				\multicolumn{5}{|c|}{out-of-focus blur} \\
				\hline
				Problem       & \texttt{rad} & SNR & \multicolumn{2}{|c|}{$\lambda$} \\ 
				\cline{4-5}
				                     &                &           & $TV$     &  $TV_\mu$ \\
				\hline
				cameraman &  4             & 35     & 1.50e$-$2  & 1.00e$-$2 \\
				                    &                 & 40     & 1.40e$-$3  & 1.20e$-$3 \\
				\hline
				satellite        &  4             & 35     & 1.75e$-$3   & 0.50e$-$3 \\
				                    &                 & 40     & 2.75e$-$4   & 1.90e$-$4 \\
				\hline                                
			\end{tabular}
			\vspace*{3mm}
			\caption{Details of test set T2.\label{tab:test_mot_of}}
	     \end{center}
	}
\end{table}

\clearpage

\begin{table}[h!]
	{\small
		\begin{center}
			\begin{tabular}{|l|c|c|c|c|c|}
				\hline
				\multicolumn{6}{|c|}{Test set T1, SNR $= 35$} \\ \hline
				Method            & Min rel err  &   MSSIM &  Iters   &  Time    &  Tol    \\  \hline
				\multicolumn{6}{|c|}{cameraman}                                                         \\  \hline
				ACQUIRE       & 9.88e$-$2 & 8.01e$-$1 &   11 & 1.12e$+$0 & 1.00e$-$3 \\
 				PDAL          & 9.95e$-$2 & 8.03e$-$1 & 2743 & 2.50e$+$1 & 1.00e$-$5 \\
 				SGP           & 9.86e$-$2 & 8.01e$-$1 &   67 & 1.10e$+$0 & 1.00e$-$4 \\
 				SPIRAL-TV     & 1.01e$-$1 & 7.99e$-$1 &   62 & 4.92e$+$0 & 1.00e$-$4 \\
 				SPLIT-BREGMAN & 1.02e$-$1 & 8.04e$-$1 &  116 & 1.55e$+$0 & 1.00e$-$4 \\
				VMILA         & 9.83e$-$2 & 8.00e$-$1 &   18 & 2.69e$-$1 & 1.00e$-$3 \\ \hline
				\multicolumn{6}{|c|}{micro}                                                                     \\ \hline
				ACQUIRE       & 5.33e$-$2 & 9.72e$-$1 &   50 & 1.47e$+$0 & 1.00e$-$4 \\
 				PDAL          & 5.40e$-$2 & 9.72e$-$1 & 9974 & 2.50e$+$1 & 1.00e$-$5 \\
 				SGP           & 5.37e$-$2 & 9.72e$-$1 &  361 & 2.83e$+$0 & 1.00e$-$6 \\
 				SPIRAL-TV     & 6.20e$-$2 & 9.72e$-$1 & 1081 & 2.50e$+$1 & 1.00e$-$7 \\
 				SPLIT-BREGMAN & 5.89e$-$2 & 9.77e$-$1 & 1107 & 3.29e$+$0 & 1.00e$-$5 \\
 				VMILA         & 5.43e$-$2 & 9.76e$-$1 &  501 & 4.61e$+$0 & 1.00e$-$6 \\ \hline
				\multicolumn{6}{|c|}{phantom}                                                                  \\ \hline
				ACQUIRE       & 1.41e$-$1 & 9.75e$-$1 &  220 & 2.50e$+$1 & 1.00e$-$6 \\
 				PDAL          & 1.40e$-$1 & 9.73e$-$1 & 2754 & 2.50e$+$1 & 1.00e$-$6 \\
 				SGP           & 1.41e$-$1 & 9.75e$-$1 &  769 & 1.68e$+$1 & 1.00e$-$7 \\
 				SPIRAL-TV     & 2.82e$-$1 & 9.22e$-$1 &  106 & 2.54e$+$0 & 1.00e$-$7 \\
 				SPLIT-BREGMAN & 1.67e$-$1 & 9.73e$-$1 & 2006 & 2.50e$+$1 & 1.00e$-$6 \\
 				VMILA         & 1.39e$-$1 & 9.80e$-$1 &  483 & 8.20e$+$0 & 1.00e$-$7 \\ \hline
				\multicolumn{6}{|c|}{satellite}                                                                    \\ \hline
				 ACQUIRE       & 1.63e$-$1 & 9.62e$-$1 &   20 & 1.89e$+$0 & 1.00e$-$3 \\
 				PDAL          & 1.68e$-$1 & 9.61e$-$1 &  275 & 2.52e$+$0 & 1.00e$-$4 \\
 				SGP           & 1.65e$-$1 & 9.61e$-$1 &   84 & 1.36e$+$0 & 1.00e$-$4 \\
 				SPIRAL-TV     & 2.46e$-$1 & 9.11e$-$1 &   51 & 8.49e$-$1 & 1.00e$-$2 \\
 				SPLIT-BREGMAN & 1.86e$-$1 & 9.45e$-$1 & 1985 & 2.50e$+$1 & 1.00e$-$5 \\
 				VMILA         & 2.04e$-$1 & 9.40e$-$1 &    9 & 7.73e$-$2 & 1.00e$-$2 \\  \hline
			\end{tabular}
			\vspace*{3mm}
			\caption{Test set T1, SNR $= 35$: minimum relative error achieved by each method and corresponding MSSIM
				value, number of iterations, execution time and tolerance.\label{tab:compar35_minerr}}
	   \end{center}
	}
\end{table}

\clearpage

\begin{table}[h!]
	{\small
		\begin{center}
			\begin{tabular}{|l|c|c|c|c|c|}
				\hline
				\multicolumn{6}{|c|}{Test set T1, SNR $= 40$} \\ \hline
				Method            & Min rel err  &   MSSIM &  Iters   &  Time    &  Tol    \\  \hline
				\multicolumn{6}{|c|}{cameraman}                                                         \\  \hline
				ACQUIRE       & 8.73e$-$2 & 8.42e$-$1 &    8 & 8.23e$-$1 & 1.00e$-$3 \\
 				PDAL          & 8.88e$-$2 & 8.22e$-$1 & 2733 & 2.50e$+$1 & 1.00e$-$5 \\
 				SGP           & 8.70e$-$2 & 8.42e$-$1 &   53 & 9.10e$-$1 & 1.00e$-$4 \\
 				SPIRAL-TV     & 8.97e$-$2 & 8.36e$-$1 &  280 & 2.50e$+$1 & 1.00e$-$6 \\
 				SPLIT-BREGMAN & 9.38e$-$2 & 8.41e$-$1 & 1962 & 2.50e$+$1 & 1.00e$-$7 \\
 				VMILA         & 8.72e$-$2 & 8.42e$-$1 &   58 & 9.39e$-$1 & 1.00e$-$4 \\  \hline
				\multicolumn{6}{|c|}{micro}                                                                    \\  \hline
				ACQUIRE       & 4.31e$-$2 & 9.82e$-$1 &  218 & 6.46e$+$0 & 1.00e$-$5 \\
 				PDAL          & 4.62e$-$2 & 9.64e$-$1 & 9543 & 2.50e$+$1 & 1.00e$-$7 \\
 				SGP           & 4.31e$-$2 & 9.81e$-$1 &  700 & 3.55e$+$0 & 1.00e$-$7 \\
 				SPIRAL-TV     & 5.02e$-$2 & 9.83e$-$1 & 1480 & 2.50e$+$1 & 1.00e$-$6 \\
 				SPLIT-BREGMAN & 5.26e$-$2 & 9.85e$-$1 & 7766 & 2.50e$+$1 & 1.00e$-$7 \\
 				VMILA         & 4.32e$-$2 & 9.85e$-$1 & 1223 & 1.04e$+$1 & 1.00e$-$7 \\ \hline
				\multicolumn{6}{|c|}{phantom}                                                                 \\  \hline
				ACQUIRE       & 1.29e$-$1 & 9.85e$-$1 &  217 & 2.50e$+$1 & 1.00e$-$6 \\
 				PDAL          & 1.28e$-$1 & 9.79e$-$1 & 2650 & 2.50e$+$1 & 1.00e$-$5 \\
 				SGP           & 1.29e$-$1 & 9.85e$-$1 &  369 & 7.54e$+$0 & 1.00e$-$6 \\
 				SPIRAL-TV     & 2.97e$-$1 & 9.10e$-$1 &   51 & 1.21e$+$0 & 1.00e$-$2 \\
 				SPLIT-BREGMAN & 1.50e$-$1 & 9.83e$-$1 & 1936 & 2.50e$+$1 & 1.00e$-$7 \\
 				VMILA         & 2.28e$-$1 & 9.51e$-$1 &   16 & 2.18e$-$1 & 1.00e$-$3 \\ \hline
				\multicolumn{6}{|c|}{satellite}                                                                    \\  \hline
				ACQUIRE       & 1.48e$-$1 & 9.70e$-$1 &   28 & 2.74e$+$0 & 1.00e$-$3 \\
 				PDAL          & 1.50e$-$1 & 9.69e$-$1 & 2417 & 2.26e$+$1 & 1.00e$-$5 \\
 				SGP           & 1.48e$-$1 & 9.70e$-$1 &  216 & 3.73e$+$0 & 1.00e$-$5 \\
 				SPIRAL-TV     & 2.47e$-$1 & 9.11e$-$1 &   85 & 1.45e$+$0 & 1.00e$-$7 \\
 				SPLIT-BREGMAN & 1.72e$-$1 & 9.53e$-$1 & 1935 & 2.50e$+$1 & 1.00e$-$5 \\
 				VMILA         & 2.08e$-$1 & 9.37e$-$1 &   10 & 1.35e$-$1 & 1.00e$-$3 \\ \hline
			\end{tabular}
			\vspace*{3mm}
			\caption{Test set T1, SNR $= 40$: minimum relative error achieved by each method and corresponding MSSIM
				value, number of iterations, execution time and tolerance.\label{tab:compar40_minerr}}
	    \end{center}
	}
\end{table}

\begin{table}[h!]
	{\small
		\begin{center}
			\begin{tabular}{|l|c|c|c|c|c|}
				\hline
				\multicolumn{6}{|c|}{Test set T2, motion blur} \\ \hline
				Method            & Min rel err  &   MSSIM &  Iters   &  Time    &  Tol    \\  \hline
				\multicolumn{6}{|c|}{cameraman, SNR $= 35$}                                  \\  \hline
				ACQUIRE       & 1.12e$-$1 & 6.85e$-$1 &   12 & 1.21e$+$0 & 1.00e$-$3 \\ 
 				PDAL          & 1.13e$-$1 & 7.51e$-$1 & 2656 & 2.50e$+$1 & 1.00e$-$6 \\ 
 				SGP           & 1.13e$-$1 & 6.77e$-$1 &  110 & 2.22e$+$0 & 1.00e$-$5 \\ 
 				SPIRAL-TV     & 1.14e$-$1 & 7.54e$-$1 &  242 & 2.51e$+$1 & 1.00e$-$5 \\ 
 				SPLIT-BREGMAN & 1.19e$-$1 & 7.50e$-$1 &  164 & 2.15e$+$0 & 1.00e$-$4 \\ 
 				VMILA         & 1.12e$-$1 & 7.51e$-$1 &   73 & 1.27e$+$0 & 1.00e$-$4 \\ 
 				\hline
				\multicolumn{6}{|c|}{cameraman, SNR $= 40$}                                 \\ \hline
				ACQUIRE       & 8.13e$-$2 & 8.09e$-$1 &   31 & 3.21e$+$0 & 1.00e$-$4 \\ 
 				PDAL          & 8.42e$-$2 & 8.04e$-$1 & 2698 & 2.50e$+$1 & 1.00e$-$5 \\ 
 				SGP           & 8.14e$-$2 & 8.09e$-$1 &  219 & 3.83e$+$0 & 1.00e$-$6 \\ 
 				SPIRAL-TV     & 8.56e$-$2 & 8.25e$-$1 &  377 & 2.52e$+$1 & 1.00e$-$7 \\ 
 				SPLIT-BREGMAN & 9.85e$-$2 & 8.06e$-$1 &  880 & 1.16e$+$1 & 1.00e$-$5 \\ 
 				VMILA         & 8.28e$-$2 & 8.27e$-$1 &  139 & 2.14e$+$0 & 1.00e$-$5 \\ \hline
				\multicolumn{6}{|c|}{satellite, SNR $= 35$}                                  \\  \hline
				ACQUIRE       & 1.12e$-$1 & 9.81e$-$1 &    9 & 8.76e$-$1 & 1.00e$-$2 \\ 
 				PDAL          & 1.14e$-$1 & 9.81e$-$1 & 2680 & 2.50e$+$1 & 1.00e$-$5 \\ 
 				SGP           & 1.12e$-$1 & 9.81e$-$1 &   64 & 1.06e$+$0 & 1.00e$-$4 \\ 
 				SPIRAL-TV     & 7.92e$-$1 & 8.78e$-$1 &   51 & 9.09e$-$1 & 1.00e$-$2 \\ 
 				SPLIT-BREGMAN & 1.20e$-$1 & 9.77e$-$1 & 2099 & 2.50e$+$1 & 1.00e$-$7 \\ 
 				VMILA         & 1.69e$-$1 & 9.60e$-$1 &   23 & 2.52e$-$1 & 1.00e$-$6 \\ \hline
				\multicolumn{6}{|c|}{satellite, SNR $= 40$}                                  \\  \hline
				ACQUIRE       & 7.01e$-$2 & 9.93e$-$1 &  114 & 1.09e$+$1 & 1.00e$-$5 \\ 
 				PDAL          & 7.52e$-$2 & 9.91e$-$1 & 2680 & 2.50e$+$1 & 1.00e$-$7 \\ 
 				SGP           & 7.00e$-$2 & 9.93e$-$1 &  356 & 5.72e$+$0 & 1.00e$-$6 \\ 
 				SPIRAL-TV     & 6.62e$-$1 & 8.89e$-$1 &   61 & 1.03e$+$0 & 1.00e$-$7 \\ 
 				SPLIT-BREGMAN & 7.94e$-$2 & 9.90e$-$1 & 2114 & 2.50e$+$1 & 1.00e$-$6 \\ 
 				VMILA         & 1.74e$-$1 & 9.59e$-$1 &   22 & 2.46e$-$1 & 1.00e$-$6 \\ \hline
		         \end{tabular}
			\vspace*{3mm}
			\caption{Test set T2, motion blur: minimum relative error achieved by each method and corresponding MSSIM
				value, number of iterations, execution time and tolerance.\label{tab:compar_mot_minerr}}
	   \end{center}
	}
\end{table}

\begin{table}[h!]
	{\small
		\begin{center}
			\begin{tabular}{|l|c|c|c|c|c|}
				\hline
				\multicolumn{6}{|c|}{Test set T2, out-of-focus blur} \\ \hline
				Method            & Min rel err  &   MSSIM &  Iters   &  Time    &  Tol    \\  \hline
				\multicolumn{6}{|c|}{cameraman, SNR $= 35$}                                  \\  \hline
				ACQUIRE       & 1.18e$-$1 & 7.41e$-$1 &   12 & 1.11e$+$0 & 1.00e$-$3 \\ 
 				PDAL          & 1.20e$-$1 & 7.50e$-$1 & 2624 & 2.50e$+$1 & 1.00e$-$5 \\ 
 				SGP           & 1.18e$-$1 & 7.39e$-$1 &  122 & 1.96e$+$0 & 1.00e$-$5 \\ 
 				SPIRAL-TV     & 1.22e$-$1 & 7.44e$-$1 &  237 & 2.51e$+$1 & 1.00e$-$5 \\ 
 				SPLIT-BREGMAN & 1.25e$-$1 & 7.49e$-$1 &  144 & 1.72e$+$0 & 1.00e$-$4 \\ 
 				VMILA         & 1.19e$-$1 & 7.50e$-$1 &   89 & 1.36e$+$0 & 1.00e$-$4 \\ \hline	
				\multicolumn{6}{|c|}{cameraman, SNR $= 40$}                                   \\  \hline
				ACQUIRE       & 9.20e$-$2 & 7.86e$-$1 &   14 & 1.33e$+$0 & 1.00e$-$3 \\ 
 				PDAL          & 9.66e$-$2 & 7.12e$-$1 & 2658 & 2.50e$+$1 & 1.00e$-$5 \\ 
 				SGP           & 9.22e$-$2 & 7.81e$-$1 &  160 & 2.64e$+$0 & 1.00e$-$5 \\ 
 				SPIRAL-TV     & 9.54e$-$2 & 7.92e$-$1 &  386 & 2.50e$+$1 & 1.00e$-$6 \\ 
 				SPLIT-BREGMAN & 1.10e$-$1 & 7.73e$-$1 & 1284 & 1.52e$+$1 & 1.00e$-$5 \\ 
 				VMILA         & 9.20e$-$2 & 7.93e$-$1 &  148 & 2.33e$+$0 & 1.00e$-$5 \\ \hline
				\multicolumn{6}{|c|}{satellite, SNR $= 35$}                                  \\  \hline
				ACQUIRE       & 1.42e$-$1 & 9.72e$-$1 &   11 & 1.13e$+$0 & 1.00e$-$2 \\ 
 				PDAL          & 1.43e$-$1 & 9.70e$-$1 & 2700 & 2.50e$+$1 & 1.00e$-$5 \\ 
 				SGP           & 1.43e$-$1 & 9.72e$-$1 &   70 & 1.07e$+$0 & 1.00e$-$4 \\ 
 				SPIRAL-TV     & 4.39e$-$1 & 8.95e$-$1 &   63 & 1.07e$+$0 & 1.00e$-$6 \\ 
 				SPLIT-BREGMAN & 1.49e$-$1 & 9.64e$-$1 & 2101 & 2.50e$+$1 & 1.00e$-$5 \\ 
 				VMILA         & 2.01e$-$1 & 9.42e$-$1 &   20 & 2.42e$-$1 & 1.00e$-$6 \\ \hline
				\multicolumn{6}{|c|}{satellite, SNR $= 40$}                                  \\  \hline
				ACQUIRE       & 9.30e$-$2 & 9.87e$-$1 &  116 & 1.11e$+$1 & 1.00e$-$5 \\ 
 				PDAL          & 9.91e$-$2 & 9.85e$-$1 & 2707 & 2.50e$+$1 & 1.00e$-$5 \\ 
 				SGP           & 9.30e$-$2 & 9.87e$-$1 &  541 & 8.25e$+$0 & 1.00e$-$6 \\ 
 				SPIRAL-TV     & 5.08e$-$1 & 8.87e$-$1 &   57 & 9.61e$-$1 & 1.00e$-$7 \\ 
 				SPLIT-BREGMAN & 1.12e$-$1 & 9.79e$-$1 & 2096 & 2.50e$+$1 & 1.00e$-$6 \\ 
 				VMILA         & 2.02e$-$1 & 9.42e$-$1 &   20 & 2.32e$-$1 & 1.00e$-$6 \\ \hline		
                         \end{tabular}
			\vspace*{3mm}
			\caption{Test set T2, out-of-focus blur: minimum relative error achieved by each method and corresponding MSSIM
				value, number of iterations, execution time and tolerance.\label{tab:compar_of_minerr}}
	    \end{center}
	}
\end{table}

\end{document}